\documentclass{article}
\usepackage[utf8]{inputenc}
\usepackage{amssymb,amsmath, amsthm, mathabx,mathtools}
\usepackage{amsmath,amsfonts,amssymb,amsthm,epsfig,epstopdf,titling,url,array}
\usepackage{tikz}
\usetikzlibrary{shapes,snakes}
\usepackage{color, enumerate}
\usepackage{comment, authblk}
\usepackage{hyperref}
\usepackage{pgfplots}
\usepackage{bm}
\usepackage{stmaryrd}
\usepackage{graphicx}
\usepackage{caption}
\usepackage{subcaption}
\usepackage{a4wide}
\usepackage{titlesec}
\usepackage[a4paper, total={6in, 8in}]{geometry}
\usepackage{epstopdf}

\usepackage{sectsty}
\usepackage{lipsum}
\usepackage{algorithm}
\usepackage[noend]{algpseudocode}
\usepackage{algpseudocode}
\usepackage{pifont}

\titleformat*{\subsection}{\large\bfseries}
\numberwithin{equation}{section}

\pgfplotsset{compat=newest}
\pgfplotsset{plot coordinates/math parser=false}
\newlength\figureheight
\newlength\figurewidth

\usepackage{enumerate}
\usepackage{mathrsfs}
\usepackage{wrapfig}

\usepackage{color}

\numberwithin{equation}{section}

\newcommand{\beq}{\begin{equation}}
\newcommand{\bEq}{\end{equation}}

\newcommand{\bx}{{\bf{x}}}
\newcommand{\by}{{\bf{y}}}

\newcommand{\al}{\alpha}

\newcommand{\be}{\begin{equation}}
\newcommand{\ee}{\end{equation}}

\newcommand{\e}{{\varepsilon}}

\newcommand{\fa}{{\mathfrak a}}
\newcommand{\fb}{{\mathfrak b}}

\usepackage{amsmath} 
\usepackage{amssymb}
\usepackage{amsthm}

\renewcommand{\cal}{\mathcal}

\newcommand{\wh}{\widehat}
\newcommand{\wt}{\widetilde}

\newcommand{\ii}{\mathrm{i}} 
\newcommand{\dd}{\mathrm{d}}

\renewcommand{\epsilon}{\varepsilon}
\renewcommand{\leq}{\leqslant}
\renewcommand{\geq}{\geqslant}

\renewcommand{\le}{\leq}
\renewcommand{\ge}{\geq}

\renewcommand{\P}{\mathbb{P}}
\newcommand{\E}{\mathbb{E}}
\newcommand{\R}{\mathbb{R}}
\newcommand{\C}{\mathbb{C}}
\newcommand{\N}{\mathbb{N}}

\newcommand{\wsG}{{\widehat{\cal G}}}
\newcommand{\wsH}{{\widehat{\cal H}}}

\DeclareMathOperator{\diag}{diag}
\DeclareMathOperator{\tr}{Tr}

\DeclareMathOperator{\Cov}{Cov}

\DeclareMathOperator{\im}{Im}

\DeclareMathOperator{\OO}{O}
\DeclareMathOperator{\oo}{o}

\DeclareMathOperator{\bC}{\mathbf{C}}

\DeclareMathOperator{\bv}{\mathbf{v}}
\DeclareMathOperator{\bu}{\mathbf{u}}
\DeclareMathOperator{\bw}{\mathbf{w}}

\DeclareMathOperator{\bbE}{\mathbb{E}}
\DeclareMathOperator{\bbN}{\mathbb{N}}
\DeclareMathOperator{\bbP}{\mathbb{P}}

\DeclareMathOperator{\sI}{\mathcal{I}}

\theoremstyle{plain} 
\newtheorem{theorem}{Theorem}[section]
\newtheorem*{theorem*}{Theorem}
\newtheorem{lemma}[theorem]{Lemma}
\newtheorem{assumption}[theorem]{Assumption}
\newtheorem*{lemma*}{Lemma}
\newtheorem{corollary}[theorem]{Corollary}
\newtheorem*{corollary*}{Corollary}
\newtheorem{proposition}[theorem]{Proposition}
\newtheorem*{proposition*}{Proposition}
\newtheorem{claim}[theorem]{Claim}

\newtheorem{definition}[theorem]{Definition}
\newtheorem*{definition*}{Definition}
\theoremstyle{remark}

\newtheorem*{example*}{Example}
\newtheorem{remark}[theorem]{Remark}

\newtheorem*{remark*}{Remark}
\newtheorem*{remarks*}{Remarks}

\renewcommand{\Im}{{\rm{Im}}}

\newcommand{\nc}{\normalcolor}

\makeatletter
\def\env@dmatrix{\hskip -\arraycolsep
  \let\@ifnextchar\new@ifnextchar
  \extrarowheight=2ex
  \array{*\c@MaxMatrixCols{>{\displaystyle}c}}}

\makeatother

\allowdisplaybreaks

\title{Sample canonical correlation coefficients of high-dimensional random vectors: local law and Tracy-Widom limit}
\author[1]{Fan Yang  \thanks{E-mail: fyang75@wharton.upenn.edu}}
\affil[1]{Department of Statistics, University of Pennsylvania}

\begin{document}
\maketitle

\begin{abstract}
Consider two random vectors $\bC_1^{1/2}\mathbf x \in \R^p$ and $\bC_2^{1/2}\mathbf y\in \mathbb R^q$, where the entries of $\mathbf x$ and $\mathbf y$ are i.i.d. random variables with mean zero and variance one, and $\bC_1$ and $\bC_2$ are respectively  $p \times p$ and $q\times q$ deterministic population covariance matrices. With $n$ independent samples of $(\bC_1^{1/2}\mathbf x,\bC_2^{1/2}\mathbf y)$, we study the sample correlation between these two vectors using canonical correlation analysis. 
Under the high-dimensional setting with ${p}/{n}\to c_1 \in (0, 1)$ and ${q}/{n}\to c_2 \in (0, 1-c_1)$ as $n\to \infty$,  we prove that the largest sample canonical correlation coefficient converges to the Tracy-Widom distribution as long as we have $\lim_{s \rightarrow \infty}s^4 \mathbb{P}(\vert  x_{ij} \vert \geq s)=0$ and $\lim_{s \rightarrow \infty}s^4 \mathbb{P}(\vert  y_{ij} \vert \geq s)=0$, which we believe to be a sharp moment condition. This extends the result in \cite{CCA2}, which established the Tracy-Widom limit under the assumption that all moments exist for the entries of $\bx$ and $\by$. Our proof is based on a new linearization method, which reduces the problem to the study of a $(p+q+2n)\times (p+q+2n)$ random matrix $H$. In particular, we shall prove an optimal local law on its inverse $G:=H^{-1}$, called resolvent. This local law is the main tool for both the proof of the Tracy-Widom law in this paper, and the study in \cite{PartII,PartI} on the canonical correlation coefficients of high-dimensional random vectors with finite rank correlations. 
\end{abstract}

\section{Introduction}\label{sec_intro}

In multivariate statistics, the canonical correlation analysis (CCA) has been one of the most general and classical methods to study the correlations between two random vectors $\mathbf x\in \R^p$ and $\mathbf y\in \R^q$ since the seminal work by Hotelling \cite{Hotelling}. CCA seeks two sequence of orthonormal vectors, such that the projections of $\mathbf x$ and $\by$ onto these vectors have maximized correlations. The corresponding sequence of correlations are called the {\it canonical correlation coefficients} (CCC). More precisely, we first find the unit vectors $\mathbf a_1\in \R^p$ and $\mathbf b_1\in \R^p$ that maximize the correlation, 
$$\rho(\mathbf a_1,\mathbf b_1)=\sup_{\|\mathbf a\|=1,\|\mathbf b\|=1}\rho(\mathbf a,\mathbf b),\quad  \rho(\mathbf a,\mathbf b):=\text{Corr}(\mathbf a^T \bx, \mathbf b^T \by ).$$
Then $\rho_1:=\rho(\mathbf a_1,\mathbf b_1)$ is the first CCC, and $(\mathbf a_1^T \bx, \mathbf b_1^T \by )$ is called the first pair of canonical variables. Suppose we have obtained the first $k$ CCC, $\rho_i$, $1\le i \le k$, and the corresponding pairs of canonical variables $(\mathbf a_i^T \bx, \mathbf b_i^T \by )$, $1\le i \le k$. We then define inductively the $(k+1)$-th CCC by seeking the vectors $(\mathbf a_{k+1},\mathbf b_{k+1})$ that maximize $\rho(\mathbf a_{k+1},\mathbf b_{k+1})$ subject to the constraint that $(\mathbf a_{k+1}^T \bx, \mathbf b_{k+1}^T \by )$ is uncorrelated with the first $k$ pairs of canonical variables. Then $\rho_{k+1}:=\rho(\mathbf a_{k+1},\mathbf b_{k+1})$ is the $(k+1)$-th  CCC.

Define the population covariance and cross-covariance matrices
$$\Sigma_{xx}:= \Cov(\mathbf x,\mathbf x),\quad \Sigma_{yy}:= \Cov(\mathbf y,\mathbf y), \quad \Sigma_{xy}=\Sigma_{yx}^T:= \Cov(\mathbf x,\mathbf y). $$
It is well-known that $\rho_i^2$  is the $i$-th largest eigenvalue of the population canonical correlation matrix ${\bm\Sigma}:=\Sigma_{xx}^{-1}\Sigma_{xy}\Sigma_{yy}^{-1}\Sigma_{yx}$. Given $n$ independent samples of $( \mathbf x, \mathbf y)$, we study the CCC through their sample counterparts, which  are defined as the eigenvalues of the {\it sample canonical correlation} (SCC) matrix  
$$\cal C_{XY}:=S_{xx}^{-1}S_{xy}S_{yy}^{-1}S_{yx},$$
where
$$S_{xx}:=\frac1n\sum_{i=1}^n \bx_i \bx_i^T, \quad S_{yy}:=\frac1n\sum_{i=1}^n \by_i \by_i^T, \quad S_{xy}=S_{yx}^T:=\frac1n\sum_{i=1}^n \bx_i \by_i^T.$$
We denote the eigenvalues of $\cal C_{XY}$, i.e. the sample CCC, as $\lambda_1\ge \lambda_2 \ge \cdots \ge \lambda_{p\wedge q}$.

In this paper, we consider the case where $\bx$ and $\by$ are independent. We are interested in the behaviors of the eigenvalues of the SCC matrix $\cal C_{XY}$, including the convergence of (almost) all the eigenvalues and the limiting distribution of the largest few eigenvalues.  If the entries of $X$ and $Y$ are i.i.d. Gaussian distributed, then the eigenvalues of $\cal C_{XY}$ reduce to those of the double Wishart matrices \cite{CCA_TW}. Moreover, the joint distribution of the eigenvalues of double Wishart matrices has been studied in the context of the so-called Jacobi ensemble and F-type matrices, and it has been shown that the largest eigenvalue converges to the type-1 Tracy-Widom distribution under a proper scaling \cite{CCA_TW2,CCA_TW}. For general distribution of the entries of $X$ and $Y$, the Tracy-Widom law of the largest eigenvalue of $\cal C_{XY}$ was established in \cite{CCA2} under the assumption that all the moments of the entries are finite. There have been many other works on high-dimensional CCA, and, without attempting to be comprehensive, we mention some of them that are most related to the topic of this paper. In \cite{Fujikoshi2017}, the author derived the asymptotic distributions of the canonical correlation coefficients when one of $p$ and $q$ is fixed as $n\to \infty$. When $p$ and $q$ are proportional to $n$, the asymptotic distributions of the spiked eigenvalues for CCA with finite rank correlations have been established in \cite{CCA}. The CLT for linear spectral statistics of CCA was proved in \cite{CCA_CLT}. Under certain sparsity assumptions, the theory of high-dimensional sparse CCA
and it applications have been discussed in \cite{gao2015,gao2017}.  In a recent paper \cite{johnstone2020}, the authors studied the asymptotic behaviors of likelihood ratios of CCA under the null hypothesis of no spikes and the alternative hypothesis with a single spike.

One purpose of this paper is to extend the Tracy-Widom law in \cite{CCA2} to the case with weaker moment assumptions. In fact, we prove that the largest eigenvalue of $\mathcal C_{XY}$ converges to the Tracy-Widom distribution as long as the following tail condition holds (see Theorem \ref{main_thm1}): 
\be\label{intro tail}\lim_{s \rightarrow \infty}s^4\left[ \mathbb{P}( \vert  x_i \vert \geq s)+\mathbb{P}(\vert  y_i \vert \geq s)\right]=0.\ee 
We believe it to be the sharp moment condition, because it has been shown to be necessary and sufficient for the Tracy-Widom limit of the largest eigenvalue of sample covariance matrices \cite{DY}. Besides the Tracy-Widom law for the largest eigenvalue, we will also prove a rigidity estimate for (almost) all the eigenvalues of $\cal C_{XY}$, including the ones in the bulk of the spectrum. This rigidity estimate was not presented \cite{CCA2}, and we expect that it will be of independent interest. Different from the methods used in \cite{CCA2}, we will develop a new linearization method, which reduces the problem to the study of a $(p+q+2n)\times (p+q+2n)$ random matrix $H$ that is linear in $X$ and $Y$; see \eqref{linearize_block} below. Moreover, we will prove an optimal local law on its inverse $G:=H^{-1}$, i.e. the so-called resolvent, which is another main result of this paper. The linearization idea and the local law allow us to relax the moment assumptions in \cite{CCA2} and prove the Tracy-Widom law under the tail condition in \eqref{intro tail}. 
 
 Besides the Tracy-Widom distribution of the largest eigenvalues, the local law on $G$ proved in this paper will also serve as the base of further studies of high-dimensional CCA with finite rank correlations. 
More precisely, we consider the sample CCC of  two high-dimensional random vectors $\wh{ \mathbf x} \in \R^p$ and $ \wh \by\in \mathbb R^q$ with finite rank correlations as following:
$$\wh{ \mathbf x} = \mathbf C_1^{1/2} \mathbf x + A \mathbf z, \quad \wh \by= \mathbf C_2^{1/2} \mathbf y + B \mathbf z,$$
where $\mathbf C_1$ and $\mathbf C_2$ are $p \times p$ and $q\times q$ deterministic non-negative definite symmetric matrices, which give the population covariances,  and $A$ and $B$ are $p\times r$ and $q\times r$ deterministic matrices, which are the factor loading matrices. Moreover, suppose that the entries of $\mathbf x\in \R^p$, $\mathbf y\in \R^q$ and $\mathbf z\in \R^r$ are real independent random variables with zero mean  and unit variance. For $n$ independent samples $( \wh \bx_i, \wh \by_i)$, $1\le i \le n$, we can arrange them into the following data matrix with a conventional scaling $n^{-1/2}$:
$$ {\mathcal X}: = \mathbf C^{1/2}_1 X + AZ, \quad  {\mathcal Y}: = \mathbf C^{1/2}_2 Y +  B Z.$$
Now $X$, $Y$ and $Z$ are respectively $p\times n$, $q\times n$ and $r\times n$ matrices with real independent entries with mean zero and variance $n^{-1}$. We consider the high-dimensional setting with low-rank perturbations, that is, ${p}/{n}\to c_1 \in (0, 1)$ and ${q}/{n}\to c_2 \in (0, 1-c_1)$ as $n\to \infty$, and $r=\OO(1)$ is fixed.
For this model, the canonical correlation matrix ${\bm\Sigma}$ is of rank $\le r$, and has at most $r$ nonzero eigenvalues $t_i:=\rho_i^2$, $1\le i \le r$. Bao et al. \cite{CCA} consider this setting for Gaussian vectors, that is, $X$, $Y$ and $Z$ are all random matrices with i.i.d. Gaussian entries. They show that $t_i$ will give rise to an outlier of the spectrum if it is above some threshold $t_c$. The outlier lies around a fixed location determined by $t_i$, and moreover, it is asymptotic Gaussian under the $\sqrt{n}$ scaling. 
The proof in \cite{CCA} depends on the fact that multivariate Gaussian distribution is rotational invariant, which is not true for more general distributions. On the other hand, the linearization method developed in this paper allows us to circumvent this issue. Based on the main results of this paper, we will extend the results in \cite{CCA} to more general distributions, assuming only certain moments conditions on the entries of $X$, $Y$ and $Z$. Due to restraint of length of this paper, we will put those results in other papers \cite{PartII,PartI}. In \cite{PartI} we will study the convergence of the spiked eigenvalues of the SCC matrices, and in \cite{PartII} we will prove a central limit theorem for the spiked eigenvalues. For all these proofs, the  local law for $G$ and the eigenvalue rigidity proved in this paper play central roles. 
 \nc

This paper is organized as follows.  In Section \ref{main_result}, we define our model and state the main results---Theorem \ref{lem null} and Theorem \ref{main_thm1}, which give the eigenvalue rigidity and Tracy-Widom law, and Theorem \ref{thm_local} and Theorem \ref{thm_largerigidity}, which give the local laws for the resolvent $G$. In Section \ref{secsectools}, we introduce the notations and collect some basic tools that will be used in the proof. Section \ref{secpf1} is devoted to the proof of  Theorem \ref{lem null} and Theorem \ref{thm_largerigidity}, and Section \ref{secpf2} contains the proof of Theorem \ref{main_thm1}. Finally, the proof of Theorem \ref{thm_local}  is divided into two parts: in Section \ref{pf thmlocal}, we prove a weaker version of Theorem \ref{thm_local}, which gives the entrywise local law for $G$; the proof of Theorem \ref{thm_local}  is then completed in Section \ref{sec_aniso} based on the results in Section \ref{pf thmlocal}.

\vspace{5pt}

\noindent{\bf Conventions.} 
The fundamental large parameter is $n$ and we always assume that $p,q$ are comparable to $n$. All quantities that are not explicitly constant may depend on $n$, and we usually omit $n$ from our notations. We use $C$ to denote a generic large positive constant, whose value may change from one line to the next. Similarly, we use $\epsilon$, $\tau$, $\delta$ and $c$ to denote generic small positive constants. If a constant depends on a quantity $a$, we use $C(a)$ or $C_a$ to indicate this dependence. 
For two quantities $a_n$ and $b_n$ depending on $n$, the notation $a_n = \OO(b_n)$ means that $|a_n| \le C|b_n|$ for some constant $C>0$, and $a_n=\oo(b_n)$ means that $|a_n| \le c_n |b_n|$ for some positive sequence $c_n\downarrow 0$ as $n\to \infty$. We also use the notations $a_n \lesssim b_n$ if $a_n = \OO(b_n)$, and $a_n \sim b_n$ if $a_n = \OO(b_n)$ and $b_n = \OO(a_n)$. For a matrix $A$, we use $\|A\|:=\|A\|_{l^2 \to l^2}$ to denote the operator norm,  $\|A\|_F $ to denote the Frobenius norm, and $\|A\|_{\max}:=\max_{i,j}|A_{ij}|$ to denote the max norm.  
For a vector $\mathbf v=(v_i)_{i=1}^n$, $\|\mathbf v\|\equiv \|\mathbf v\|_2$ stands for the Euclidean norm. 
In this paper, we often write an identity matrix as $I$ or $1$ without causing any confusions. If two random variables $X$ and $Y$ have the same distribution, we write $X\stackrel{d}{=} Y$.

\vspace{5pt}

\noindent{\bf Acknowledgements.} The author would like to thank Zongming Ma for bringing this problem to his attention and for helpful discussions. 

\section{Definitions and main results}\label{main_result}

\subsection{The model}

We consider two data matrices
$$ {\mathcal X}: = \mathbf C^{1/2}_1 X , \quad  {\mathcal Y}: = \mathbf C^{1/2}_2 Y ,$$
where $ \mathbf C_1$ and $\mathbf C_2$ are $p\times p$ and $q\times q$ deterministic population covariance matrices, 
and 
$X=(x_{ij})$ and $Y=(y_{ij})$ are $p\times n$ and $q\times n$ random matrices, respectively. We assume that the entries $x_{ij}$, $1 \le i \le p$, $1\le j \le n$ and $y_{ij}$, $1 \le i\le q$, $1\le j \le n$ are independent (but not necessarily identically distributed) random variables satisfying
\begin{equation}\label{assm1}
\mathbb{E} x_{ij}=\mathbb{E} y_{ij} =0, \ \quad \ \mathbb{E} \vert x_{ij} \vert^2=\mathbb{E} \vert y_{ij} \vert^2  =n^{-1}.
\end{equation}
For definiteness, in this paper we focus on the real case, that is, all the random variables are real. However, we remark that our proof can be applied to the complex case after minor modifications. 
In this paper, we consider the high dimensional setting, i.e., 
\begin{equation}
c_1(n) := \frac{p}{n} \to \hat c_1 \in (0,1), \ \ \  c_2(n) := \frac{q}{n} \to \hat c_2 \in (0,1), \ \ \ \text{with} \ \ \ c_1(n) + c_2(n) \in (0,1).  \label{assm2}
 \end{equation}
For simplicity, we will always abbreviate $c_1(n)\equiv c_1$ and $c_2(n)\equiv c_2$ in the rest of the paper. Without loss of generality, we can assume that $c_1\ge c_2$. 
In this paper, we are interested in the eigenvalues of the  sample canonical correlation matrix 
\begin{align*}
\cal C_{\cal X\cal Y}:= \left({\cal X} {\cal X}^T\right)^{-1/2} \left( {\cal X} {\cal Y}^T\right)\left(  {\cal Y} {\cal Y}^T\right)^{-1}\left( {\cal Y} {\cal X}^T\right)  \left( {\cal X} {\cal X}^T\right)^{-1/2} 
\end{align*}
Since the canonical correlations are invariant under block diagonal transformations $(X,Y)\to (\bC_1^{1/2}X,\bC_2^{1/2}Y)$,  it is equivalent to study the eigenvalues of 
\begin{align*}
\cal C_{X Y}:= S_{xx}^{-1/2} S_{xy} S_{yy}^{-1}S_{yx}S_{xx}^{-1/2} ,
\end{align*}
where 
\be\label{def Sxy}S_{xx} := {X}{ X}^T, \quad S_{yy} := {Y}{ Y}^T, \quad S_{xy} = S^T_{yx}:=XY^T.\ee
We will also use the following matrix
$$\cal C_{YX}:= S_{yy}^{-1/2} S_{yx} S_{xx}^{-1}S_{xy}S_{yy}^{-1/2},$$
and denote its eigenvalues by $ \lambda_1 \ge \lambda_2 \ge \cdots \ge \lambda_q\ge 0$. Note that $\cal C_{ X Y}$ shares the same eigenvalues with $\cal C_{ Y X}$, except that it has $(p-q)$ more trivial zero eigenvalues $ \lambda_ {q+1} = \cdots = \lambda_{p}=0$. 

We now summarize the main assumptions for future reference. For our purpose, we shall relax the assumption \eqref{assm1} a little bit. 


\begin{assumption}\label{main_assm}
Fix a small constant $\tau>0$. Let $X=(x_{ij})$ and $Y=(Y_{ij})$ be two real independent $p\times n$ and $q\times n$ matrices, whose entries are independent random variables that satisfy the following moment conditions: 
\begin{align}
\max_{i,j}\left|\mathbb{E} x_{ij}\right|   \le n^{-2-\tau}, \quad & \max_{i,j}\left|\mathbb{E} y_{ij}\right|   \le n^{-2-\tau},\label{entry_assm0} \\
\max_{i,j}\left|\mathbb{E} | x_{ij} |^2  - n^{-1}\right|   \le n^{-2-\tau}, \quad & \max_{i,j}\left|\mathbb{E} | y_{ij} |^2  - n^{-1}\right|   \le n^{-2-\tau}. \label{entry_assm1}
\end{align}
Note that (\ref{entry_assm0}) and (\ref{entry_assm1}) are slightly more general than (\ref{assm1}).
Moreover, we assume that
\begin{equation}
\tau \le c_2 \le  c_1 , \quad  c_1 + c_2 \le 1 -\tau.  \label{assm20}
 \end{equation} 
\end{assumption}



\subsection{The Tracy-Widom limit and eigenvalue rigidity}

We denote the ESD of $\cal C_{YX}$ by 
$$ F_n(x):= \frac1q\sum_{i=1}^q \mathbf 1_{\lambda_i \le x}.$$
If $X$ and $Y$ are both i.i.d.\,Gaussian matrices, then it is known that, almost surely, $F_n$ converges weakly to a deterministic probability distribution $F(x)$ with density \cite{Wachter}
\be\label{LSD}
f(x)= \frac{1}{2\pi c_2} \frac{\sqrt{(\lambda_+ - x)(x-\lambda_-)}}{x(1-x)}, \quad \lambda_- \le x \le \lambda_+,
\ee
where
\be\label{lambdapm}\lambda_\pm:= \left( \sqrt{c_1(1-c_2)} \pm \sqrt{c_2(1-c_1)}\right)^2.\ee
The convergence of the ESD actually holds under a more general distribution assumption on the entries of $X$ and $Y$ as proved by \cite{CCA_ESD}. 
We define the quantiles of the density \eqref{LSD}, which correspond to the classical locations of the eigenvalues of $\cal C_{YX}$. 
\begin{definition} [Classical locations of eigenvalues]
The classical location $\gamma_j$ of the $j$-th eigenvalue is defined as
\begin{equation}\label{gammaj}
\gamma_j:=\sup_{x}\left\{\int_{x}^{+\infty} f(x)\dd x > \frac{j-1}{q}\right\},
\end{equation}
where $f$ is defined in \eqref{LSD}. Note that we have $\gamma_1 = \lambda_+$ and $\lambda_+ - \gamma_j \sim (j/n)^{2/3}$ for $j>1$.
\end{definition}

 

Before stating the main results, we first define the following notion of stochastic domination, which was first introduced in \cite{Average_fluc} and subsequently used in many works on random matrix theory, such as \cite{isotropic,principal,local_circular,Delocal,Semicircle,Anisotropic}. It simplifies the presentation of the results and their proofs by systematizing statements of the form ``$\xi$ is bounded by $\zeta$ with high probability up to a small power of $N$".

\begin{definition}[Stochastic domination]\label{stoch_domination}
(i) Let
\[\xi=\left(\xi^{(n)}(u):n\in\bbN, u\in U^{(n)}\right),\hskip 10pt \zeta=\left(\zeta^{(n)}(u):n\in\bbN, u\in U^{(n)}\right)\]
be two families of nonnegative random variables, where $U^{(n)}$ is a possibly $n$-dependent parameter set. We say $\xi$ is stochastically dominated by $\zeta$, uniformly in $u$, if for any fixed (small) $\epsilon>0$ and (large) $D>0$, 
\[\sup_{u\in U^{(n)}}\bbP\left[\xi^{(n)}(u)>n^\epsilon\zeta^{(n)}(u)\right]\le n^{-D}\]
for large enough $n\ge n_0(\epsilon, D)$, and we shall use the notation $\xi\prec\zeta$. Throughout this paper, the stochastic domination will always be uniform in all parameters that are not explicitly fixed (such as matrix indices, and $z$ that takes values in some compact set). Note that $n_0(\epsilon, D)$ may depend on quantities that are explicitly constant, such as $\tau$ in Assumption \ref{main_assm}. If for some complex family $\xi$ we have $|\xi|\prec\zeta$, then we will also write $\xi \prec \zeta$ or $\xi=\OO_\prec(\zeta)$.

\medskip
\noindent (ii) We extend the definition of $\OO_\prec(\cdot)$ to matrices in the weak operator sense as follows. Let $A$ be a family of random matrices and $\zeta$ be a family of nonnegative random variables. Then $A=\OO_\prec(\zeta)$ means that $\left|\left\langle\mathbf v, A\mathbf w\right\rangle\right|\prec\zeta \| \mathbf v\|_2 \|\mathbf w\|_2 $ for any deterministic vectors $\mathbf v$ and $\mathbf w$. 

\medskip
\noindent (iii) We say an event $\Xi$ holds with high probability if for any constant $D>0$, $\mathbb P(\Xi)\ge 1- n^{-D}$ for large enough $n$. {Moreover, we say an event $\Xi$ holds with high probability on an event $\Omega$, if for any constant $D>0$, $\mathbb P(\Omega\setminus \Xi)\le n^{-D}$ for large enough $n$. In particular, $\xi \prec \zeta$ on $\Omega$ means that for any fixed $\epsilon>0$, $\xi \le n^\e\zeta$ with high probability on $\Omega$.
}
\end{definition}

For $X$ and $Y$, we introduce the following bounded support condition.

\begin{definition}[Bounded support condition] \label{defn_support}
We say a random matrix $X$ satisfies the {\it{bounded support condition}} with $\phi_n$, if
\begin{equation}
\max_{i,j}\vert x_{ij}\vert \prec  \phi_n. \label{eq_support}
\end{equation}
Usually $ \phi_n$ is a deterministic parameter and satisfies $ n^{-{1}/{2}} \leq \phi_n \leq n^{- c_\phi} $ for some (small) constant $c_\phi>0$. Whenever (\ref{eq_support}) holds, we say that $X$ has support $\phi_n$. 
\end{definition}

Then we have the following eigenvalue rigidity and edge universality result for $\cal C_{Y X }$, which extends the result in \cite{CCA2}. 
\begin{theorem}\label{lem null}
Suppose Assumption \ref{main_assm} holds. Suppose $X$ and $Y$ have bounded support $\phi_n$ such that $ n^{-{1}/{2}} \leq \phi_n \leq n^{- c_\phi} $ for some constant $c_\phi>0$. Assume that 
\be\label{conditionA3} 
\begin{split}
& \max_{i,j}\mathbb{E} | x_{ij} |^3  =\OO(n^{-3/2}), \quad  \max_{i,j}\mathbb{E} | x_{ij} |^4  \prec n^{-2},\\
&\max_{i,j}\mathbb{E} | y_{ij} |^3=\OO(n^{-3/2}), \quad \max_{i,j}\mathbb{E} | y_{ij} |^4  \prec n^{-2}.
 \end{split}
\ee
Then the eigenvalues $\lambda_i$ of the {sample canonical correlation (SCC)} matrix $\cal C _{ YX}$ satisfy the following eigenvalue rigidity estimate: if $\lambda_-\ge \e$ for some constant $\e>0$, then
\be\label{rigidity}
|\lambda_i - \gamma_i | \prec \left[i \wedge (q+1-i)\right]^{-1/3} n^{-2/3},\quad 1\le i \le q,
\ee
Otherwise, if $\lambda_-=\oo(1)$, then \eqref{rigidity} hold for all $1\le i \le (1-\e)q$ for any constant $\e>0$. Moreover, we have that for any fixed $k$,
\begin{equation}\label{joint TW}
\begin{split}
\lim_{n\to \infty}\mathbb{P}&\left( \left(n^{\frac{2}{3}}\frac{\lambda_{i} - \lambda_+}{c_{TW}} \leq s_i\right)_{1\le i \le k} \right) = \lim_{n\to \infty} \mathbb{P}^{GOE}\left(\left(n^{\frac{2}{3}}(\lambda_i - 2) \leq s_i\right)_{1\le i \le k} \right), 
\end{split}
\end{equation}
for all $s_1 , s_2, \ldots, s_k \in \mathbb R$, where
$$c_{TW}:= \left[ \frac{\lambda_+^2 (1-\lambda_+)^2}{\sqrt{c_1c_2(1-c_1)(1-c_2)}}\right]^{1/3},$$
and $\mathbb P^{GOE}$ stands for the law of the Gaussian orthogonal ensemble (GOE) of dimension $n\times n$. 
\end{theorem}

Recall that the joint distribution of the $k$ largest eigenvalues of GOE can be written in terms of the Airy kernel for any fixed $k$ \cite{Forr}. Moreover, taking $k=1$ in \eqref{joint TW}, we obtain that
$$n^{\frac{2}{3}}\frac{\lambda_{i} - \lambda_+}{c_{TW}} \Rightarrow F_1,$$
where $F_1$ is the Type-1 Tracy-Widom distribution.
The result \eqref{joint TW} was proved in \cite{CCA2} under the assumption that all the moments of $\sqrt{n}x_{ij}$ and $\sqrt{n}y_{ij}$ exist. On the other hand, 
combining our result with  
 a simple cutoff argument allows us to obtain the following corollary under the finite $(4+\e)$-th moment assumption. Since we do not assume the entries of $X$ and $Y$ are identically distributed, the means and variances of the truncated entries may be different. This is why we assume the slightly more general conditions \eqref{entry_assm0} and \eqref{entry_assm1}.

\begin{corollary}\label{main_cor}
Suppose \eqref{assm20} holds. Assume that $X=(x_{ij})$ and $Y=(Y_{ij})$ are two real independent $p\times n$ and $q\times n$ matrices, whose entries are independent random variables that satisfy \eqref{assm1} and  
\be\label{condition_4e} 
\max_{i,j}\mathbb{E}  |\sqrt{n} x_{ij} | ^{4+\tau}  \le C, \quad \max_{i,j}\mathbb{E}|\sqrt{n} y_{ij} |^{4+\tau}  \le C,
\ee 
for some constants $\tau,C>0$. 
{Then the Tracy-Widom law \eqref{joint TW} holds. Moreover, the rigidity estimate \eqref{rigidity} holds on an event $\Omega$ (cf. \eqref{event Omega}) with probability $1-\oo(1)$. } 
\end{corollary}
\begin{proof}
We choose the constants $c_\phi >0$ small enough such that $\left(n^{1/2-c_\phi}\right)^{4+\tau} \ge n^{2+\e}$ for some constant $\e>0$. Then we introduce the following truncation  
\be\label{event Omega}
\wt X :=\mathbf 1_{\Omega} X, \quad \wt Y :=\mathbf 1_{\Omega} Y,\quad \Omega :=\left\{\max_{i,j} |x_{ij}|\le n^{-c_\phi},\max_{i,j} |y_{ij}|\le n^{-c_\phi}\right\}.
\ee
By the moment conditions (\ref{condition_4e}) and a simple union bound, we have
\begin{equation}\label{XneX}
\mathbb P(\wt X \ne X, \wt Y \ne Y) =\OO ( n^{-\e}).
\end{equation}
Using (\ref{condition_4e}) and integration by parts, it is easy to verify that 
\begin{align*}
\mathbb E  \left|x_{ij}\right|1_{|x_{ij}|> n^{-c_\phi}} =\OO(n^{-2-\e}), \quad \mathbb E \left|x_{ij}\right|^2 1_{|x_{ij}|> n^{-c_\phi}} =\OO(n^{-2-\e}),
\end{align*}
which imply that
$$|\mathbb E  \tilde x_{ij}| =\OO(n^{-2-\e}), \quad  \mathbb E |\tilde x_{ij}|^2 = n^{-1} + \OO(n^{-2-\e}).$$
Moreover, we trivially have
$$\mathbb E  |\tilde x_{ij}|^4 \le \mathbb E  |x_{ij}|^4 =\OO(n^{-2}).$$
Similar estimates also hold for the entries of $Y$. Hence $\wt X$ and $\wt Y$ are random matrices satisfying Assumption \ref{main_assm} and condition \eqref{conditionA3}. 
Now combing \eqref{XneX} and Theorem \ref{lem null}, we conclude the corollary.
\end{proof}

If we assume that the entries of $X$ and $Y$ are identically distributed, respectively, then the Tracy-Widom law actually holds under the weaker tail condition \eqref{tail_cond}.

\begin{theorem}
\label{main_thm1}
Suppose \eqref{assm20} holds. Assume that $x_{ij}=n^{-1/2} \wh x_{ij}$ and $y_{ij}= n^{-1/2} \wh y_{ij}$, where $\{\wh x_{ij}\}$ and $\{\wh y_{ij}\}$ are independent  families of i.i.d. random variables with mean zero and variance one. Then for any fixed $k$, \eqref{joint TW} holds under the following tail condition:
\begin{equation}
\lim_{t \rightarrow \infty } t^4 \left[\mathbb{P}\left( \vert \wh x_{11} \vert \geq t\right)+ \mathbb{P}\left( \vert \wh y_{11} \vert \geq t\right)\right]=0. \label{tail_cond}
\end{equation}
 
\end{theorem}

\begin{remark}
The tail condition \eqref{tail_cond} has been shown to be a necessary condition for the largest eigenvalue to converge to the Tracy-Widom law in the case of Wigner matrices \cite{LY} and sample covariance matrices \cite{DY}. For example, let $H$ be a Wigner matrix. It is shown in \cite{LY} that if the entries of $H$ do not satisfy the tail condition as in \eqref{tail_cond}, then its largest eigenvalue $\lambda_1$ satisfies that for any fixed $s>0$, $ \P(\lambda_1>s)>c_s$ for a constant $c_s>0$ depending on $s$. This shows that the Tracy-Widom law cannot hold for $\lambda_1$. A similar result is shown for sample covariance matrices in \cite{DY}. 

We conjecture that the tail condition \eqref{tail_cond} is also necessary for the Tracy-Widom law of the largest sample canonical correlation coefficient, but the proof in \cite{DY,LY} cannot be applied to our setting directly. To illustrate the point, suppose the entries of $X$ do not satisfy \eqref{tail_cond}, and the entries of $Y$ are i.i.d. Gaussian. Then the largest singular value $\lambda_1$ of $X$ satisfies that for any fixed $s>0$, $ \P(\lambda_1>s)>c_s$ for a constant $c_s>0$ depending on $s$ \cite{DY}. However, it is not clear whether such a result will be sufficient to show that the largest eigenvalue of $\cal C_{X Y}= S_{xx}^{-1/2} S_{xy} S_{yy}^{-1}S_{yx}S_{xx}^{-1/2}$ deviates from $\lambda_+$. To solve this problem, we also need to understand the behaviors of the singular vectors of $X$, and we will pursue it in a future work. 
\end{remark}

\begin{remark}
If we do not assume that the entries of $X$ and $Y$ are identically distributed, then our proof still works if we assume that 
  \begin{equation*}
\lim_{t \rightarrow \infty } t^4 \left[\max_{i,j}\mathbb{P}\left( \vert \wh x_{ij} \vert \geq t\right)+ \max_{i,j}\mathbb{P}\left( \vert \wh y_{ij} \vert \geq t\right)\right]=0.  
\end{equation*}
However, this is not the sharp moment condition. For example, we can consider a case where $n^\e$ many entries of $X$ and $Y$ have variance 1 and infinite third moments, and all the other entries are i.i.d. random variables satisfying \eqref{tail_cond}. Then using a perturbation argument (that is similar to the one used in Section \ref{secpf2}), we can show that the Tracy-Widom law \eqref{joint TW} still holds as long as $\e$ is small enough. The sharp moment condition in the non-i.i.d. case is still unknown even for the cases of Wigner and sample covariance matrices. 
\end{remark}

\subsection{The linearization method and local law}\label{sec_maintools}


The self-adjoint linearization method has been proved to be useful in studying the local laws of random matrices of the Gram type \cite{Alt_Gram, AEK_Gram, DY2,CCA2,CCA_TW2,Anisotropic, XYY_circular,yang2018}. We now introduce a generalization of this method, which will be the starting point of this paper. 


For now, we assume that $XX^T$ and $YY^T$ are both non-singular almost surely. This is trivially true if, say, the entries of $X$ and $Y$ have continuous densities. For any $\lambda > 0$, it is an eigenvalue of $\cal C_{XY}$ if and only if the following equation holds:
\be\label{eq det0}\det\left( \left(XY^T\right)\left(YY^T\right)^{-1}\left(YX^T\right) - \lambda XX^T\right) = 0 .\ee
 By Schur complement, it is equivalent to
$$\det \begin{pmatrix} \lambda XX^T &  \lambda^{1/2}XY^T \\  \lambda^{1/2}YX^T &  \lambda YY^T \end{pmatrix} = 0  \ \  \Leftrightarrow \ \ \det \begin{pmatrix} X & 0 \\ 0 &  Y \end{pmatrix} \begin{pmatrix} \lambda I_n & \lambda^{1/2} I_n \\ \lambda^{1/2}I_n & \lambda I_n \end{pmatrix} \begin{pmatrix} X^T & 0 \\ 0 &  Y^T \end{pmatrix} = 0.$$
Using Schur complement again, if $\lambda \notin \{0, 1\}$, then it is equivalent to 
\be\label{deteq}\det \begin{pmatrix} 0 & \begin{pmatrix} X & 0\\ 0 & Y\end{pmatrix}\\ \begin{pmatrix}  X^T & 0\\ 0 &  Y^T\end{pmatrix}  & \begin{pmatrix}  \lambda  I_n & \lambda^{1/2}I_n\\ \lambda^{1/2} I_n &  \lambda I_n\end{pmatrix}^{-1}\end{pmatrix} = 0 . \ee

Inspired by the above discussion, we define the following $(p+q+2n)\times (p+q+2n)$ self-adjoint block matrix 
 \begin{equation}\label{linearize_block}
   H(\lambda) : = \begin{pmatrix} 0 & \begin{pmatrix}  X & 0\\ 0 &  Y\end{pmatrix}\\ \begin{pmatrix} X^T & 0\\ 0 &  Y^T\end{pmatrix}  & \begin{pmatrix}  \lambda  I_n & \lambda^{1/2}I_n\\ \lambda^{1/2} I_n &  \lambda I_n\end{pmatrix}^{-1}\end{pmatrix} .
 \end{equation}
 We can also extend the argument $\lambda$ to $z\in \C_+:=\{z\in \C: \im z>0\}$ and define $H(z)$ in general, where we take $z^{1/2}$ to be the branch with positive imaginary part.  We then define the resolvent (or Green's function) as
\begin{equation}\label{eqn_defG}
 G(z):= \left[H(z)\right]^{-1} , \quad z\in \mathbb C_+  ,
\end{equation}
whenever the inverse exists. 



\begin{definition}[Index sets]\label{def_index}
For simplicity of notations, we define the index sets
$$\cal I_1:=\llbracket 1,p\rrbracket, \ \quad \ \cal I_2:=\llbracket p+1,p+q\rrbracket,$$
and 
$$\cal I_3:=\llbracket p+q+1,p+q+n\rrbracket, \ \quad \ \cal I_4:=\llbracket p+q+n+1,p+q+2n\rrbracket. $$
 We will consistently use the latin letters $i,j\in\sI_{1,2}$ and greek letters $\mu,\nu\in\sI_{3,4}$. Moreover, we shall use the notations $\fa,\fb\in \cal I:=\cup_{i=1}^4 \cal I_i$. We label the indices of the matrices according to
 $$X= (x_{i\mu}:i\in \mathcal I_1, \mu \in \mathcal I_3), \quad Y= (y_{j\nu}:j\in \mathcal I_2, \nu \in \mathcal I_4).$$
Moreover, we denote $\overline i:= i+p$ for $i\in \cal I_1$, $\overline j:= j-p$ for $j\in \cal I_2$, $\overline \mu : = \mu +n $ for $\mu \in \cal I_3$, and $\overline \nu : = \nu - n $ for $\nu \in \cal I_4$. 
\end{definition}

\begin{definition}[Resolvents]\label{resol_not}
We denote the $\cal I_\al \times \cal I_\al$ block of $ G(z)$ by $ \cal G_\al(z)$ for $\al=1,2,3,4$. We denote the $(\cal I_1\cup \cal I_2)\times (\cal I_1\cup \cal I_2)$ block of $ G(z)$ by $ \cal G_L(z)$, the $(\cal I_1\cup \cal I_2)\times (\cal I_3\cup \cal I_4)$ block by $ \cal G_{LR}(z)$, the $(\cal I_3\cup \cal I_4)\times (\cal I_1\cup \cal I_2)$ block  by $ \cal G_{RL}(z)$, and the $(\cal I_3\cup \cal I_4)\times (\cal I_3\cup \cal I_4)$ block by $ \cal G_R(z)$. We introduce the following random quantities:
\be\label{def m1234} m_\al(z) :=\frac1n\tr  \cal G_{\al}(z) = \frac{1}{n}\sum_{\fa \in \cal I_\al}  G_{\fa\fa}(z) ,\quad \al=1,2,3,4. \ee
Recalling the notations in \eqref{def Sxy}, we define $\cal H:=S_{xx}^{-1/2}S_{xy}S_{yy}^{-1/2}$ and
\be\label{Rxy}
\begin{split}
 R_1(z):=(\cal C_{XY}-z)^{-1}&=(\cal H\cal H^T-z)^{-1}, \\
 R_2(z):=(\cal C_{YX}-z)^{-1}&=(\cal H^T\cal H-z)^{-1},  \quad m(z):= q^{-1}\tr  R_2(z).
 \end{split}
\ee
Note that we have $R_1\cal H = \cal HR_2$, $\cal H^T R_1 = R_2 \cal H^T $, and 
\be\label{R12} \tr  R_1 = \tr  R_2 - \frac{p-q}{z}= q  m(z) - \frac{p-q}{z},\ee
since $\cal C_{XY}$ has $(p-q)$ more zeros eigenvalues than $\cal C_{YX}$. 
\end{definition}

By Schur complement formula, we immediately obtain that
\be\label{GL1}
\begin{split}
 \cal G_L & = \begin{pmatrix} S_{xx}^{-1/2}R_1S_{xx}^{-1/2} & - z^{-1/2}S_{xx}^{-1/2}R_1\cal HS_{yy}^{-1/2} \\ - z^{-1/2}S_{yy}^{-1/2}\cal H^T R_1S_{xx}^{-1/2} & S_{yy}^{-1/2}R_2S_{yy}^{-1/2}\end{pmatrix}, 
\end{split}
\ee
and
\begin{equation*}
\begin{split}\cal G_1= S_{xx}^{-1/2}R_1S_{xx}^{-1/2} = \left(S_{xy}S_{yy}^{-1}S_{yx} - z S_{xx}\right)^{-1}, \\ \cal G_2 = S_{yy}^{-1/2}R_2S_{yy}^{-1/2}= \left(S_{yx}S_{xx}^{-1}S_{xy}  - z S_{yy}\right)^{-1}.
\end{split}
\end{equation*}
The other blocks are
\be\label{GR1}
\cal G_R =   \begin{pmatrix}  z  I_n & z^{1/2}I_n\\ z^{1/2}I_n &  z  I_n\end{pmatrix} +   \begin{pmatrix}  z  I_n & z^{1/2}I_n\\ z^{1/2}I_n &  z  I_n\end{pmatrix}  \begin{pmatrix} X^T & 0 \\ 0 & Y^T \end{pmatrix} \cal G_L \begin{pmatrix} X & 0 \\ 0 &  Y \end{pmatrix} \begin{pmatrix}  z  I_n & z^{1/2}I_n\\ z^{1/2}I_n &  z  I_n\end{pmatrix}  ,
\ee
and
\be\label{GLR1}
\begin{split}
&{\cal G}_{LR}(z)= -\cal G_L(z) \begin{pmatrix} X & 0 \\ 0 &  Y \end{pmatrix} \begin{pmatrix}  z  I_n & z^{1/2}I_n\\ z^{1/2}I_n &  z  I_n\end{pmatrix}  , \\ 
&{\cal G}_{RL}(z)= -  \begin{pmatrix}  z  I_n & z^{1/2}I_n\\ z^{1/2}I_n &  z  I_n\end{pmatrix}  \begin{pmatrix} X^T & 0 \\ 0 & Y^T \end{pmatrix} {\cal G}_L(z).
\end{split}
\ee
Expanding the product in \eqref{GR1} using \eqref{GL1} and calculating the partial traces, one can verify directly that
\be\label{m3m}m_3 (z)= z+\frac1n\left( -2 z p - z^{2}\tr R_1 + z \tr  R_2\right)= c_2 z(1-z)m(z) + (1-c_1-c_2)z,\ee
and 
\be\label{m4m}
\begin{split}m_4(z)& = z+\frac1n\left(  - 2z q - z^{2}\tr  R_2+ z  \tr R_1 \right) \\
&=  c_2 z(1-z)m(z) - (c_1-c_2)+ (1- 2c_2) z.\end{split}\ee
where we also used \eqref{R12}. In particular, we have the identity
\be\label{m34}
m_3(z) - m_4 (z)= (1-z)(c_1-c_2) . 
\ee

We now give the deterministic limit of $m_\al$, $\al=1,2,3,4$, as $n\to \infty$: for $\lambda_\pm$ defined in \eqref{lambdapm},
\begin{align}
&m_{1c}(z) 
= \frac{ - z +c_1+c_2+\sqrt{(z-\lambda_-)(z-\lambda_+)} }{2(1-c_1)z(1-z)} - \frac{c_1}{(1-c_1)z}, \label{m1c}\\
&m_{2c}(z) = \frac{ -z +c_1 + c_2+ \sqrt{(z-\lambda_-)(z-\lambda_+)}}{2(1-c_2)z(1-z)} -\frac{c_2}{(1-c_2)z}, \label{m2c}\\
&m_{3c}(z) 
= \frac{1}{2}\left[ (1-2c_1) z + c_1 - c_2 + \sqrt{(z-\lambda_-)(z-\lambda_+)}\right] , \label{m3c}\\
&m_{4c}(z)= \frac{1}{2}\left[ (1-2c_2) z +  c_2 - c_1 + \sqrt{(z-\lambda_-)(z-\lambda_+)}\right] ,\label{m4c}
\end{align}
where we take the branch of the square root functions with non-negative imaginary parts. One can verify when $z\to 1$, $m_{1c}(z)$ and $m_{2c}(z)$ have finite limits, which we define as $m_{1c}(1)$ and $m_{2c}(1)$. Moreover, by \eqref{m3m}  the deterministic limit of $m$ is
\begin{align}\label{mc}
&m_c(z)= \frac{m_{3c}(z) + (c_1+c_2 - 1)z}{c_2z(1-z)} = \frac{1-c_2}{c_2} m_{2c}(z).
\end{align}
We then define the matrix limit of $G(z)$ as
\be \label{defn_pi}
\Pi(z) := \begin{pmatrix} \begin{pmatrix} c_1^{-1}m_{1c}(z)I_p & 0\\ 0 & c_2^{-1}m_{2c}(z)I_q\end{pmatrix} & 0 \\ 0  & \begin{pmatrix}  m_{3c}(z)I_n  & h(z)I_n\\  h(z)I_n &  m_{4c}(z)  I_n\end{pmatrix}\end{pmatrix} ,\ee
where
\be
\begin{split}
h(z):&=  \frac{z^{-1/2}m_{3c}(z)}{1+(1-z)m_{2c}(z)} =  \frac{z^{-1/2}m_{4c}(z)}{1+(1-z)m_{1c}(z)} \\
&= \frac{z^{1/2}}{2} \left[ - z + (2-c_1-c_2) + {\sqrt{(z-\lambda_-)(z-\lambda_+)}}\right].\label{hz}
\end{split}
\ee

Through a direct calculation, we can check that the following equations hold for $(m_{1c},m_{2c},m_{3c},m_{4c})$:
\begin{align}
& m_{1c}= - \frac{c_1}{m_{3c}} , \quad {m_{2c}} = -\frac{ c_2}{m_{4c}}, \quad  m_{3c}(z) - m_{4c} (z)= (1-z)(c_1-c_2) ,\label{selfm12}\\
& m_{3c}(z) 
=\frac{ 1-(z-1)m_{2c}(z)}{z^{-1} - (m_{1c}(z)+m_{2c}(z)) + (z-1)m_{1c}(z)m_{2c}(z)}, \label{selfm3}  \\
& m_{3c}^2(z) + \left[ (2c_1 -1)z - c_1+c_2\right]m_{3c}(z) + c_1(c_1-1)z(z-1) =  0. \label{selfm32}
\end{align}
{Conversely, we can also solve these equations to get $(m_{1c},m_{2c},m_{3c},m_{4c})$. First, using equations \eqref{selfm12} and \eqref{selfm3} we obtain that 
\begin{align*}
0&= z^{-1}m_{3c} + m_{3c}\left(\frac{c_1}{m_{3c}}+\frac{c_2}{m_{4c}}\right) + (z-1)m_{3c}\frac{c_1c_2}{m_{3c}m_{4c}} -1-(z-1)\frac{c_2}{m_{4c}} \\
&= z^{-1}m_{3c} + (c_1-1) +\frac{c_2 m_{3c} + c_1c_2 (z-1) - c_2(z-1)}{m_{3c}-(1-z)(c_1-c_2)}  ,
\end{align*}
which gives equation \eqref{selfm32} after multiplying $z\left[m_{3c}-(1-z)(c_1-c_2)\right]$ on both sides. Similarly, from equations \eqref{selfm12} and \eqref{selfm32}, we can also derive equation \eqref{selfm3}. Hence the system of equations \eqref{selfm12} and \eqref{selfm32} is equivalent to the system of equations \eqref{selfm12} and \eqref{selfm3}. Second, solving \eqref{selfm32} and using 
\be\label{add eq34} \left[ (2c_1 -1)z - c_1+c_2\right]^2 - 4 c_1(c_1-1)z(z-1)  = (z-\lambda_-)(z-\lambda_+),\ee
we can solve \eqref{selfm32} to get $m_{3c}$ in \eqref{m3c} (by taking the proper branch of the square root function). Then using the third equation in \eqref{selfm12}, we can obtain $m_{4c}$ in \eqref{m4c}. Finally, plugging $m_{3c}$ into the first equation in \eqref{selfm12} and using \eqref{add eq34}, we can obtain that
\begin{align*}
m_{1c}= -2c_1\frac{(1-2c_1) z + c_1 - c_2 - \sqrt{(z-\lambda_-)(z-\lambda_+)}}{[(1-2c_1) z + c_1 - c_2]^2 - (z-\lambda_-)(z-\lambda_+)} = - \frac{(1-2c_1) z + c_1 - c_2 - \sqrt{(z-\lambda_-)(z-\lambda_+)}}{  2(1-c_1)z(1-z)},
\end{align*}
which gives \eqref{m1c}. Similarly, plugging $m_{4c}$ into the second equation in \eqref{selfm12}, we can obtain \eqref{m2c}.}


For simplicity of notations, we introduce the notion of generalized entries.

\begin{definition}[Generalized entries]
For $\mathbf v,\mathbf w \in \mathbb C^{\mathcal I}$, $\fa\in \mathcal I$ and an $\mathcal I\times \mathcal I$ matrix $\cal A$, we shall denote
\begin{equation}
\cal A_{\mathbf{vw}}:=\langle \mathbf v,\cal A\mathbf w\rangle, \quad  \cal A_{\mathbf{v}\fa}:=\langle \mathbf v,\cal A\mathbf e_\fa\rangle, \quad \cal A_{\fa\mathbf{w}}:=\langle \mathbf e_\fa,\cal A\mathbf w\rangle,
\end{equation}
where $\mathbf e_\fa$ is the standard unit vector along $\fa$-th coordinate axis, and the inner product is defined as $\langle \mathbf v, \mathbf w\rangle:= \bv^* \bw$ with $\bv^*$ denoting the conjugate transpose. Given a vector $\mathbf v\in \mathbb C^{\mathcal I_\al}$, $\al=1,2,3,4$, we always identify it with its natural embedding in $\C^{\cal I}$. For example, we shall identify $\mathbf v\in \mathbb C^{\mathcal I_1}$ with $\left( {\begin{array}{*{20}c}
   {\mathbf v}  \\
   \mathbf 0_{q+2n} \\
\end{array}} \right)\in \C^{\cal I}$.
\end{definition}

Now we are ready to state the local laws for $G(z)$. For any constant $\e >0$, we define a domain of the spectral parameter $z$ as
\begin{equation}
S(\e):= \left\{z=E+ \ii \eta: \e \leq E \leq 1, n^{-1+\e} \leq \eta \leq \e^{-1} \right\}. \label{SSET1}
\end{equation}
We define the distance to the two edges as
\begin{equation}
\kappa \equiv \kappa_E := \min\left\{ \vert E -\lambda_-\vert,\vert E -\lambda_+\vert\right\} , \ \ \text{for } z= E+\ii \eta.\label{KAPPA}
\end{equation}

\begin{theorem} [Local laws]\label{thm_local} 
Suppose the assumptions of Theorem \ref{lem null} hold. Then for any fixed $\e>0$, the following estimates hold. 
\begin{itemize}
\item[(1)] {\bf Anisotropic local law}: For any $z\in S(\epsilon)$ and deterministic unit vectors $\mathbf u, \mathbf v \in \mathbb C^{\mathcal I}$,
\begin{equation}\label{aniso_law}
\left|  G_{\mathbf u\mathbf v}(z)   - \Pi_{\mathbf u\mathbf v} (z)  \right| \prec \phi_n + \Psi(z),
\end{equation}
where $\Psi(z)$ is a deterministic control parameter defined as
\begin{equation}\label{eq_defpsi}
\Psi (z):= \sqrt {\frac{\Im \, m_{c}(z)}{{n\eta }} } + \frac{1}{n\eta}, \quad z=E+\ii\eta.
\end{equation}

\item[(2)] {\bf Weak averaged local law}: For any $z \in S( \epsilon)$,  we have 
\begin{equation}
\vert m_{\al}(z)-m_{\al c}(z) \vert \prec \min \left\{\phi_n,\frac{\phi_n^2}{\sqrt{\kappa+\eta}}\right\} + \frac{1}{n\eta}, \quad \al=1,2,3,4. \label{aver_in1} 
\end{equation}
Moreover, outside of the spectrum we have the following stronger estimate
\begin{equation}\label{aver_out1}
 | m_\al(z)-m_{\al c}(z)|\prec \min \left\{\phi_n,\frac{\phi_n^2}{\sqrt{\kappa+\eta}}\right\} + \frac{1}{n(\kappa +\eta)} + \frac{1}{(n\eta)^2\sqrt{\kappa +\eta}}, \quad \al=1,2, 3,4,
\end{equation}
uniformly in $z\in S_{out}(\e):=S(\epsilon)\cap \{z=E+\ii\eta: E\notin [\lambda_-,\lambda_+], n\eta\sqrt{\kappa + \eta} \ge n^\epsilon\}$. 
\end{itemize}
The above estimates are uniform in the spectral parameter $z$ and any set of deterministic vectors of cardinality $n^{\OO(1)}$. 
\end{theorem}



With Theorem \ref{thm_local} as input, we can prove an even stronger estimate on $m(z)$ that is independent of $\phi_n$. This averaged local law will give the rigidity of eigenvalues for $\mathcal Q_1$ in \eqref{rigidity}. For fixed $\wt\e>0$, we define the following domains
$$
\wt S(\e,\wt \e):= \left\{z=E+ \ii \eta: \e \leq E \leq 1-\wt\e, n^{-1+\e} \leq \eta \leq \e^{-1} \right\}, \quad \wt S_{out}(\e,\wt\e):= \wt S(\e,\wt \e)\cap  S_{out}(\e).
$$
Note that these two domains are away from $z=1$. 

\begin{theorem}[Strong averaged local law] \label{thm_largerigidity}
Suppose the assumptions of Theorem \ref{lem null} hold. Then for any fixed $\epsilon,\wt\e>0$, we have
\begin{equation}
 \vert m(z)-m_{c}(z) \vert \prec (n \eta)^{-1}, \label{aver_in}
\end{equation}
uniformly in $z \in \wt S(\epsilon,\wt \e)$. Moreover, outside of the spectrum we have the following stronger estimate 
\begin{equation}\label{aver_out0}
 | m(z)-m_{c}(z)|\prec \frac{1}{n(\kappa +\eta)} + \frac{1}{(n\eta)^2\sqrt{\kappa +\eta}} , 
\end{equation}
uniformly in $z\in \wt S_{out}(\epsilon,\wt \e)$. These estimates also hold for $(m_{\al}(z) -m_{\al c}(z))$, $\al=1,2,3,4$. 
Finally, given any small constant $ 0< \e_0< 1 - \lambda_+$, we have
\begin{equation}
\max_{E\ge \e_0} \vert n(E)-n_{c}(E) \vert \prec n^{-1}  ,  \label{Kdist}
\end{equation}
where 
\begin{equation}\label{ncE}
n(E):=\frac{1}{q} \# \{ \lambda_j \ge E\}, \ \quad \ n_{c}(E):=\int^{1-\e_0}_E f(x)dx,
\end{equation}
for $f(x)$ defined in \eqref{LSD}.
\end{theorem}

The rest of the paper is devoted to proving these main results---Theorems \ref{lem null}, \ref{main_thm1}, \ref{thm_local} and \ref{thm_largerigidity}. {Before ending this section, we give a heuristic derivation of the limit $\Pi(z)$ in \eqref{defn_pi}. For the rigorous argument, we refer the reader to Section \ref{pf thmlocal}.

For $i\in \cal I_1\cup \cal I_2$, $\mu\in \cal I_3$ and $\overline \mu =\mu + n\in \cal I_4$, we denote by $H^{(i)}$ the $(p+q+2n-1)\times (p+q+2n-1)$ matrix obtained by removing the $i$-th row and column of $H$, and by $H^{[\mu]}$ the $(p+q+2n-2)\times (p+q+2n-2)$ matrix obtained by removing the $\mu$-th and $\overline \mu$-th rows and columns of $H$ (cf. Definition \ref{den minor}). Using Schur complement formula (cf. equation \eqref{resolvent1}), we obtain that for $i\in \cal I_1$, 
$$\frac{1}{G_{ii}}= - \sum_{\mu, \nu \in \cal I_3}x_{i\mu} x_{i\nu}G^{(i)}_{\mu \nu}.$$
Since $G^{(i)}$ is independent of the $i$-th row and column of $H$, the right-hand side should concentrate around its partial expectation over all $x_{i\mu}$, $\mu\in \cal I_2$:
\be\label{derv1}\frac{1}{G_{ii}}\approx - \frac{1}{n}\sum_{\mu \in \cal I_3} G^{(i)}_{\mu \mu} \approx -\frac{1}{n}\sum_{\mu \in \cal I_3} G_{\mu \mu} =-m_3 \quad \text{with high probability}.\ee
Here in the second step, we used the intuition that removing only one row and column out of the $(p+q+2n)$ rows and columns of $H$ should have a negligible affect on the partial trace over $\mu \in \cal I_3$. In the last step, we used the definition of $m_3$ in \eqref{def m1234}. Now from \eqref{derv1}, we immediately obtain that  
$$m_1= \frac{1}{n}\sum_{i\in \cal I_1} {G_{ii}}\approx -\frac{p/n}{m_3 }= -\frac{c_1}{m_3}\quad \text{with high probability},$$
which leads to the first equation in \eqref{selfm12}. With a similar argument, we can show that $m_2\approx  -c_2/m_4$, which leads to the second equation in \eqref{selfm12}. Moreover, $m_3$ and $m_4$ satisfy the third equation in \eqref{selfm12} by \eqref{m34}. 

On the other hand, using Schur complement formula (cf. equation \eqref{eq_res11}) and a similar argument as above, we obtain that for $\mu\in \cal I_3$, 
\be\nonumber
\begin{split}
\begin{pmatrix} 
   { G_{\mu\mu} } & { G_{\mu \overline \mu} }  \\
   { G_{ \overline\mu \mu} } & { G_{ \overline \mu \overline \mu} }  \\
\end{pmatrix}^{ - 1}  &=\frac{1}{z-1}\begin{pmatrix}1 & -z^{-1/2} \\ -z^{-1/2} & 1 \end{pmatrix}  -\begin{pmatrix}\sum_{i, j \in \cal I_1}x_{i\mu  }x_{j\mu} G^{[\mu]}_{ij}  & \sum_{i \in \cal I_1, j\in \cal I_2} x_{i\mu}y_{j\overline \mu} G^{[\mu]}_{ij}  \\ \sum_{i \in \cal I_1, j\in \cal I_2}x_{i \mu} y_{j\overline \mu }  G^{[\mu]}_{ji} & \sum_{i, j \in \cal I_2} y_{i\overline\mu }y_{j\overline\mu} G^{[\mu]}_{ij}   \end{pmatrix} \\
& \approx   \begin{pmatrix}(z-1)^{-1}- m_1 & - {z^{-1/2}}/(z-1) \\  - {z^{-1/2}}/(z-1) &   (z-1)^{-1} - m_2 \end{pmatrix} \qquad \text{with high probability}   .
\end{split}
\ee
Taking matrix inverse on both sides, we obtain that with high probability,
\begin{align}
\begin{pmatrix} 
   { G_{\mu\mu} } & { G_{\mu \overline \mu} }  \\
   { G_{ \overline\mu \mu} } & { G_{ \overline \mu \overline \mu} }  \\
\end{pmatrix}  & \approx   \frac{ z-1}{[1-(z-1)m_{1}][ 1-(z-1)m_{2}] -z^{-1}  }\begin{pmatrix}1-(z-1)m_{2} & z^{-1/2} \\  z^{-1/2} &   1-(z-1)m_{1} \end{pmatrix}  \label{derv2}
\end{align}
Using \eqref{derv2}, we obtain that with high probability,
\be\label{derv3}m_3= \frac{1}{n}\sum_{\mu \in \cal I_3}G_{\mu\mu}\approx  \frac{ (z-1)[1-(z-1)m_{2} ]}{[1-(z-1)m_{1}][ 1-(z-1)m_{2}] -z^{-1}  } = \frac{1-(z-1)m_{2}}{z^{-1}-(m_{1}+m_2)+(z-1)m_1 m_{2} },\ee
which leads to equation \eqref{selfm3}.

From the above (non-rigorous) derivation, we have seen that $(m_1,m_2,m_3,m_4)$ satisfies equations \eqref{selfm12} and \eqref{selfm3} approximately. Then we should have $(m_1,m_2,m_3,m_4)\approx (m_{1c},m_{2c},m_{3c},m_{4c})$ with high probability. Moreover, by \eqref{derv1} we have $G_{ii}\approx -m_3^{-1}\approx -m_{3c}^{-1} = c_1^{-1}m_{1c}$ with high probability for $i\in \cal I_1$. Similarly, we can get that $G_{jj}\approx  c_2^{-1}m_{2c}$ with high probability for $j\in \cal I_2$. By \eqref{derv2} and \eqref{derv3} we have $G_{\mu\mu} \approx m_3\approx m_{3c}$ with high probability for $\mu \in \cal I_3$. Similarly, we have $G_{\nu\nu} \approx m_4 \approx m_{4c}$ with high probability for $\nu \in \cal I_4$. Finally, by \eqref{derv2} and \eqref{derv3} we have that for $\mu \in \cal I_3$,
$$G_{\mu \overline \mu} \approx \frac{z^{-1/2}}{1-(z-1)m_{2}}m_{3} \approx h(z) \quad \text{with high probability},$$ 
where we recall $h(z)$ defined in \eqref{hz}. The above arguments explain the nonzero entries in $\Pi(z)$. Using Schur complement formula and concentration estimates, we can also show that all the other entries of $G(z)$ are approximately zero. The reader can refer to Section \ref{pf thmlocal} for more details.

}

\section{Basic notations and tools}\label{secsectools}

In this preliminary section, we introduce some basic notations and tools that will be used in the proof. First, the following lemma collects basic properties of stochastic domination $\prec$, which will be used tacitly in the proof.

\begin{lemma}[Lemma 3.2 in \cite{isotropic}]\label{lem_stodomin}
Let $\xi$ and $\zeta$ be families of nonnegative random variables.
\begin{itemize}
\item[(i)] Suppose that $\xi (u,v)\prec \zeta(u,v)$ uniformly in $u\in U$ and $v\in V$. If $|V|\le n^C$ for some constant $C$, then $\sum_{v\in V} \xi(u,v) \prec \sum_{v\in V} \zeta(u,v)$ uniformly in $u$.

\item[(ii)] If $\xi_1 (u)\prec \zeta_1(u)$ and $\xi_2 (u)\prec \zeta_2(u)$ uniformly in $u\in U$, then $\xi_1(u)\xi_2(u) \prec \zeta_1(u)\zeta_2(u)$ uniformly in $u$.

\item[(iii)] Suppose that $\Psi(u)\ge n^{-C}$ is deterministic and $\xi(u)$ satisfies $\mathbb E\xi(u)^2 \le n^C$ for all $u$. Then if $\xi(u)\prec \Psi(u)$ uniformly in $u$, we have $\mathbb E\xi(u) \prec \Psi(u)$ uniformly in $u$.
\end{itemize}
\end{lemma}

We have the following lemma, which can be verified through direct calculation using \eqref{m1c}-\eqref{m4c}. 

\begin{lemma}\label{lem_mbehavior}
Fix any constants $c, C>0$. If \eqref{assm20} holds, then 
for $z\in \C_+ \cap \{z: c\le |z| \le C\}$ we have
\begin{equation}\label{absmc}
\left|z^{-1} - (m_{1c}(z)+m_{2c}(z)) + (z-1)m_{1c}(z)m_{2c}(z) \right|\sim 1,
 \ee 
 and
 \be\label{Immc} 
  \vert m_{3c}(z) \vert \sim |h(z)| \sim 1,  \quad 0\le \im m_{3c}(z) \sim \begin{cases}
    {\eta}/{\sqrt{\kappa+\eta}}, & \text{ if } E \notin [\lambda_-,\lambda_+] \\
    \sqrt{\kappa+\eta}, & \text{ if } E \in [\lambda_-,\lambda_+]\\
  \end{cases}.
\end{equation}
%
The estimate \eqref{Immc} also holds for $m_{1c}$, $m_{2c}(z)$, $m_{4c}(z)$ and $m_c(z)$. 
\end{lemma}
By \eqref{absmc} and (\ref{Immc}), we have for $z\in S(\e)$ (recall \eqref{eq_defpsi}),
\begin{equation}\label{psi12}
\begin{split}
&\|\Pi\|=\OO(1), \quad \Psi \gtrsim n^{-1/2} , \quad \Psi^2 \lesssim (n\eta)^{-1}, \\
 & \Psi(z) \sim  \sqrt {\frac{\Im \, m_{\al c}(z)}{{n\eta }} } + \frac{1}{n\eta} \ \ \text{with} \ \ \al=1,2,3,4.
 \end{split}
\end{equation}

 Note that $S_{xx}$ (resp. $S_{yy}$) is a standard sample covariance matrix, and it is well-known that its eigenvalues are all inside the support of the Marchenko-Pastur law $[(1-\sqrt{c_1})^2 , (1+\sqrt{c_1})^2]$ (resp. $[(1-\sqrt{c_2})^2 , (1+\sqrt{c_2})^2]$) with probability $1-\oo(1)$ \cite{No_outside}. Hence both $S_{xx}^{-1}$ and $S_{yy}^{-1}$ behaves well under the assumption \eqref{assm20}. In our proof, we shall need a slightly stronger probability bound, which is given by the following lemma. Denote the eigenvalues of $S_{xx}$ and $S_{yy}$ by $\lambda_1 (S_{xx}) \ge \cdots \ge \lambda_p (S_{xx})$ and $\lambda_1 (S_{yy}) \ge \cdots \ge \lambda_q (S_{yy})$.
 
\begin{lemma}\label{SxxSyy}
Suppose Assumption \ref{main_assm} holds. Suppose $X$ and $Y$ have bounded support $\phi_n$ such that $ n^{-{1}/{2}} \leq \phi \leq n^{- c_\phi} $ for some constant $c_\phi>0$. Then for any constant $\e>0$, we have with high probability,
\be\label{op rough1} (1-\sqrt{c_1})^2 - \e \le  \lambda_p(S_{xx})   \le  \lambda_1(S_{xx}) \le (1+\sqrt{c_1})^2 + \e ,
\ee
and
\be\label{op rough2} (1-\sqrt{c_2})^2 - \e \le  \lambda_q(S_{yy})  \le  \lambda_1(S_{yy}) \le (1+\sqrt{c_2})^2 + \e .
\ee
\end{lemma}
\begin{proof}
Note that $X$ can be written as
$$ X= \cal M_1 \odot \wt X + \cal M_2,$$
where $\odot$ denotes the Hadamard product, $\wt X$ is a $p\times n$ random matrices whose entries are independent random variables that satisfy \eqref{assm1}, and {$\cal M_1$ and $\cal M_2$ are $p\times n$ deterministic matrices with $(\cal M_1)_{ij}=\sqrt{n\mathbb E (x_{ij} - \mathbb E x_{ij})^2}=1+\OO(n^{-1/2-\tau/2})$ and $(\cal M_2)_{ij}=\OO(n^{-2-\tau})$}. In particular, we have that
\be\label{op rough1.25}\|\cal M_2\| \le \|\cal M_2\|_{F}= \OO(n^{-1-\tau}).\ee
Moreover, $\wt X$ has bounded support $\OO(\phi_n)$. Then we claim that for any constant $\e>0$, 
\be\label{op rough1.5} 
(1-\sqrt{c_1})^2 - \e \le  \lambda_p(\wt X\wt X^T) \le   \lambda_1(\wt X\wt X^T) \le (1+\sqrt{c_1})^2 + \e 
\ee
with high probability. This result essentially follows from \cite[Theorem 2.10]{isotropic}, although the authors considered the case with $\phi_n \prec n^{-1/2}$ only. The results for more general $\phi_n$ follows from \cite[Lemma 3.12]{DY}, but only the bounds for the largest eigenvalues are given there in order to avoid the issue with the smallest eigenvalue when $c_1$ is close to 1. However, under the assumption \eqref{assm20}, the lower bound for the smallest eigenvalue follows from the exactly the same arguments as in \cite{DY}. Hence we omit the details. Now using \eqref{op rough1.25}, \eqref{op rough1.5} and the estimates on the entries of $\cal M_1$, we conclude \eqref{op rough1}. The estimate \eqref{op rough2} can be proved in the same way. 
\end{proof}

Next we provide a rough bound on the operator norms of the resolvents.
\begin{lemma}\label{lem op1}
For $z=E+\ii \eta\in \C_+$ such that $c\le |z|\le c^{-1}$ for some constant $c>0$, we have
\be\label{op RH1} \left\| R(z)\right\| \le \frac{C}{\eta}, \quad R(z):=\begin{pmatrix} R_1  & - z^{-1/2} R_1\cal H  \\ - z^{-1/2} \cal H^T R_1  &  R_2 \end{pmatrix},
\ee
and
\be\label{op G1} \left\| G(z)\right\| \le \frac{C(1+\|S_{xx}^{-1}\|+\|S_{yy}^{-1}\|)}{\eta},
\ee
for some constant $C>0$.
\end{lemma}
\begin{proof}
Let
 $\cal H = \sum_{k = 1}^{q} \sqrt {\lambda_k} \xi_k \zeta _{k}^T $
be a singular value decomposition of $\cal H$, where
$$\lambda_1\ge  \ldots \ge \lambda_{q} \ge 0 = \lambda_{q+1} = \ldots = \lambda_{p},$$
$\{\xi_{k}\}_{k=1}^{p}$ are the left-singular vectors, and $\{\zeta_{k}\}_{k=1}^{q}$ are the right-singular vectors. Then we have
\begin{equation}
R\left( z \right) = \sum\limits_{k = 1}^q \frac{1}{\lambda_k-z}\left( {\begin{array}{*{20}c}
   {{\xi _k \xi _k^T  }} & {-z^{-1/2}\sqrt {\lambda _k } \xi _k \zeta _{ k}^T}  \\
   {-z^{-1/2} \sqrt {\lambda _k } \zeta _{k} \xi _k^T  } & {\zeta _k\zeta _k^T }  \\
\end{array}} \right) - \frac1z \left( {\begin{array}{*{20}c}
   {\sum_{k=q+1}^p{\xi _k \xi _k^T  }} & 0  \\
   {0 } & {0}  \\
\end{array}} \right). \label{spectral1}
\end{equation}
The estimate \eqref{op RH1} follows immediately from this representation and the fact that $|\lambda_k-z|\ge \eta$. The bound \eqref{op G1} holds for $\cal G_L$ by noticing that
\be\label{simple obs}
\begin{split}
 \cal G_L & = \begin{pmatrix} S_{xx}^{-1/2} & 0 \\ 0 & S_{yy}^{-1/2} \end{pmatrix}R(z) \begin{pmatrix} S_{xx}^{-1/2} & 0 \\ 0 & S_{yy}^{-1/2} \end{pmatrix}.
\end{split}
\end{equation}
For $\cal G_R$, $\cal G_{LR}$ and $\cal G_{RL}$, stronger bounds hold by \eqref{simple obs} and \eqref{GR1}-\eqref{GLR1}:
$$\left\| \cal G_{R}(z)\right\| \le \frac{C}{\eta},\quad \left\| \cal G_{LR}(z)\right\| + \left\| \cal G_{RL}(z)\right\| \le \frac{C(1+\|S_{xx}^{-1/2}\|+\|S_{yy}^{-1/2}\|)}{\eta}  ,$$
where we used $\|S_{xx}^{-1/2} X\| \le 1$ and  $\|S_{yy}^{-1/2} Y\| \le 1$. 
\end{proof}

One subtle point is that in order to apply Lemma \ref{lem_stodomin} (iii) in our proof, we need a bound on the high moments of $\|S_{xx}^{-1}\|$ and $\|S_{yy}^{-1}\|$ (since we will take expectation over the products of many resolvent entries). 
However instead of using such a bound, we shall regularize the resolvents a little bit in the following way.

\begin{definition}[Regularized resolvents]\label{resol_not2}
For $z = E+ \ii \eta \in \mathbb C_+,$ we define the regularized resolvent $\wh G(z)$ as
$$ \wh G(z) := \left[H(z)- z n^{-10} \begin{pmatrix} I_{p+q} & 0 \\ 0 & 0 \end{pmatrix} \right]^{-1} .$$
Then ${\wsG}_L(z)$, ${\wsG}_R(z)$, ${\wsG}_\al(z)$ and $\wh m_\al(z)$, $\al=1,2,3,4$, are defined in the obvious way. We define 
$${\wsH }:=\wh S_{xx} ^{-1/2}S_{xy}\wh S_{yy}^{-1/2}, \quad \wh S_{xx}:= S_{xx}+n^{-10}, \quad \wh S_{yy}:= S_{yy}+n^{-10}.$$
Then $\wh R_1$, $\wh R_2$ and $\wh m(z)$ are defined in the obvious way. We also define $\wh R(z)$ and the spectral decomposition
\be\label{specrtal2}
\begin{split}
&\wh R(z):=\begin{pmatrix}  \wh R_1  & - z^{-1/2} \wh R_1{\wsH }  \\ - z^{-1/2} {\wsH }^T  \wh R_1  &   \wh R_2 \end{pmatrix} \\
& = \sum\limits_{k = 1}^q \frac{1}{ \wh\lambda_k-z}\left( {\begin{array}{*{20}c}
   {{ \wh\xi _k \wh\xi _k^T  }} & {-z^{-1/2}\sqrt { \wh\lambda _k } \wh\xi _k \wh\zeta _{ k}^T}  \\
   {-z^{-1/2} \sqrt { \wh\lambda _k }  \wh\zeta _{k}  \wh\xi _k^T  } & { \wh\zeta _k \wh\zeta _k^T }  \\
\end{array}} \right)  - \frac1z \left( {\begin{array}{*{20}c}
   {\sum_{k=q+1}^p{\wh\xi _k \wh\xi _k^T  }} & 0  \\
   {0 } & {0}  \\
\end{array}} \right).
\end{split}
\ee
\end{definition}

By Schur complement formula, we have
\be\label{GL}
\begin{split}
 {\wsG}_L & = \begin{pmatrix} \wh S_{xx}^{-1/2} \wh R_1\wh S_{xx}^{-1/2} & - z^{-1/2}\wh S_{xx}^{-1/2}\wh R_1 {\wsH }\wh S_{yy}^{-1/2} \\ - z^{-1/2}\wh S_{yy}^{-1/2} {\wsH }^T  \wh R_1\wh S_{xx}^{-1/2} & \wh S_{yy}^{-1/2}\wh R_2 \wh S_{yy}^{-1/2}\end{pmatrix},
\end{split}
\ee
\be\label{GR}
 {\wsG}_R =   \begin{pmatrix}  z  I_n & z^{1/2}I_n\\ z^{1/2}I_n &  z  I_n\end{pmatrix} +   \begin{pmatrix}  z  I_n & z^{1/2}I_n\\ z^{1/2}I_n &  z  I_n\end{pmatrix}  \begin{pmatrix} X^T & 0 \\ 0 & Y^T \end{pmatrix} {\wsG}_L \begin{pmatrix} X & 0 \\ 0 &  Y \end{pmatrix} \begin{pmatrix}  z  I_n & z^{1/2}I_n\\ z^{1/2}I_n &  z  I_n\end{pmatrix}  ,
\ee
and
\be\label{GLR}
\begin{split}
&\wsG_{LR}(z)= -\wsG_L(z) \begin{pmatrix} X & 0 \\ 0 &  Y \end{pmatrix} \begin{pmatrix}  z  I_n & z^{1/2}I_n\\ z^{1/2}I_n &  z  I_n\end{pmatrix}  , \\
 & \wsG_{RL}(z)= -  \begin{pmatrix}  z  I_n & z^{1/2}I_n\\ z^{1/2}I_n &  z  I_n\end{pmatrix}  \begin{pmatrix} X^T & 0 \\ 0 & Y^T \end{pmatrix} \wsG_L(z).
 \end{split}
\ee
With a straightforward calculation and using \eqref{m34}, we obtain that
\be\label{m3-4} 
 \wh m_3(z) -\wh m_4 (z)= (1-z)(c_1-c_2) + z(1-z) n^{-11} \left(\tr \wsG_1(z) - \tr \wsG_2(z)\right). 
\ee
For the regularized resolvents, it is easy to prove the following result using the same argument for the proof of Lemma \ref{lem op1}.

\begin{lemma}\label{lem op}
For $z=E+\ii \eta\in \C_+$ such that $c\le |z|\le c^{-1}$ for some constant $c>0$, \eqref{op RH1} and \eqref{op G1} hold for $\wh G(z)$ and $\wh R(z)$. Moreover, we have 
\be\label{op G}
\left\| \wh G(z)\right\| \le \frac{Cn^{10}}{\eta} , 
\ee 
for some constant $C>0$.
\end{lemma}
 
In the proof, we will take $\eta\gg n^{-1}$, and the deterministic bound \eqref{op G} then justifies the application of Lemma \ref{lem_stodomin} (iii) when we calculate the expectation of polynomials of the entries of $\wh G(z)$. For simplicity of presentation, we will not repeat this again in the proof. 

\begin{remark}\label{removehat}
The results for $\wh G(z)$ can be extended to $G(z)$ with a standard perturbative argument. We will show that there exists a high probability event $\Xi$ on which $ \|\wh G(z)\|_{\max}=\OO(1)$ for $z$ in some bounded regime. Then we define 
$$ G_t(z) := \left[H(z)- t z n^{-10} \begin{pmatrix} I_{p+q} & 0 \\ 0 & 0 \end{pmatrix} \right]^{-1}, \quad G_0(z)\equiv G(z), \quad G_1(z)\equiv \wh  G(z) .$$
Taking the derivative, we get
\be\label{partialtG}\partial_t G_t(z) = zn^{-10} G_t(z) \begin{pmatrix} I_{p+q} & 0 \\ 0 & 0 \end{pmatrix} G_t(z) . \ee
Thus applying the Gronwall's inequality to 
$$ \|G_t(z)\|_{\max}\le \|\wh G(z)\|_{\max} + C n^{-9} \int_t^1 \|G_s(z)\|_{\max}^2  \dd s, $$
we can obtain that $\|G_t(z)\|_{\max}\le C$ for all $0\le t \le 1$ on $\Xi.$ Then using \eqref{partialtG} again, we obtain that $ \|G(z) - \wh G(z)\|_{\max} \le n^{-8}$ on $\Xi$. Such a small error will not affect any of our results. We emphasize that the above argument is purely deterministic on $\Xi$, so we do not lose any probability. 
\end{remark}

 We record the following resolvent estimates, which will be used in the proof of Theorem \ref{thm_local}. 

\begin{lemma}\label{lem_gbound0}
For any deterministic unit ${\bv_\al}  \in \C^{\mathcal I_\al}$, $\al=1,2$, we have for $\beta=1,2,3,4$,
\be 
\sum_{\fa\in \cal I_\beta} |G_{\fa\mathbf v_\al}(z)|^2 = \sum_{\fa\in \cal I_\beta} | G_{{{\bv_\al} } \fa}(z)|^2 \prec  |G_{{{\bv_\al} }{ {\bv_\al} }}(z)|+\frac{ \im G_{{{\bv_\al} }{ {\bv_\al} }}(z) }{\eta},\quad z= E+ \ii\eta .\label{eq_sgsq1}
\ee
For any deterministic unit $\bv_\beta \in \C^{\mathcal I_\beta}$, $\beta=3,4$, we have for $\al=1,2,3,4$,
\be 
\begin{split}
 \sum_{\fa \in \mathcal I_\al } {\left| {G_{\fa \mathbf v_\beta } } \right|^2 }  \prec 1 +  \frac{\im (\cal U \cal G_R)_{{\bv_\beta}{\bv_\beta}}}{\eta} , \quad   \sum_{\fa \in \mathcal I_\al } {\left| {G_{\mathbf v_\beta  \fa } } \right|^2 }  \prec 1  + \frac{\im ( \cal G_R\cal U^T)_{{\bv_\beta}{\bv_\beta}}}{\eta} , \label{eq_sgsq2} 
\end{split}
\ee
where
$$ \cal U := z^{1/2} \begin{pmatrix}  \overline z  I_n &  \overline z^{1/2}I_n\\  \overline z^{1/2}I_n &   \overline z  I_n\end{pmatrix}    \begin{pmatrix}  z  I_n & z^{1/2}I_n\\ z^{1/2}I_n &  z  I_n\end{pmatrix}^{-1}  . $$
Similar estimates hold for $\wh G$.
\end{lemma}
\begin{proof}
First, we prove some simple resolvent estimates on $R(z)R^*(z)$ and $R^*(z)R(z)$ for $R(z)$ in \eqref{op RH1}. Using spectral decomposition \eqref{spectral1}, for any vector $\bw\in \C^p$ and $z=E+\ii\eta$, we have
\be\label{spec app}
\begin{split}
\bw^* R_1^*(z)R_1(z) \bw = \bw^* R_1(z)R_1^*(z) \bw =\sum\limits_{k = 1}^p \frac{|\langle \bw,\xi _k \rangle|^2 }{|\lambda_k-E|^2 +\eta^2} = \frac{\im (R_1(z))_{\bw\bw}}{\eta}.
\end{split}
\ee
For $R_2$, we have a similar estimate. Notice that with Schur complement, we can write
$$R(z) = z^{-1/2}\left( \mathscr H- z^{1/2}\right)^{-1},\quad \mathscr H= \begin{pmatrix} 0 & -\mathcal H \\ -\mathcal H^T & 0\end{pmatrix}.$$
Then using a same argument as in \eqref{spec app}, we obtain that for any vector $\bw \in \C^{p+q}$,
\be\label{spec app2}
\begin{split}
\bw^* R^*(z)R(z) \bw =|z|^{-1}\frac{ \im \left(z^{1/2}\bw^* R(z) \bw \right)}{\im z^{1/2}}.
\end{split}
\ee

We first pick $\al=1$ and $\beta=1,2$. Using \eqref{GL1} and Lemma \ref{SxxSyy}, we get
\begin{align*}
\sum_{\fa\in \cal I_\beta} |G_{\fa{{\bv_\al} }}(z)|^2 &\le \sum_{\fa\in \cal I_1 \cup \cal I_2} |G_{\fa{{\bv_\al} }}(z)|^2 = \left( \cal G_L^* \cal G_L\right)_{{{\bv_\al} }{{\bv_\al} }} \prec   \begin{pmatrix}{\bv_\al^*} S_{xx}^{-1/2} ,0 \end{pmatrix}R^*(z)R(z) \begin{pmatrix} S_{xx}^{-1/2}{\bv_\al}  \\ 0  \end{pmatrix} \\
& = {\bv_\al^*} S_{xx}^{-1/2}\left[ R_1^*(z)R_1(z) + |z|^{-1}R_1^*(z)\left(\cal H\cal H^T \right)R_1(z)\right]S_{xx}^{-1/2}{\bv_\al}  \\
&= \left(1+\frac{\overline z}{|z|}\right){\bv_\al^*} S_{xx}^{-1/2} R_1^*(z)R_1(z)S_{xx}^{-1/2}{\bv_\al}   + |z|^{-1}{\bv_\al^*} S_{xx}^{-1/2} R_1(z)S_{xx}^{-1/2}{\bv_\al}   \\
&  \prec    \frac{\im  G_{{{\bv_\al} }{{\bv_\al} }}}{\eta} +  | G_{{\bv_\al}{\bv_\al} }|,
\end{align*}
where in the third step we used \eqref{simple obs}, in the fourth step $\cal H\cal H^T = (R^*_1(z))^{-1}+ \overline z$, and in the last step \eqref{simple obs} and \eqref{spec app}. For $\al=1$ and $\beta=3,4$, the proof is similar except that we need to use \eqref{GLR1} and 
$$ \left\| \begin{pmatrix} X & 0 \\ 0 &  Y \end{pmatrix} \begin{pmatrix}  z  I_n & z^{1/2}I_n\\ z^{1/2}I_n &  z  I_n\end{pmatrix} \begin{pmatrix}  \overline z  I_n &  \overline z^{1/2}I_n\\  \overline z^{1/2}I_n &   \overline z  I_n\end{pmatrix} \begin{pmatrix} X^T & 0 \\ 0 &  Y^T \end{pmatrix}  \right\| \prec 1 $$
by Lemma \ref{SxxSyy}. For $\al=2$, the proof is the same. This concludes \eqref{eq_sgsq1}.

Then we consider the case $\al=3$ and $\beta=3,4$. Using \eqref{GR1}, \eqref{simple obs}, Lemma \ref{SxxSyy} and \eqref{spec app2}, we get that
\begin{align*}
& \sum_{\fa\in \cal I_\beta} |G_{\fa{\bv_\al} }(z)|^2 \le \sum_{\fa\in \cal I_3 \cup \cal I_4} |G_{\fa{\bv_\al} }(z)|^2 \prec 1 + \bw_\al^*  R^*(z) R(z)   \bw_\al  \lesssim \frac{ \im \left(z^{1/2}\bw^* R(z) \bw \right)}{\eta},
\end{align*}
where 
$$\bw_\al:=  \begin{pmatrix} S_{xx}^{-1/2}X & 0 \\ 0 &  S_{yy}^{-1/2}Y \end{pmatrix} \begin{pmatrix}  z  I_n & z^{1/2}I_n\\ z^{1/2}I_n &  z  I_n\end{pmatrix}{\bv_\al}  .$$
By \eqref{GR1}, we have
\begin{align*}
z^{1/2}\bw^* R(z) \bw ={\bv_\al^*}  \cal U \left[\cal G_R - \begin{pmatrix}  z  I_n & z^{1/2}I_n\\ z^{1/2}I_n &  z  I_n\end{pmatrix}  \right] {\bv_\al}  .
\end{align*}
Then using  
$$\im{\bv_\al^*}  \cal U   \begin{pmatrix}  z  I_n & z^{1/2}I_n\\ z^{1/2}I_n &  z  I_n\end{pmatrix}  {\bv_\al}  =\OO(\eta) ,$$
we conclude \eqref{eq_sgsq2}. 
\end{proof}

The anisotropic local law (\ref{aniso_law}) together with the rigidity estimate \eqref{rigidity} implies the following delocalization properties of eigenvectors. 

\begin{lemma}[Isotropic delocalization of eigenvectors] \label{lem delocalX}
Suppose \eqref{rigidity} hold, and \eqref{aniso_law} holds for $G$. Then for any fixed $\delta>0$ and any deterministic unit vectors $\mathbf u_\al \in \mathbb C^{\mathcal I_\al}$, $\al=1,2,3,4$, the following estimates hold: 
\begin{equation}
  \left|\langle \mathbf u_1,S_{xx}^{-1/2}\xi_k\rangle \right|^2+\left|\langle \mathbf u_2,S_{yy}^{-1/2}\zeta_k\rangle \right|^2 \prec n^{-1},\quad 1 \le k \le  q, \label{delocal1}
\end{equation}
and
\begin{equation}
  \left|\langle \mathbf u_3,X^TS_{xx}^{-1/2}\xi_k\rangle \right|^2+\left|\langle \mathbf u_4,Y^TS_{yy}^{-1/2}\zeta_k\rangle \right|^2 \prec n^{-1},\quad 1 \le k \le  q.\label{delocal2}
\end{equation}
If \eqref{rigidity} only holds for $i\le (1-\e)q$, then \eqref{delocal1} and \eqref{delocal2} hold for $1\le k\le (1-\e)q$.
\end{lemma}

\begin{proof}
Choose $z_0=E+\ii \eta_0 \in S( \epsilon)$ with $\eta_0 = n^{-1+\epsilon}$. By (\ref{aniso_law}) for $G$, we have $\im  \langle \mathbf u_1, G(z_0) \mathbf u_1\rangle = \OO(1) $ with high probability.
Then using \eqref{GL1} and the spectral decomposition (\ref{spectral1}), we get
\begin{equation}\label{spectraldecomp0}
\sum_{k=1}^q \frac{\eta_0 \vert \langle \mathbf u_1, S_{xx}^{-1/2}\xi_k\rangle \vert^2}{(\lambda_k-E)^2+\eta_0^2}  = \im\, \langle \mathbf u_1, {G}(z_0)\mathbf u_1\rangle =  \OO(1)  \quad \text{ with high probability.}
\end{equation}
By (\ref{rigidity}), we have that $\lambda_k + \ii \eta_0 \in S(\epsilon)$ with high probability. Then choosing $E=\lambda_k$ in (\ref{spectraldecomp0}) yields that
\begin{equation*}
\vert \langle \mathbf u_1, S_{xx}^{-1/2}\xi_k\rangle \vert^2 \lesssim  \eta_0 \quad \text{with high probability.}
\end{equation*} 
Since $\epsilon$ is arbitrary, we get $\vert \langle \mathbf u_1, S_{xx}^{-1/2}\xi_k\rangle \vert^2  \prec n^{-1}$. In a similar way, we can prove $\left|\langle \mathbf u_2,S_{yy}^{-1/2}\zeta_k\rangle \right|^2 \prec \eta_0$.

Now for $z_0=\lambda_k+\ii \eta_0\in S( \epsilon)$, we denote
$$\wt\bu_3:= \begin{pmatrix}  z_0  I_n & z_0^{1/2}I_n\\ z_0^{1/2}I_n &  z_0  I_n\end{pmatrix}^{-1}\begin{pmatrix} \bu_3 \\ 0\end{pmatrix}.$$
 Note that $z_0$ is well-separated from $0$ and $1$ by a distance of order 1 by \eqref{rigidity}, so we have $\|\wt\bu_3\|_2=\OO(1)$. By (\ref{aniso_law}), we have $\im \wt\bu_3^T G(z_0) \wt\bu_3 = \OO(1) $ with high probability. Using \eqref{GR1} and the spectral decomposition (\ref{spectral1}), we get that with high probability,
\begin{equation}\nonumber
\eta_0^{-1}\vert \langle \mathbf u_3, X^TS_{xx}^{-1/2}\xi_k\rangle \vert^2 \le \sum_{l=1}^q \frac{\eta_0 \vert \langle \mathbf u_3, X^TS_{xx}^{-1/2}\xi_l\rangle \vert^2}{(\lambda_l-E)^2+\eta_0^2} =  \im  \wt\bu_3^T G(z_0) \wt\bu_3 + \OO(1) .
\end{equation}
This gives $ \left|\langle \mathbf u_3,X^TS_{xx}^{-1/2}\xi_k\rangle \right|^2\prec n^{-1}$. Similarly, we get $\left|\langle \mathbf u_4,Y^TS_{yy}^{-1/2}\zeta_k\rangle \right|^2\prec n^{-1}$.
\end{proof}

The second moment of the error $ \langle \mathbf u, (G(z)-\Pi(z)) \mathbf v\rangle $ in fact satisfies a stronger bound. It will be used in the proof of Theorem \ref{main_thm1}.

\begin{lemma}
\label{thm_largebound}
Suppose the assumptions of Theorem \ref{lem null} hold. Then for any fixed $\epsilon>0$,  we have 
\begin{equation}\label{weak_off}
\mathbb{E} \vert \langle \mathbf u ,  G(z)\mathbf v\rangle - \langle \mathbf u , \Pi(z)\mathbf v\rangle  \vert^2 \prec \Psi^2(z), 
\end{equation}
for any $z\in S(\e)$ (recall \eqref{SSET1}) and deterministic unit vectors $\mathbf u , \mathbf v  \in \mathbb C^{\mathcal I}$. 
\end{lemma}

\section{Proof of Theorem \ref{lem null}, Theorem \ref{thm_largerigidity} and Lemma \ref{thm_largebound} }\label{secpf1}

%

In this section, we prove Theorem \ref{lem null}, Theorem \ref{thm_largerigidity} and Lemma \ref{thm_largebound} using Theorem \ref{thm_local}, whose proof is postponed to  Sections \ref{pf thmlocal}-\ref{sec_aniso}. The following proofs will use a comparison argument developed in \cite{LY} for Wigner matrices, which was later extended to sample covariance matrices \cite{DY} and separable covariance matrices \cite{yang2018}. This argument can be extended to our setting without difficulties, where the only inputs are the linearization in \eqref{linearize_block} and the anisotropic local law, Theorem \ref{thm_local}. Hence we will not give all the details, and only focus on the part that is significantly different from the pervious works. 

Given any random matrices $X$ and $Y$ satisfying the assumptions in Theorem \ref{lem null}, we can construct matrices $\wt X$ and $\wt Y$ that match the first four moments as $X$ and $Y$ but with smaller support $\phi_n \prec n^{-1/2}$, which is the content of the next lemma.

\begin{lemma} [Lemma 5.1 in \cite{LY}]\label{lem_decrease}
Suppose $X$ and $Y$ satisfy the assumptions in Theorem \ref{lem null}. Then there exists another matrix $\wt{X}=(\wt x_{ij})$ and $\wt{Y}=(\wt y_{ij})$, such that $\wt{X}$ and $\wt Y$ satisfy the bounded support condition (\ref{eq_support}) with $\phi_n\prec n^{-1/2}$, and the following moments matching holds: 
\begin{equation}\label{match_moments}
\mathbb Ex_{ij}^k =\mathbb E\wt x_{ij}^k, \quad \mathbb Ey_{ij}^k =\mathbb E\wt y_{ij}^k, \quad k=1,2,3,4.
\end{equation}
\end{lemma}
We can define $\wt H(z)$ and $\wt G(z)$ by replacing $(X,Y)$ with $(\wt X,\wt Y)$. Of course $\wt{\cal G}_L(z)$, $\wt{\cal G}_R(z)$, $ \wt m_\al(z)$, $\wt{\cal H}$, $\wt R_{1,2}$, etc.\;can be defined in the obvious way. 

\begin{proof}[Proof of Lemma \ref{thm_largebound}]
By Theorem \ref{thm_local}, we see that \eqref{weak_off} hold for $\wt G(z)$ using \eqref{psi12}. Thus Lemma \ref{thm_largebound} follows immediately from the following comparison lemma. 
\begin{lemma}\label{comparison lem1}
Let $(X, Y)$ and $(\wt{X},\wt Y)$ be pairs of random matrices defined as above. Suppose Theorem \ref{thm_local} holds for both $G(z)$ and $\wt G(z)$. For any small constant $\e>0$,  we have that
\begin{equation} 
\mathbb{E} \left|  \langle \mathbf u ,  G(z)\mathbf v\rangle - \langle \mathbf u ,  \wt G(z)\mathbf v\rangle \right \vert^2 \prec \Psi^2(z),
\end{equation}
for any $z\in S(\epsilon)$ and deterministic unit vectors $\mathbf u , \mathbf v  \in \mathbb C^{\mathcal I}$, 
\end{lemma}
The proof of this lemma is the same as the one for Lemma 3.7 in \cite[Section 7]{yang2018} and the one for Lemma 3.8 in \cite[Section 6]{LY}. The only inputs are Theorem \ref{thm_local} and the moment matching conditions in \eqref{lem_decrease}. Hence we omit the details. 
\end{proof}

We have a similar comparison lemma for the estimates \eqref{aver_in1} and \eqref{aver_out1}. Notice that \eqref{aver_in} and \eqref{aver_out0} holds for $\wt G(z)$ since
$$ \Psi^2(z) \lesssim (n\eta)^{-1} , \quad \text{and} \quad \Psi^2(z) \lesssim \frac{1}{n(\kappa +\eta)} + \frac{1}{(n\eta)^2\sqrt{\kappa +\eta}} \ \ \text{for} \ \ z\in S_{out}(\e),$$ 
by \eqref{Immc}.

\begin{lemma}\label{comparison lem2}
Let $(X, Y)$ and $(\wt{X},\wt Y)$ be pairs of random matrices defined as above. Fix any small constant $\e>0$. For $z \in   S(\epsilon)$ or $z\in S_{out}(\e)$, if there exist deterministic quantities $J\equiv J(N)$ and $K\equiv K(N)$ such that $J \le n^{-c}$ and $K \le n^{-c}$ for some constant $c>0$, and 
\begin{equation}
 \wt{G}(z)-\Pi =\OO_\prec (J), \quad |\wt m_\al(z) - m_{\al c}(z)| \prec K, \ \ \al=1,2,3,4. \label{KEYBOUNDS}
\end{equation}
Then we have
\begin{equation}
 \vert m_\al(z)-m_{\al c}(z) \vert \prec  \Psi^2(z) + J^2+K , \ \ \al=1,2,3,4. \label{KEYEYEYEY}
\end{equation}
\end{lemma}
\begin{proof} 
The proof is similar to the one for \cite[Lemma 5.4]{LY} or \cite[Lemma 7.1]{yang2018} (the latter is closer to our current setting and we can copy its proof almost verbatim). Hence we omit the details. 
\end{proof}

Now we are ready to give the proof of Theorem \ref{thm_largerigidity} using this lemma. 
\begin{proof}[Proof of Theorem \ref{thm_largerigidity}]
By Theorem \ref{thm_local} for $\wt G$ with $\phi_n=n^{-1/2}$, one can choose $J= \Psi(z)$ and 
\begin{equation*}
K=\frac{1}{n \eta}, \quad \text{ or }  \quad   \frac{1}{n(\kappa +\eta)} + \frac{1}{(n\eta)^2\sqrt{\kappa +\eta}} \ \ \text{for} \ \ z\in S_{out}(\e).
\end{equation*}
Then using \eqref{KEYEYEYEY} and $|m(z)-m_c(z)|\lesssim |1-z|^{-1} |m_3(z)-m_{3c}(z)|$ by \eqref{m3m}, we get \eqref{aver_in} and \eqref{aver_out0}. Note that due to the $|1-z|^{-1}$ factor, we need to stay away from $z=1$, which is the main reason why we need to restrict ourself to the domain $\wt S(\epsilon,\wt \e)$ or $\wt S_{out}(\epsilon,\wt \e)$. The estimate \eqref{Kdist} follows from \eqref{aver_in} through a standard argument, see e.g.\;the proofs for \cite[Theorems 2.12-2.13]{EKYY1}, \cite[Theorem 2.2]{EYY} or \cite[Theorem 3.3]{PY}. 
\end{proof}

Next we give the proof of Theorem \ref{lem null}. We first prove the rigidity result \eqref{rigidity}. 
\begin{proof}[Proof of \eqref{rigidity}]
 Without loss of generality, we only consider the case $\lambda_- \gtrsim 1$ in the proof. For the case with $\lambda_-=\oo(1)$, since we only need to prove a weaker result with $1\le i \le (1-\e)q$, the proof is the same except that we do not need to provide the bound in \eqref{lower1} below.

Using \eqref{Kdist} and the method in \cite{EKYY1,EYY}, we can prove the following rigidity estimate: for any fixed $\delta>0$ and all $n^\delta \le i \le  q-n^\delta $, 
\eqref{rigidity} holds. To obtain this estimate for the largest and smallest $n^\delta$ eigenvalues, we still need to provide the following upper and lower bounds: 
for any constant $\e>0$,
\be\label{upper1}
\lambda_1 \le \lambda_+ + n^{-2/3+\e} \quad \text{with high probability,}
\ee
and 
\be\label{lower1}
\lambda_q \ge \lambda_- - n^{-2/3+\e} \quad \text{with high probability.}
\ee
Given these bounds, the estimate \eqref{Kdist} and the method in \cite{EKYY1,EYY} allow us to conclude \eqref{rigidity} for all $i$. 

First, we claim that for any small constants $c,\e>0$, with high probability,
\be\label{remain1}
\#\{i:\lambda_i \in [\lambda_+ + n^{-2/3+\e},1-c]\} =0, \quad \text{and} \quad \#\{i:\lambda_i \in [c,\lambda_- - n^{-2/3+\e}]\}=0.
\ee
We choose $\eta=n^{-2/3}$ and $ E=\lambda_+ + \kappa \le 1-c$ outside of the spectrum with some $\kappa\ge n^{-2/3+2\e} \gg n^\e \eta$. Then using \eqref{aver_out0}, we get that
\be\label{add1}|\im m(z) - \im m_c(z)| \prec  \frac{1}{n(\kappa +\eta)} + \frac{1}{(n\eta)^2\sqrt{\kappa +\eta}} \lesssim \frac{n^{-\e}}{n\eta} .\ee
On the other hand, if there is an eigenvalue $\lambda_j$ satisfying $|\lambda_j - E| \le \eta$ for some $1\le j \le n$, then 
\be\label{add2}\im m(z) = \frac1q\sum_{i=1}^q \frac{\eta}{|\lambda_i-E|^2 + \eta^2} \gtrsim \frac{1}{n\eta}.\ee
On the other hand, by \eqref{Immc} we have 
$$\im m_c(z) =\OO\left( \frac{\eta}{\sqrt{\kappa + \eta}}\right) = \OO\left( \frac{n^{-\e}}{n\eta}\right).$$
Together with \eqref{add2}, this contradicts \eqref{add1}. Hence we conclude the first estimate in \eqref{remain1} since $\e$ can be arbitrarily chosen. The second estimate in \eqref{remain1} can be proved in  the same way by choosing $E=\lambda_--\kappa$.
 
Then it remains to prove that for a sufficiently small constant $c>0$, with high probability,
\be\label{remain2}
\#\{i:\lambda_i \in [1-c,1]\}=0 \quad \text{and} \quad \#\{i:\lambda_i \in [0,c]\}=0.
\ee
We pick i.i.d. Gaussian $X^G$ and $Y^G$, which are independent of the matrices $X$ and $Y$ we are considering. We denote the eigenvalues of $\cal C_{X^GY^G}$ by $\lambda_1^G \ge \lambda_2^G \ge \cdots \ge \lambda_p^G$. Then with Lemma 1 in Section 8.2 of \cite{CCA2}, we know that $|\lambda_1^G - \lambda_+|\prec n^{-2/3}$, which implies
 \be\label{remainG}
\#\{i:\lambda^{G}_i \in [\lambda_+ + n^{-2/3+\e},1]\}=0 \quad \text{with high probability.}
\ee
Now we define a continuous path of random matrices as 
\be\label{XtYt}X_t: = \sqrt{1-t} X^G + \sqrt{t} X, \quad Y_t: = \sqrt{1-t} Y^G + \sqrt{t} Y, \quad t\in [0,1].\ee
Correspondingly, we define $H_t(z)$ and $G_t(z)$ by replacing $(X,Y)$ with $(X_t, Y_t)$ in the definitions \eqref{linearize_block} and \eqref{eqn_defG}. We denote the eigenvalues of $\cal C_{X_tY_t}$ by $\lambda_i^t$. We claim that with high probability,
\be\label{remain3cont}
\text{for any $1\le i \le q$}, \quad  \lambda_i^t \ \text{ is continuous in $t$ for all}\ t\in [0,1].
\ee
and
\be\label{remain3}
\|G_t(1-c)\|_{\max} \ \text{ is finite for all}\ \ t\in [0,1].
\ee
Recall that with high probability, the eigenvalues of $\cal C_{X_0Y_0}$ are all inside $[0,\lambda_+ + n^{-2/3+\e}]$. Moreover, we claim that if \eqref{remain3} holds, then 
\be\label{add claim1}m_t(1-c)=\frac1q\sum_{i=1}^q \frac{1}{\lambda_i^t-(1-c)} \ \ \ \text{ is finite for all }\  t\in [0,1]. \ee 
In fact, by \eqref{remain3} we have that with high probability,
$$m_{3,t}(1-c):=\frac1{n}\sum_{\mu \in \cal I_3} [G_t(1-c)]_{\mu\mu}  \ \ \ \text{ is finite for all }\  t\in [0,1],$$
which implies \eqref{add claim1} using 
$$m_t(1-c)=\frac{m_{3,t}(1-c)-(1-c_1-c_2)(1-c)}{c_2 c(1-c)}$$
by equation \eqref{m3m}. The claim \eqref{add claim1} means that there is no eigenvalue $\lambda_i^t$ crossing the point $E=1-c$ for all $t\in [0,1]$. Hence using the continuity of eigenvalues in \eqref{remain3cont}, we conclude the first estimate in \eqref{remain2}, which, together with \eqref{remain1}, concludes \eqref{upper1}.

By the definition of $\cal C_{X_tY_t}$, to prove \eqref{remain3cont}, it suffices to prove that with high probability, $(X_tX^T_t)^{-1}$ and $(Y_tY^T_t)^{-1}$ are continuous in $t$ for all $t\in [0,1]$. For this purpose, we only need to show that with high probability,
\be \nonumber
 X_t X_t^T \ \text{ and } \ Y_t Y_t^T \ \ \text{are non-singular for all}\  t\in [0,1].
\ee
We consider discrete times $t_k = kn^{-10}$. Note that $X_t$ satisfies the assumptions of Lemma \ref{SxxSyy}. Hence with a simple union bound we get that there exists a high probability event $\Xi_1$ such that
$$ \mathbf 1(\Xi_1)\min_{0\le k \le n^{10}} \lambda_p(X_{t_k} X_{t_k}^T) \ge \mathbf 1(\Xi_1) \frac12(1-\sqrt{c_1})^2  .$$
Moreover, using the bounded support condition for $X$, we get that there exists a high probability event $\Xi_2$ such that
\be\label{roughXt} \mathbf 1(\Xi_2)\max_{i,\mu}|(X_t)_{i\mu}| \le 1 \ \Rightarrow \ \mathbf 1(\Xi_2)\sup_{0 \le t \le 1} \|X_t \| =\OO(n).\ee
This implies 
$$ \sup_{t_{k-1} \le t \le t_k}\|X_{t} X_{t}^T-X_{t_k} X_{t_k}^T\| \lesssim n^{-5}\cdot n^2 =n^{-3}$$
Therefore, on the event $\Xi_1\cap \Xi_2$ we have
$$\inf_{0\le t\le 1}\lambda_p(X_t X_t^T) = \min_{1\le k \le n^{10}} \inf_{t_{k-1} \le t \le t_k} \lambda_p(X_t X_t^T) \ge \frac12(1-\sqrt{c_1})^2 - \OO\left(n^{-3}\right) \gtrsim 1. $$
We have a similar estimate for $Y_tY_t^T$. This concludes \eqref{remain3cont}.

To prove \eqref{remain3}, we consider discrete times $t_k = kn^{-100}$. Note that $X_t$ and $Y_t$ satisfy the assumptions of Theorem \ref{thm_local}, hence the local law \eqref{aniso_law} holds for any $t\in [0,1]$. 
We claim that there exists a high probability event $\Xi$, on which
\be\label{remain4} 
\max_{0\le k \le n^{100}}\| \wh G_{t_k}(1-c+ \ii n^{-10})\|_{\max} =\OO(1), 
\ee
where we recall that $\wh G$ is defined in Definition \ref{resol_not2}. Now suppose \eqref{remain4} holds. 
With the deterministic bound \eqref{op G} and \eqref{roughXt}, we have that for any $t_{k-1}\le t \le t_k$, 
\begin{align*} 
&\left|\wh G_t(1-c+ \ii n^{-10}) - \wh G_{t_k}(1-c+ \ii n^{-10})\right| \\
 &\le Cn^{-50}\|\wh G_t(1-c+ \ii n^{-10})\| \left( \|X\|+\|X^G\|\right) \|\wh G_{t_k}(1-c+ \ii n^{-10})\| \\
& \le n^{-50}\cdot \left( {Cn^{20}} \right)^2\cdot n = \OO(n^{-9}), \quad \text{on $\Xi_2$.}
\end{align*}
Thus we conclude that on $\Xi\cap \Xi_2$, 
$$\max_{0\le t\le 1} \|\wh G_{t}(1-c+ \ii n^{-10})\|_{\max} \le C.$$ 
Finally, the perturbation argument in Remark \ref{removehat} allows us to remove the $\ii n^{-10}$ and the regularization in $\wh G$, which gives \eqref{remain3} on $\Xi$.  

It remains to prove \eqref{remain4}. Since $X_t$ and $Y_t$ also satisfy the assumptions of Theorem \ref{lem null}, by \eqref{remain1} we know that for any fixed $t$, the eigenvalues $\lambda_i^t$ are either inside $[0,\lambda_+ + n^{-2/3+\e}]$ or $[1-c/2,1]$  with high probability. With a simple union bound, we obtain that
$$\min_{0\le k \le n^{100}}\min_{1\le i \le p}|(1-c)-\lambda_i^{t_k}| \gtrsim 1 \quad \text{with high probability}.$$
Together with \eqref{spectral1}, this immediately gives that
\begin{align}\nonumber
\max_{0\le k \le n^{100}}\left\| R_{t_k}(z) \right\| \le C, \quad z= 1-c+ \ii n^{-10}.
\end{align}
Combining this bound with \eqref{GL1}-\eqref{GLR1} and Lemma \ref{SxxSyy}, we get
\begin{align*} 
\max_{0\le k \le n^{100}}\left\| G_{t_k}(z) \right\| \le C, \quad z= 1-c+ \ii n^{-10}.
\end{align*}
 Finally, applying the arguments in Remark \ref{removehat} gives \eqref{remain4} for $\wh G$. This concludes \eqref{remain3}, which further gives the first estimate in \eqref{remain2}.

Finally, the second estimate in \eqref{remain2} can be proved in the same way using the continuous interpolation in \eqref{XtYt}, except that we still need to provide a similar estimate as in \eqref{remainG} for the smallest eigenvalues in the Gaussian case: there exists a constant $c>0$ such that 
\be\label{remainsmallG}
\#\{i:\lambda^{G}_i \in [0,c]\}=0 \quad \text{with high probability.}
\ee
In fact, it is known that the eigenvalues of $\cal C_{X^GY^G}$ reduce to those of the double Wishart matrices \cite{CCA_TW}, that is, the eigenvalues of 
$(\cal W_1+\cal W_2)^{-1}\cal W_1$, where $\cal W_1\sim W_q(p,I)$ (i.e. $\cal W_1$ is a $q\times q$ Wishart matrix with $p$ samples) and $\cal W_1\sim W_q(n-p,I)$ (i.e. $\cal W_2$ is a $q\times q$ Wishart matrix with $(n-p)$ samples). Note that we have $1-q/p\gtrsim 1$ and $1-q/(n-p)\gtrsim 1$ under \eqref{assm20} and the assumption that $\lambda_-\gtrsim 1$. Hence by Lemma \ref{SxxSyy}, we have
$$ \lambda^{G}_q \gtrsim \frac{\lambda_q(\cal W_1)}{\lambda_1(\cal W_1)+\lambda_1(\cal W_2)}\gtrsim 1 \quad \text{with high probability.}$$
This gives \eqref{remainsmallG}, which further concludes the second estimate in \eqref{remain2}.
\end{proof}

 Finally we prove \eqref{joint TW}, which will conclude Theorem \ref{lem null}. 
\begin{proof}[Proof of \eqref{joint TW}]
The proof is similar to the one for \cite[Theorem 3.16]{DY}, so we only outline the argument. For the matrices $\wt{X}$ and $\wt Y$ constructed in Lemma \ref{lem_decrease}, the Tracy-Widom limit of $\cal C_{\wt X\wt Y}$ has been proved in \cite{CCA2}. 

\begin{lemma}
\label{lem_smallcomp}
Let $X$ and $  Y$ be  random matrices satisfying the assumptions in Theorem \ref{lem null} and the bounded support condition with $\phi_n \prec n^{-1/2}$. Then \eqref{joint TW} holds.
\end{lemma}
\begin{proof}
The Tracy-Widom law in \eqref{joint TW} was proved as Theorem 2.1 in \cite{CCA2} under a slight stronger assumption that all the moments of the entries $\sqrt{n}x_{ij}$ and $\sqrt{n}y_{ij}$ are finite, that is, for any fixed $k\in \N$, there exists a constant $C_k$ such that  
\be\label{stronger moment1}\max_{i,j} \E |x_{ij}|^{k} \le C_k n^{-k/2},\quad \max_{i,j} \E |y_{ij}|^{k} \le C_k n^{-k/2}. \ee
On the other hand, under the bounded support condition with $\phi_n \prec n^{-1/2}$, we only have 
\be\label{stronger moment2}\max_{i,j} \E |x_{ij}|^{k} \prec n^{-k/2},\quad \max_{i,j} \E |y_{ij}|^{k} \prec n^{-k/2}. \ee
However, the proof in \cite{CCA2} can be repeated verbatim by replacing \eqref{stronger moment1} with \eqref{stronger moment2} at various places to conclude Lemma \ref{lem_smallcomp}. We omit the details.
\end{proof}

Now it is easy to see that \eqref{joint TW} in the general case follows from the following comparison lemma.

\begin{lemma}\label{eq_edgeuniv} 
Let $(X, Y)$ and $(\wt{X},\wt Y)$ be two pairs of random matrices as in Lemma \ref{lem_decrease}. Then for any fixed $k$, there exist constants $\epsilon,\delta >0$ such that, 
for all $s_1 , s_2, \ldots, s_k \in \mathbb R$, we have
\begin{equation}\label{edgeXX}
\begin{split}
\wt{\mathbb{P}}  \left(\left(n^{\frac{2}{3}}(\lambda_i-\lambda_+)\leq s_i -n^{-\epsilon}\right)_{1\le i \le k}\right) -n^{-\delta}  \leq \mathbb{P} \left(\left(n^{\frac{2}{3}}(\lambda_i-\lambda_+)\leq s_i \right)_{1\le i \le k} \right) \\
\leq \wt{\mathbb{P}}  \left(\left(n^{\frac{2}{3}}(\lambda_i-\lambda_+)\leq s_i + n^{-\epsilon}\right)_{1\le i \le k}\right) +n^{-\delta} ,
\end{split}
\end{equation}
where $\mathbb{P}$ and $\wt{\mathbb{P}} $ denote the laws for $(X, Y)$ and $(\wt{X},\wt Y)$, respectively.
\end{lemma}

To prove Lemma \ref{eq_edgeuniv}, it suffices to prove the following Green's function comparison result. Its proof is the same as the ones for \cite[Lemma 5.5]{LY} and \cite[Lemma 5.5]{DY}, so we omit the details. 

\begin{lemma}
\label{lem_compdiffsupport} 
Let $(X, Y)$ and $(\wt{X},\wt Y)$ be two pairs of random matrices as in Lemma \ref{lem_decrease}. Suppose $F:\mathbb R \to \mathbb R$ is a function whose derivatives satisfy
\begin{equation}
\sup_{x} \vert {F^{(k)}(x)}\vert {(1+\vert x \vert)^{-C_1}} \leq C_1, \quad k=1,2,3, \label{FCondition}
\end{equation}
for some constant $C_1>0$. Then for any sufficiently small constant $\delta>0 $ and for any 
\begin{equation}\label{Eeta}
E,E_1,E_2 \in I_{\delta}:=\left\{x: \vert x -\lambda_{+}\vert \leq n^{-{2}/{3}+\delta}\right\} \ \ \text{and} \ \ \eta:=n^{-{2}/{3}-\delta},
\end{equation}
we have
\begin{equation}
\left| \mathbb{E}F\left(n\eta \im m(z)\right)-\mathbb{E}F\left(n\eta \im \wt{m} (z)\right) \right| \leq n^{-c_\phi + C_2 \delta} , \ \ z=E+\ii\eta,  \label{BDBD}
\end{equation}
and
\begin{align}
\left| \mathbb{E}F\left(n \int_{E_2}^{E_1} \operatorname{Im} m(y+\ii\eta)dy\right)- \mathbb{E} F\left(n \int_{E_2}^{E_1} \operatorname{Im} \wt{m}(y+\ii\eta)dy\right) \right| \leq n^{-c_\phi + C_2 \delta} , \label{BDBD1}
\end{align}
where $c_\phi $ is a constant as given in Theorem \ref{lem null} and $C_2>0$ is some constant independent of $c_\phi$ and $\delta$. Moreover, a general multivariate comparison estimate as in \cite[Theorem 6.4]{EYY} holds: fix any $k\in \N$ and let $F:\R^k\to \R$ be a bounded smooth function with bounded derivatives, then for any sequence of real numbers $E_k < \cdots < E_1 < E_0$ satisfying \eqref{Eeta}, we have
\be
\begin{split}
\left| \mathbb{E}F\left( \left(n \int_{E_k}^{E_0} \operatorname{Im} m(y+\ii\eta)dy\right)_{1\le i \le k}\right)- \mathbb{E}F\left( \left(n \int_{E_k}^{E_0} \operatorname{Im} \wt m(y+\ii\eta)dy\right)_{1\le i \le k}\right) \right| \\
\leq n^{-c_\phi + C_2 \delta} . \label{BDBDmulti}
\end{split}
\ee
\end{lemma}

\begin{proof}[Proof of Lemma \ref{eq_edgeuniv}]
Although not explicitly stated, it was shown in \cite{EYY} that if \eqref{aver_in}, \eqref{rigidity} and Lemma \ref{lem_compdiffsupport} hold, then the edge universality \eqref{edgeXX} holds. More precisely, in Section 6 of \cite{EYY}, the edge universality problem was reduced to proving Theorem 6.3 of \cite{EYY}, which corresponds to our Lemma \ref{lem_compdiffsupport}. In order for this conversion to work, only the the averaged local law and the rigidity of eigenvalues are used, which correspond to  \eqref{aver_in} and \eqref{rigidity}, respectively.
\end{proof}

Finally, \eqref{joint TW} follows immediately from Lemma \ref{lem_decrease}, Lemma \ref{lem_smallcomp} and Lemma \ref{eq_edgeuniv}. 
\end{proof}

\section{Proof of Theorem \ref{main_thm1}}\label{secpf2}
 In this section, we prove Theorem  \ref{main_thm1}. The proof is an extension of the one for Theorem 2.7 in \cite{yang2018}. 
Given the matrices $X$ and $Y$ satisfying Assumption \ref{assm1} and the tail condition (\ref{tail_cond}), we introduce a cutoff on their matrix entries at the level $n^{-\epsilon}$. For any fixed $\epsilon>0$, define
\begin{equation*}
\alpha^{(1)}_n:=\mathbb{P}\left( \vert \wh x_{11} \vert > n^{1/2-\epsilon}\right), \ \ \beta^{(1)}_n:=\mathbb{E}\left[\mathbf{1}{\left( |\wh x_{11}|> n^{1/2-\epsilon}\right)}\wh x_{11} \right].
\end{equation*}
By (\ref{tail_cond}) and integration by parts, we get that for any fixed $\delta>0$ and large enough $n$,
\begin{equation}
\alpha^{(1)}_n \leq \delta n^{-2+4\epsilon}, \ \ \vert \beta^{(1)}_n \vert \leq \delta n^{-{3}/{2}+3 \epsilon} . \label{BBBOUNDS}
\end{equation}
 Let $\rho^{(1)}(\dd x)$ be the law of $\wh x_{11}$. Then we define independent random variables $\wh x_{ij}^s$, $\wh x_{ij}^l$, $c^{(1)}_{ij}$, $1\le i \le p, 1\le j \le n$, in the following ways.
\begin{itemize}
\item $\wh x_{ij}^s$ has law $\rho_s$, which is defined such that
\begin{equation}
 \rho_s^{(1)}(\Omega)= \frac1{1-\alpha^{(1)}_n}\int \mathbf 1\left( x+\frac{\beta^{(1)}_n}{1-\alpha^{(1)}_n} \in \Omega \right)\mathbf{1}\left(\left| x \right| \leq n^{{1}/{2}-\epsilon} \right)  \rho^{(1)}(\dd x) \nonumber
\end{equation}
for any event $\Omega$. Note that if $\wh x_{11}$ has density $\rho(x)$, then the density for $\wh x_{11}^s$ is 
\begin{equation}
\rho_s^{(1)}(x)= \mathbf{1}\left(\left| x-\frac{\beta^{(1)}_n}{1-\alpha^{(1)}_n}  \right| \leq n^{{1}/{2}-\epsilon} \right) \frac{\rho\left(x-\frac{\beta^{(1)}_n}{1-\alpha^{(1)}_n}\right)}{1-\alpha^{(1)}_n}. \nonumber
\end{equation}
\item $\wh x_{ij}^l$ has law $\rho_l$, such that 
\begin{equation}
 \rho_l^{(1)}(\Omega)= \frac1{\alpha^{(1)}_n}\int \mathbf 1\left( x+\frac{\beta^{(1)}_n}{1-\al^{(1)}_n} \in\Omega \right)\mathbf{1}\left(\left| x \right| > n^{{1}/{2}-\epsilon} \right)  \rho^{(1)}(\dd x) \nonumber
\end{equation}
for any event $\Omega$.

\item $c^{(1)}_{ij}$ is a Bernoulli 0-1 random variable with $\mathbb{P}(c^{(1)}_{ij}=1)=\alpha^{(1)}_n$ and $\mathbb{P}(c^{(1)}_{ij}=0)=1-\alpha^{(1)}_n$.
\end{itemize}
Let $X^s$, $X^l$ and $X^c$ be random matrices such that $x^s_{ij} = n^{-1/2}\wh x_{ij}^s$, $x^l_{ij} = n^{-1/2}\wh x_{ij}^l$ and $x^c_{ij} = c^{(1)}_{ij}$.
It is easy to check that for independent $X^s$, $X^l$ and $X^c$,
\begin{equation}
x_{ij} \stackrel{d}{=} x^s_{ij}\left(1-x^c_{ij}\right)+x^l_{ij}x^c_{ij} - \frac{1}{\sqrt{n}}\frac{\beta^{(1)}_n}{1-\alpha^{(1)}_n}. \label{T3}
\end{equation}
We have a similar decompostion for $Y$:
\begin{equation}
y_{ij} \stackrel{d}{=} y^s_{ij}\left(1-y^c_{ij}\right)+y^l_{ij}y^c_{ij} - \frac{1}{\sqrt{n}}\frac{\beta^{(2)}_n}{1-\alpha^{(2)}_n},  \label{T33}
\end{equation}
where the relevant terms are defined in the obvious way using
\begin{equation*}
\alpha^{(2)}_n:=\mathbb{P}\left( \vert \wh y_{11} \vert > n^{1/2-\epsilon}\right), \ \ \beta^{(2)}_n:=\mathbb{E}\left[\mathbf{1}{\left( |\wh y_{11}|> n^{1/2-\epsilon}\right)}\wh y_{11} \right].
\end{equation*}
Notice that the deterministic matrices consist of the constant terms in \eqref{T3} or \eqref{T33} have operator norms  $\OO(n^{-1+3\e})$, which perturb the eigenvalues at most by $\OO(n^{-1+3\e})$. Such a small error is negligible for our result, and hence we will omit the constant terms in \eqref{T3} or \eqref{T33} throughout the proof.

By (\ref{tail_cond}) and integration by parts, it is easy to check that
\begin{align}\label{estimate_qs}
\mathbb{E} \wh x^{s}_{11} =0, \quad \mathbb{E}\vert \wh x^s_{11}\vert^2=1 - \OO(n^{-1+2 \epsilon}), \quad \mathbb{E}\vert \wh x^s_{11}\vert^3= \OO(n^{-1/2+ \epsilon}), \quad \mathbb{E}\vert \wh x^s_{11} \vert^4=\OO(\log n).
\end{align}
We have similar estimates for the $y^s_{11}$ variable. Thus $X_1:=(\mathbb{E}\vert \wh x^s_{11} \vert^2)^{-{1}/{2}}X^s$ and $Y_1:=(\mathbb{E}\vert \wh y^s_{11} \vert^2)^{-{1}/{2}}Y^s$ are random matrices that satisfy the assumptions for $X$ and $Y$ in Theorem \ref{lem null} with $\phi_n =\OO(n^{-\e})$. Again, the $\OO(n^{-1+2 \epsilon})$ in the denominator  can be neglected.

We define the sample canonical correlation matrix $\cal C^s_{ XY}$ by replacing $(X,Y,Z)$ with $(X^s,Y^s,Z^s)$ in the definition, and let $\lambda_i^s$ be its eigenvalues. Then by Theorem \ref{lem null},
\be\label{univ_small}
\begin{split}
\lim_{n\to \infty}\mathbb{P}&\left(  n^{{2}/{3}}\frac{ \lambda^s_{1} - \lambda_+}{c_{TW}} \leq s_1 \right) = \lim_{n\to \infty} \mathbb{P}^{GOE}\left( n^{{2}/{3}}(\lambda_1 - 2) \leq s_1  \right).
\end{split}
\ee
Here and throughout the following proof, we only consider the largest eigenvalue. It is easy to extend to the case with multiple largest eigenvalues.
Now we write the first two terms on the right-hand side of (\ref{T3}) as
$$x^s_{ij}\left(1-x^c_{ij}\right)+x^l_{ij}x^c_{ij}  = x^s_{ij} + \Delta^{(1)}_{ij} x^c_{ij}, \quad \Delta^{(1)}_{ij}:=x^l_{ij}-x^s_{ij}.$$
Similarly, we have
$$y^s_{ij}\left(1-y^c_{ij}\right)+y^l_{ij}y^c_{ij}  = y^s_{ij} + \Delta^{(1)}_{ij} y^c_{ij}, \quad \Delta^{(2)}_{ij}:=y^l_{ij}-y^s_{ij}.$$
We define the matrices $\cal E^{(1)} :=(\Delta^{(1)}_{ij} x^c_{ij})$ and $\cal E^{(2)} :=(\Delta^{(2)}_{ij} y^c_{ij})$.  It remains to show that the effect of $\cal E^{(1)}$ and $\cal E^{(2)}$ on the eigenvalue  $\lambda_1$ is negligible. 

We introduce the following event  
\begin{align*}
\mathscr A:=& \left\{\#\{(i,j):x^c_{ij}=1\}\leq n^{5\epsilon}\right\} \cap \left\{x^c_{ij}=x^c_{kl}=1  {\Rightarrow} (i,j)=(k,l) \ \text{or} \  \{i,j\} \cap \{k,l\}=\emptyset \right\}.
\end{align*}
Using Bernstein inequality, we have that
\begin{equation}\label{LDP_B}
\mathbb{P}\left(\left\{\#\{(i,j):x^c_{ij}=1\}\leq n^{5\epsilon}\right\} \right) \geq 1- \exp(- n^{\epsilon}),
\end{equation}
for sufficiently large $n$. Suppose the number $n_0$ of the nonzero elements in $X^c$ is given with $n_0 \le n^{5\epsilon} $. Then it is easy to check that
\be
\begin{split}\label{LDP_C}
\mathbb P\left(\exists \, i = k, j\ne l \ \text{or} \  i\ne k, j =l \text{ such that } x^c_{ij}=x^c_{kl}=1 \left| \#\{(i,j):x^c_{ij}=1\} = n_0 \right. \right) \\
= \OO(n_0^2n^{-1}). 
\end{split}
\ee
Combining the estimates (\ref{LDP_B}) and (\ref{LDP_C}), we get that 
\begin{equation}\label{prob_A}
\mathbb P(\mathscr A) \ge 1 -  \OO(n^{-1+10\epsilon}).
\end{equation}
On the other hand, by condition (\ref{tail_cond}), we have
\begin{equation}\nonumber
\mathbb{P}\left(|\cal E^{(1)}_{ij}| \geq \omega \right) \leq \mathbb{P}\left(|\wh x_{ij}| \geq \frac{\omega}{2}n^{1/2}\right) = \oo(n^{-2}) ,
\end{equation}
for any fixed constant $\omega>0$. With a simple union bound, we get
\begin{equation}\label{EneR}
\mathbb P\left(\max_{i,j} |\cal E^{(1)}_{ij}| \geq \omega \right)= \oo(1).
\end{equation}
Similarly, we can define the event
$$ \mathscr B:= \left\{\#\{(i,j):y^c_{ij}=1\}\leq n^{5\epsilon}\right\} \cap \left\{y^c_{ij}=y^c_{kl}=1  {\Rightarrow}  (i,j)=(k,l) \ \text{or} \  \{i,j\} \cap \{k,l\}=\emptyset \right\}.$$
By (\ref{prob_A}), (\ref{EneR}) and similar estimates for matrix $Y$, we get
\be\label{prob_R}\mathbb P( \mathscr A \cap  \mathscr B)=1-\oo(1),\quad \mathbb P( \mathscr C_1)=1-\oo(1),\ee
where
$$ \mathscr C_1 := \left\{\max_{i,j}|\cal E^{(1)}_{ij}| \leq \omega\right\}\cap \left\{ \max_{i,j}|\cal E^{(2)}_{ij}| \leq \omega\right\}.$$

Then recall \eqref{deteq}, we only need to study the determinant of $ H_1 (\lambda)$ on event $\mathscr A \cap  \mathscr B \cap \mathscr C_1 $, where we define $ H_t (\lambda)$, $t\in [0,1]$, as
\be\nonumber
\begin{split} 
 H_t (\lambda): = H^s(\lambda)&+ t \begin{pmatrix} 0 & \begin{pmatrix} {\cal E}^{(1)}   & 0\\ 0 &  {\cal E}^{(2)}   \end{pmatrix}\\ \begin{pmatrix} ( {\cal E}^{(1)} )^T & 0\\ 0 &  ( {\cal E}^{(3)} )^T\end{pmatrix}  & 0\end{pmatrix} ,
\end{split}
\ee
where 
$$H^s(\lambda):=\begin{pmatrix} 0 & \begin{pmatrix}  X^s  & 0\\ 0 &  Y^s \end{pmatrix}\\ \begin{pmatrix} (X^s)^T & 0\\ 0 &  (Y^s)^T\end{pmatrix}  & \begin{pmatrix}  \lambda  I_n & \lambda^{1/2}I_n\\ \lambda^{1/2} I_n &  \lambda I_n\end{pmatrix}^{-1}\end{pmatrix}$$
We would like to use a continuity argument to extend \eqref{univ_small} in the $t=0$ case all the way to the $t=1$ case. 
It is easy to observe that with probability $1-\oo(1)$, the eigenvalues $\lambda_1^t \equiv \lambda_1(t)$  is continuous in $t$ for all $t\in [0,1].$ In fact, on even $\mathscr A \cap  \mathscr B \cap \mathscr C_1 $, we have
\be\label{opEE}\|{\cal E}^{(1)}\| \le \omega, \quad \|{\cal E}^{(2)}\| \le \omega.\ee
Hence with Lemma \ref{SxxSyy}, as long as $\omega$ is chosen sufficiently small, $[(X^s + t{\cal E}^{(1)})(X^s + t{\cal E}^{(1)})^T]^{-1}$ and $[(Y^s + t{\cal E}^{(2)})(Y^s + t{\cal E}^{(2)})^T]^{-1}$ will be continuous in $t$ on a high probability event, which implies the continuity of eigenvalues. 
Now we claim that for $\mu:= \lambda_{1}(0) \pm n^{-3/4} \equiv \lambda_{1}^s \pm n^{-3/4}$,
\begin{equation}\label{suff_claim}
{\mathbb P\left(  \det  H_t( \mu) \ne 0 \text{ for all }0\le t \le 1 \right) = 1-\oo(1)} \ .
\ee
Suppose \eqref{suff_claim} holds true, then by continuity $\lambda_{1}\equiv \lambda_{1}(t=1) \in [ \lambda^s_{ 1} - n^{-3/4}, \lambda^s_{1} + n^{-3/4}]$ with probability $1-\oo(1)$, which concludes the proof together with \eqref{univ_small}.

The rest of the proof is devoted to proving \eqref{suff_claim}. In the following proof, we condition on the event $\mathscr A \cap  \mathscr B$ and the event $\mathscr C_{n_x n_y}$ that $X^c$ and $Y^c$ have $n_x$ and $n_y$ nonzero entries for some $\max\{n_x,n_y\}\le n^{5\e}$. Without loss of generality, we can assume the positions of the $n_x$ nonzero entries of $X^c$ are $( 1,1),  \cdots, ( n_x, n_x)$, and the positions of the $n_y$ nonzero entries of $Y^c$ are $( 1,1),  \cdots, ( n_y, n_y)$, that is, we also condition on these two event. For other choices of the positions of nonzero entries, the proof is the same. 
Then we rewrite 
$$\wt H_t(\mu)= H^s(\mu) + t O \begin{pmatrix} 0 &  \cal D_e  \\  \cal D_e  & 0\end{pmatrix} O^T, \quad O:=  \begin{pmatrix} \mathbf F_1  & 0  \\ 0 &  \mathbf F_2  \end{pmatrix}  ,$$ 
where  
$$\cal D_e: =\begin{pmatrix} \Sigma_e^{(1)} & 0 \\ 0 & \Sigma_e^{(2)} \end{pmatrix}, \quad \Sigma_e^{(1)}:= \diag \left( \cal E^{(1)}_{11}, \cdots, \cal E^{(1)}_{ n_x n_x}\right), \quad \Sigma_e^{(2)}:= \diag \left( \cal E^{(2)}_{11}, \cdots, \cal E^{(2)}_{n_y n_y}\right),$$
and
$$\mathbf F_1 := \begin{pmatrix} \begin{pmatrix}\mathbf e_{1}^{(p)}, \cdots, \mathbf e_{ n_x }^{(p)} \end{pmatrix} & 0 \\ 0 & \begin{pmatrix}\mathbf e_{ 1}^{(q)}, \cdots, \mathbf e_{ n_y}^{(q)} \end{pmatrix} \end{pmatrix}, \quad \mathbf F_2 := \begin{pmatrix} \begin{pmatrix}\mathbf e_{ 1}^{(n)}, \cdots, \mathbf e_{ n_x}^{(n)} \end{pmatrix} & 0 \\ 0 & \begin{pmatrix}\mathbf e_{ 1}^{(n)}, \cdots, \mathbf e_{ n_y}^{(n)}\end{pmatrix} \end{pmatrix}.$$ 
Here $\mathbf e_i^{(l)}$ means the standard unit vector along $i$-th coordinate direction in $\R^l$. 
 
Applying the identity $\det(1+\cal A\cal B)=\det(1+\cal B\cal A)$, we obtain that if $\mu$ is such that $\det G^s(\mu) \ne 0$, then
\be\label{nece extra1}
\det H_t(\mu) =\det  G^s(\mu) \cdot\det\left(1+ t F(\mu)  \right), \quad  F(\mu):=t \begin{pmatrix} 0 &  \cal D_e  \\  \cal D_e   & 0\end{pmatrix}  O^T G^s(\mu) O. 
\ee
In the following proof, we use $z=\lambda_+ + \ii n^{-{2}/{3}}$. Then we can write 
\be\label{G-Pi0} O^T G^s(\mu) O = O^T \left[G^s(\mu) - G^s(z)\right] O + O^T [G^s(z)-\Pi(z)] O+ O^T  \Pi(z) O  . \ee
By Lemma \ref{thm_largebound}, we have that
$$\mathbb E  \left| \left[O^T \left( G^s(z) - \Pi(z)\right) O\right]_{ij}^2 \right]\prec \Psi^2(z) = \OO(n^{-2/3}), \quad 1\le i,j \le  n_x + n_y,$$
where we used (\ref{Immc}) and \eqref{eq_defpsi} in the second step. Then with Markov's inequality and a union bound, we can get that
\begin{equation}\nonumber
\max_{1\le i,j \le n_x+n_y} \left| \left[O^T \left( G^s(z) - \Pi(z)\right) O\right]_{ij} \right| \le n^{-1/6}
\end{equation}
holds with probability $1- \OO(n^{-1/3+5\e})$. In particular, this gives that with probability $1- \OO(n^{-1/3+5\e})$,
\be\label{G-Pi1} \|O^T [G^s(z)-\Pi(z)] O\| \lesssim n^{-1/6+5\e}.\ee
On the other hand, we claim that 
\be\label{final claim4}
\|O^T \left[G^s(\mu) - G^s(z)\right] O \| \le n^{-1/6} \ \ \ \text{with probability $1-\oo(1)$}.
\ee
If \eqref{final claim4} holds, together with \eqref{G-Pi0} and \eqref{G-Pi1}, we get that with probability $1-\oo(1)$, 
$$\|O^T G^s(\mu) O\| \le \|\Pi(z)\|+ \OO(n^{-1/6+5\e}) \le 2\|\Pi(z)\| \  \Rightarrow \ \max_{0\le t \le 1} t\|F(\mu)\| \le 2\omega\|\Pi(z)\|\le \frac12 ,$$
as long as $\omega$ is chosen small enough. Hence we have with probability $1-\oo(1)$,  $1+t F(\mu) $ is non-singular for all $t\in[0,1],$ which concludes \eqref{suff_claim}.


Finally it remains to prove \eqref{final claim4}. Since the largest eigenvalues for GOE are separated in the scale $n^{-2/3}$, by \eqref{joint TW} we have that
\begin{equation}\label{repulsion_estimate}
\mathbb P\left( \min_{i}|\lambda_i^s  - \mu| \ge n^{-3/4} \right) = 1-\oo(1).
\end{equation}
On the other hand, the rigidity result (\ref{rigidity}) gives that 
\begin{equation}\label{rigid_estimate}
|\mu - \lambda_+ | \prec n^{-2/3} .
\end{equation}
Then using Lemma \ref{SxxSyy}, Lemma \ref{lem delocalX}, (\ref{repulsion_estimate}), (\ref{rigid_estimate}) and (\ref{rigidity}), we can get that for any set $\Omega$ of deterministic unit vectors of cardinality $n^{\OO(1)}$, 
\begin{equation}
\sup_{\mathbf u,\mathbf v\in \Omega}\left\vert   \mathbf u^* \left(G^{s}(z)-G^{s}(\mu)\right) \mathbf v \right\vert \le n^{-{1}/{4}+3\e} \label{REALCOMPLEX}
\end{equation}
with probability $1-\oo(1)$. We only give the derivation of \eqref{REALCOMPLEX} for $\mathbf u,\mathbf v \in \mathbb C^{\mathcal I_1}$. For the rest of the cases $\mathbf u \in \C^{\cal I_\al}$ and $\mathbf v \in \mathbb C^{\mathcal I_\beta}$, $\al,\beta=1,2,3,4$, the proof is similar.  For deterministic unit vectors $\mathbf u,\mathbf v \in \mathbb C^{\mathcal I_1}$, we have with probability $1-\oo(1)$ that
\begin{align*}
& \left| \left\langle \mathbf u, \left(G^{s}(z)-G^{s}(\mu)\right) \mathbf v\right\rangle\right| \\
& \le \sum_{k}  \frac{|z-\mu|\left|\langle \mathbf u,S_{xx}^{-1/2}\xi_k\rangle \langle \xi_kS_{xx}^{-1/2},\bv\rangle \right| }{|\lambda_k^s - z| |\lambda_k^s - \mu|}   + \frac{|z-\mu|}{|z\mu|} {\sum_{k=q+1}^p{\left|\langle \bu, S_{xx}^{-1/2} \xi _k\rangle\right|\left| \langle \xi _kS_{xx}^{-1/2}, \bv\rangle \right| }}\\
& \prec \frac{1}{n^{2/3}}\sum_{k\ge q/2}{\left|\langle \bu, S_{xx}^{-1/2} \xi _k\rangle\right|\left| \langle \xi _kS_{xx}^{-1/2}, \bv\rangle \right| } +  \frac{1}{n^{5/3}}\sum_{ k<q/2} \frac{1}{|\lambda_k^s - z||\lambda_k^s - \mu|}  \\
& \le \frac{\|S_{xx}^{-1/2}  \bu\|^2 + \|S_{xx}^{-1/2}  \bv\|^2}{n^{2/3}} + \frac{1}{n^{5/3}}\sum_{1\le k \le n^\e} \frac{1}{|\lambda_k^s - z||\lambda_k^s - \mu|} + \frac{1}{n^{5/3}}\sum_{ n^\e <k <q/2 } \frac{1}{|\lambda_k^s - z||\lambda_k^s - \mu|} \\
& \prec \frac{1}{n^{2/3}} + \frac{n^{\e}}{n^{1/4}} + \frac{1}{n^{2/3}}\left(\frac{1}{n}\sum_{n^\e <k <q/2} \frac{1}{|\lambda_k^s - z||\lambda_k^s - \mu|}\right) \prec n^{-1/4+\e},
\end{align*}
where in the first step we used (\ref{spectral1}) and \eqref{simple obs}; in the second step we used (\ref{delocal1}) and $|\lambda_k - z||\lambda_k - \mu| \gtrsim 1$ for $k\ge q/2$ due to (\ref{rigidity}); in the third step we used Cauchy-Schwarz inequality; in the fourth step we used Lemma \ref{SxxSyy} and (\ref{repulsion_estimate}); in the last step we used $|\lambda_k^s - z||\lambda_k^s - \mu| \sim (k/n)^{-4/3}$ for $k>n^\e$ by the rigidity estimate (\ref{rigidity}). 

Thus we have proved \eqref{REALCOMPLEX}, which implies \eqref{final claim4}, which further concludes (\ref{suff_claim}). This completes the proof of Theorem \ref{main_thm1}. 

\section{Proof of Theorem \ref{thm_local}: the entrywise local law}\label{pf thmlocal}

The proof of Theorem \ref{thm_local} is divided into two steps. In this section, we prove a weaker local law as in Proposition \ref{prop_diagonal} below. Based on this estimate, we shall complete the proof of Theorem \ref{thm_local} in next section.

Note that \eqref{aniso_law} justifies the arguments in Remark \ref{removehat}, so we only need to prove this theorem for $\wh G(z)$. However, for simplicity of notations, we will still use the notations $G(z)$, while keeping in mind that there is a small regularization term in $G(z)$ such that the deterministic bounds in  \eqref{op G} holds. In particular, we will tacitly use the following fact: for $z\in S(\e)$ and a (complex) polynomial of the entries of $G(z)$, say $\cal P(G)$, if $|\cal P(G)|\prec \Phi(z)$ for some deterministic parameter $\Phi(z) \ge n^{-C}$, then 
$$|\mathbb E \cal P(G)| \prec \Phi(z)$$
by Lemma \ref{lem_stodomin} (iii).

The goal of this section is to prove the averaged local laws \eqref{aver_in1} and \eqref{aver_out1}, and the entrywise local law in Proposition \ref{prop_diagonal} below. For simplicity of notations, we first assume that the entries of $X$ and $Y$ satisfy
\begin{equation}\label{entry_assmX1}
\mathbb E x_{i\mu} = \mathbb E y_{j\nu} = 0, 
\end{equation}
and 
\begin{equation}\label{entry_assmX2}
 \mathbb{E} \vert x_{i\mu} \vert^2=\mathbb{E} \vert y_{j\nu} \vert^2  =n^{-1},
 \ee 
 for $  i\in \mathcal I_1$, $ j\in \mathcal I_2$, $ \mu\in \mathcal I_3$ and $ \nu\in \mathcal I_4$. Later in Section \ref{subsec_central}, we will discuss how to relax \eqref{entry_assmX1} and \eqref{entry_assmX2} to \eqref{entry_assm0} and \eqref{entry_assm1}. 

\begin{proposition}\label{prop_diagonal}
Suppose that \eqref{entry_assmX1}, \eqref{entry_assmX2} and the assumptions of Theorem \ref{lem null} hold. Then for any fixed $\e>0$, we have that 
\begin{equation}\label{entry_law}
\left| \left[\Pi^{-1}(z) \left(G (z) - \Pi (z)\right) \Pi^{-1}(z)\right]_{\fa\fb} \right| \prec \phi_n + \Psi(z),
\end{equation}
uniformly in $\fa,\fb\in \cal I$ and $z\in S(\epsilon)$.
\end{proposition}

With (\ref{entry_law}), we shall use a polynomialization method as in \cite[Section 5]{isotropic} and \cite[Section 5]{XYY_circular} to get the anisotropic local law (\ref{aniso_law}). This will be presented in Section \ref{sec_aniso}.

\subsection{Basic tools}

In this subsection, we introduce more notations and collect some basic tools that will be used in the proof.

\begin{definition}[Minors]\label{den minor}
For any $ \cal J \times \cal J$ matrix $\cal A$ and $\mathbb T \subseteq \mathcal J$, where $\cal J$ and $\mathbb T$ are some index stes, we define the minor $\cal A^{(\mathbb T)}:=(\cal A_{ab}:a,b \in \mathcal J\setminus \mathbb T)$ as the $ (\cal J\setminus \mathbb T)\times (\cal J\setminus \mathbb T)$ matrix obtained by removing all rows and columns indexed by $\mathbb T$. Note that we keep the names of indices when defining $\cal A^{(\mathbb T)}$, i.e. $(\cal A^{(\mathbb{T})})_{ab}= \cal A_{ab}$ for $a,b \notin \mathbb{{T}}$. Correspondingly, we define the resolvent minor as
\begin{align*}
G^{(\mathbb T)}(z):&=(  H^{(\mathbb T)}(z) )^{-1}.
\end{align*}
As in Definition \ref{resol_not}, its blocks are denoted as $ \cal G^{(\mathbb T)}_\al(z)$, $\al=1,2,3,4$, and $ \cal G^{(\mathbb T)}_L(z)$, $ \cal G^{(\mathbb T)}_{LR}(z)$, $ \cal G^{(\mathbb T)}_{RL}(z)$, $ \cal G^{(\mathbb T)}_R(z)$; its partial traces are  
$$ m^{(\mathbb T)}_\al(z) :=\frac1n\tr  \cal G^{(\mathbb T)}_{\al}(z) = \frac{1}{n}\sum_{\fa \in \cal I_\al}  G^{(\mathbb T)}_{\fa\fa}(z) ,\quad \al=1,2,3,4. $$
Moreover, we define $S_{xx}^{(\mathbb T)}$, $S_{xy}^{(\mathbb T)}$, $S^{(\mathbb T)}_{yy}$, $\cal H^{(\mathbb T)}$, $R^{(\mathbb T)}(z)$ and $m^{(\mathbb T)}(z)$ by replacing $(X,Y)$ with $(X^{(\mathbb T)},Y^{(\mathbb T)})$.

For $\mathbb T\subset \mathcal I_3\cup \cal I_4$, we denote $[\mathbb T]:=\{\mu\in\mathcal I_3\cup \cal I_4: \mu \in \mathbb T \text{ or } \overline \mu \in \mathbb T\}$ (recall Definition \ref{def_index}). Then we define the minor $H^{[\mathbb T]}:=H^{([\mathbb T])}$, and correspondingly $G^{[\mathbb T]}:= (H^{[\mathbb T]})^{-1} $. 

For convenience, we will adopt the convention that for any minor $\cal A^{(T)}$, $\cal A^{(T)}_{ab} = 0$ if $a \in \mathbb T$ or $b \in \mathbb T$. We will abbreviate $(\{\fa\})\equiv (\fa)$, $[\fa]\equiv [\{\fa\}]$, $(\{\fa, \fb\})\equiv (\fa\fb)$, $[\{\fa, \fb\}]\equiv [\fa\fb]$, and $\sum_{a}^{(\mathbb T)} := \sum_{a\notin \mathbb T} .$ 
\end{definition}


Using Schur complement formula, one can obtain the following resolvent identities. 
\begin{lemma}{(Resolvent identities).}

\begin{itemize}
\item[(i)]
For $i\in \mathcal I_1\cup \cal I_2$, we have 
\begin{equation}
\frac{1}{{G_{ii} }} = - zn^{-10} - \left( {WG^{\left( i \right)} W^T} \right)_{ii} . \label{resolvent1} 
\end{equation}
where we abbreviate 
$$W:= \begin{pmatrix} X & 0 \\ 0 & Y\end{pmatrix}.$$


 \item[(ii)]
 For $i\ne j \in \mathcal I_1\cup \cal I_2$, we have
\begin{equation}
G_{ij} =-G_{ii} \left(WG^{(i)}\right)_{ij} = - \left(G^{(j)}W^T\right)_{ij} G_{jj} = G_{ii} G_{jj}^{\left( i \right)} \left( {WG^{\left( {ij} \right)} W^T } \right)_{ij}. \label{resolvent3}
\end{equation}
For $i\in \mathcal I_1\cup \cal I_2$ and $\mu\in \mathcal I_3\cup \cal I_4$, we have
\begin{equation}\label{resolvent6}
\begin{split}
 G_{i\mu } &= - G_{ii} ( W G^{(i)})_{i\mu} = -  ( G^{(\mu)}W)_{i\mu}G_{\mu\mu}, \\
 G_{\mu i}  & = -G_{\mu\mu} (W^T G^{(\mu)})_{\mu i} =  -(G^{(i)}W^T)_{\mu i}G_{ii} .
\end{split}
\ee  

 \item[(iii)]
 For $\fa \in \mathcal I$ and $\fb, \mathfrak c \in \mathcal I \setminus \{\fa\}$,
\begin{equation}
G_{\fb \mathfrak c} = G_{\fb \mathfrak c}^{\left( \fa \right)}  + \frac{G_{\fb\fa} G_{\fa \mathfrak c}}{G_{\fa\fa}}, \ \ \frac{1}{{G_{\fb\fb} }} = \frac{1}{{G_{\fb\fb}^{(\fa)} }} - \frac{{G_{\fb\fa} G_{\fa\fb} }}{{G_{\fb\fb} G_{\fb\fb}^{(\fa)} G_{\fa\fa} }}. \label{resolvent8}
\end{equation}

 \item[(iv)]
All of the above identities hold for $G^{(\mathbb T)}$ instead of $G$ for $\mathbb T \subset \mathcal I$.
\end{itemize}
\label{lemm_resolvent}
\end{lemma}


For $\cal I \times \mathcal I$ matrix $\cal A$ and $\mu,\nu \in \mathcal I_3$, we define the $2\times 2$ minors as
\begin{equation}\label{Aij_group}
 \cal A_{[\mu\nu]}=\left( {\begin{array}{*{20}c}
   {\cal A_{\mu\nu} } & {\cal A_{\mu \overline \nu} }  \\
   {\cal A_{ \overline\mu \nu} } & {\cal A_{ \overline \mu \overline \nu} }  \\
\end{array}} \right),
\end{equation}
where we recall the notations in Definition \ref{def_index}. 
We shall call $A_{[\mu\nu]}$ a diagonal group if $\mu=\nu$, and an off-diagonal group otherwise.
Similarly, for $i\in \cal I_1$, $j \in \cal I_2$ and $\mu\in \cal I_3$, we define the minors
\begin{equation}\label{Aij_group}
\begin{split}
 \cal A_{ij,[\mu]}=\left( {\begin{array}{*{20}c}
   {\cal A_{i\mu} } & {\cal A_{i\overline \mu} }  \\
   {\cal A_{j\mu} } & {\cal A_{j\overline \mu} }  \\
\end{array}} \right), \quad & \cal A_{[\mu],ij}=\left( {\begin{array}{*{20}c}
   {\cal A_{\mu i} } & {\cal A_{\mu j} }  \\
   {\cal A_{\overline \mu i} } & {\cal A_{ \overline \mu j} }  \\
\end{array}} \right),\\ 
 \cal A_{i,[\mu]}=\left(  {\cal A_{i\mu} }, {\cal A_{i\overline \mu} } 
 \right), \quad & \cal A_{[\mu],i}=\left( {\begin{array}{*{20}c}
   {\cal A_{\mu i} }  \\
   {\cal A_{\overline \mu i} }\\
\end{array}} \right).
\end{split}
\end{equation}
For $G$, sometimes it is convenient to deal with $2\times 2$ blocks directly, and we record the following resolvent identities obtained from Schur complement formula.

\begin{lemma}{(Resolvent identities for $G_{[\mu\nu]}$ groups).}\label{lemm_resolvent_group}
  \begin{itemize}
  \item[(i)] For $\mu\in \mathcal I_3$, we have 
 \begin{equation}\label{eq_res11}
   G_{[\mu\mu]}^{ - 1}  =\frac{1}{z-1}\begin{pmatrix}1 & -z^{-1/2} \\ -z^{-1/2} & 1 \end{pmatrix}  - \begin{pmatrix}( X^T G^{[\mu]} X)_{\mu\mu}& ( X^T G^{[\mu]}Y)_{\mu \overline \mu} \\  (Y^TG^{[\mu]}  X)_{\overline \mu  \mu} &  (Y^T G^{[\mu]}  Y)_{\overline \mu\overline\mu} \end{pmatrix}   .
    \end{equation}
   
    \item[(ii)]
    For $i \in \cal I_1$ and $\mu\in \cal I_3$, we have
    \be
    \begin{split}
   G_{i, [\mu]}& = - \begin{pmatrix}( G^{[\mu]} X)_{i\mu} ,  ( G^{[\mu]} Y)_{i\overline \mu} \end{pmatrix}G_{[\mu\mu]} \\
   &=  G_{ii}^{[\mu]}\begin{pmatrix} (XG^{(i\mu\overline \mu)} X)_{i\mu} \ , \ (XG^{(i\mu\overline \mu)} Y)_{i\overline \mu} \end{pmatrix}G_{[\mu\mu]} ,
   \label{eq_res22}
    \end{split}   
    \ee 
    and
    \begin{align}
   G_{[\mu],i}& = -  G_{[\mu\mu]} \begin{pmatrix}  (X^T G^{[\mu]})_{ \mu i}   \\ ( Y^TG^{[\mu]})_{\overline \mu i} \end{pmatrix} = G_{[\mu\mu]} \begin{pmatrix}  (X^T G^{(i\mu\overline \mu)}X^T)_{ \mu i}   \\ ( Y^TG^{(i\mu\overline \mu)}X^T)_{\overline \mu i} \end{pmatrix}G_{ii}^{[\mu]}.\label{eq_res23}
       \end{align}
       We have similar expansions for $G_{j, [\mu]}$ and $G_{[\mu],j}$ for $j\in \cal I_2$ by interchanging $X$ and $Y$. 
    
   \item[(iii)] 
   For $\mu\ne \nu \in \cal I_3$, we have
     \begin{align}
   G_{[\mu\nu]}  &= - G_{[\mu\mu]} \begin{pmatrix} (X^T G^{[\mu]})_{\mu\nu} &  (X^T G^{[\mu]})_{\mu \overline \nu} \\ (Y^TG^{[\mu]})_{\overline \mu \nu} & (Y^TG^{[\mu]})_{\overline\mu  \overline\nu} \end{pmatrix}   = -  \begin{pmatrix} (G^{[\nu]} X)_{\mu\nu  } & (G^{[\nu]} Y)_{\mu \overline \nu}\\ (G^{[\nu]} X)_{\overline \mu \nu } & (G^{[\nu]} Y)_{\overline \mu \overline\nu }\end{pmatrix} G_{[\nu\nu]} 
   \nonumber\\
 & = G_{[\mu\mu]}G_{[\nu\nu]}^{[\mu]} \begin{pmatrix} (X^T G^{[\mu\nu]} X)_{\mu\nu} & (X^T G^{[\mu\nu]}Y)_{\mu\overline \nu} \\ (Y^T G^{[\mu\nu]} X)_{\overline \mu\nu} & (Y^T G^{[\mu\nu]} Y)_{\overline \mu \overline \nu}\end{pmatrix}. \label{eq_res21}
    \end{align}

  \item[(iv)] For $\mu \in \mathcal I_3 $ and $\fa_1,\fa_2,\fb_1,\fb_2 \in \mathcal I\setminus \{\mu,\nu\}$, we have 
  \begin{equation}\label{eq_res3}
     \begin{pmatrix}G_{\fa_1\fb_1} & G_{\fa_1\fb_2} \\ G_{\fa_2\fb_1} & G_{\fa_2\fb_2} \end{pmatrix} = \begin{pmatrix}G^{[\mu]}_{\fa_1\fb_1} & G^{[\mu]}_{\fa_1\fb_2} \\ G^{[\mu]}_{\fa_2\fb_1} & G^{[\mu]}_{\fa_2\fb_2} \end{pmatrix}+\begin{pmatrix}G_{\fa_1\mu} & G_{\fa_1\overline\mu} \\ G_{\fa_2\mu} & G_{\fa_2\overline\mu} \end{pmatrix}G^{-1}_{[\mu\mu]}\begin{pmatrix}G_{\mu\fb_1} & G_{\mu\fb_2} \\ G_{\overline \mu\fb_1} & G_{\overline \mu\fb_2} \end{pmatrix}, 
    \end{equation}
    and
    \begin{equation}\label{eq_res4}
    \begin{split}
  &   \begin{pmatrix}G_{\fa_1\fa_1} & G_{\fa_1\fa_2} \\ G_{\fa_2\fa_1} & G_{\fa_2\fa_2} \end{pmatrix}^{-1} = \begin{pmatrix}G^{[\mu]}_{\fa_1\fa_1} & G^{[\mu]}_{\fa_1\fa_2} \\ G^{[\mu]}_{\fa_2\fa_1} & G^{[\mu]}_{\fa_2\fa_2} \end{pmatrix}^{ - 1}  \\
&     -  \begin{pmatrix}G_{\fa_1\fa_1} & G_{\fa_1\fa_2} \\ G_{\fa_2\fa_1} & G_{\fa_2\fa_2} \end{pmatrix}^{-1} \begin{pmatrix}G_{\fa_1\mu} & G_{\fa_1\overline\mu} \\ G_{\fa_2\mu} & G_{\fa_2\overline\mu} \end{pmatrix}G^{-1}_{[\mu\mu]}\begin{pmatrix}G_{\mu\fb_1} & G_{\mu\fb_2} \\ G_{\overline \mu\fb_1} & G_{\overline \mu\fb_2} \end{pmatrix}\begin{pmatrix}G^{[\mu]}_{\fa_1\fa_1} & G^{[\mu]}_{\fa_1\fa_2} \\ G^{[\mu]}_{\fa_2\fa_1} & G^{[\mu]}_{\fa_2\fa_2} \end{pmatrix}^{ - 1}  .
     \end{split}
    \end{equation}
  \item[(iii)] All of the above identities hold for $G^{(\mathbb T)}$ instead of $G$ for $\mathbb T\subset \mathcal I$.
  \end{itemize}
\end{lemma}

The following lemma gives large deviation bounds for bounded supported random variables. 

\begin{lemma}[Lemma 3.8 of \cite{EKYY1}]\label{largedeviation}
 Let $(x_i)$, $(y_j)$ be independent families of centered and independent random variables, and $(\cal A_i)$, $(\cal B_{ij})$ be families of deterministic complex numbers. Suppose the entries $x_i$, $y_j$ have variance at most $n^{-1}$ and satisfy the bounded support condition (\ref{eq_support}) with $\phi_n\le n^{-c_\phi }$ for some constant $c_\phi>0$. Then we have the following bound:
\begin{align*}
&\Big\vert \sum_i \cal A_i x_i \Big\vert \prec  \phi_n \max_{i} \vert \cal A_i \vert+ \frac{1}{\sqrt{n}}\Big(\sum_i |\cal A_i|^2 \Big)^{1/2} , \\  
& \Big\vert \sum_{i,j} x_i \cal B_{ij} y_j \Big\vert \prec \phi_n^2 \cal B_d  + \phi_n \cal B_o + \frac{1}{n}\Big(\sum_{i\ne j} |\cal B_{ij}|^2\Big)^{{1}/{2}} , \\
 &\Big\vert \sum_{i} \overline x_i \cal B_{ii} x_i - \sum_{i} (\mathbb E|x_i|^2) \cal B_{ii}  \Big\vert  \prec \phi_n \cal B_d  ,\\ 
 & \Big\vert \sum_{i\ne j} \overline x_i \cal B_{ij} x_j \Big\vert  \prec \phi_n\cal B_o + \frac{1}{n}\Big(\sum_{i\ne j} |\cal B_{ij}|^2\Big)^{{1}/{2}} ,
\end{align*}
where $\cal B_d:=\max_{i} |\cal B_{ii} |$ and $\cal B_o:= \max_{i\ne j} |\cal B_{ij}|.$
\end{lemma}  

Corresponding to the lower right block of $\Pi(z)$ in \eqref{defn_pi}, we define the $2\times 2$ matrix
\be\label{def smallpi}
\begin{split} 
&\pi(z): = \begin{pmatrix} m_{3c}(z) & h(z) \\ h(z) & m_{4c}(z) \end{pmatrix}   \\
&= \frac{ 1-c_1 - c_2 + \sqrt{(z-\lambda_-)(z-\lambda_+)}}{2}\begin{pmatrix} 1 & z^{1/2} \\ z^{1/2} & 1 \end{pmatrix} + \frac{z-1}{2}\begin{pmatrix} 1-2c_1 & -z^{1/2} \\ -z^{1/2} & 1-2c_2 \end{pmatrix} .\end{split}\ee
Using \eqref{selfm12}-\eqref{selfm32}, one can check that
\be\label{def smallpi2}
\pi(z)^{-1} =\frac{1}{z-1} \begin{pmatrix}  1 - (z-1)m_{1c} & - z^{-1/2} \\ - z^{-1/2}   &  1 - (z-1)m_{2c} \end{pmatrix} .\ee

For the proof of Proposition \ref{prop_diagonal}, it is convenient to introduce the following random control parameters. 

\begin{definition}[Control parameters]\label{defn_randomerror}
We define the random errors
\begin{equation}  \label{eqn_randomerror}
\begin{split}
 \Lambda _o : &=  \max_{i \ne j \in \mathcal I_1\cup \cal I_2} \left| {G_{ij}  } \right| +  \max_{\mu\ne \nu \in \cal I_3} \|\pi^{-1}G_{[\mu\nu]}\pi^{-1}\| \\
 &+ \max_{i\in \cal I_1\cup \cal I_2,\mu\in \cal I_3}\left( \left\| G_{i,[\mu]} \pi^{-1}\right\| + \left\| \pi^{-1}G_{[\mu],i} \right\| \right), \\
 \Lambda: &= \Lambda_o + \max_{i\in \cal I_1 \cup \cal I_2}|G_{ii} - \Pi_{ii}| +  \max_{\mu \in \cal I_3}   \left\|\pi^{-1}\left(G_{[\mu\mu]}  - \pi\right)\pi^{-1}\right\|   ,
\end{split}
\end{equation}
and
\begin{equation}\label{eqn_randomerror2}
\theta := |m_1-m_{1c}| +  |m_2-m_{2c}| + \left\|\frac1n \sum_{\mu} \left( G_{[\mu\mu]} - \pi\right)\right\|.
 \ee
Here $\Lambda$ controls the entrywise error, $\Lambda_o$ controls the size of the off-diagonal entries, and $\theta$ gives the averaged error. 
Note that these parameters all depend on $z$, and we did not write down this dependence explicitly in the definitions. Moreover, replacing $G$ with $G^{(\mathbb T)}$ for any $\mathbb T \subset \cal I$, we can define parameters $\Lambda_o^{(\mathbb T)}$, $\Lambda^{(\mathbb T)}$ and $\theta^{(\mathbb T)}$. We then define the random control parameter (recall $\Psi$ defined in \eqref{eq_defpsi})
\begin{equation}\label{eq_defpsitheta}
\Psi _\Lambda(z)  : = \sqrt {\frac{{\Im \, m_{c}(z)  + \Lambda(z)}}{{n\eta }}} + \frac{1}{n\eta}.
\end{equation}
\end{definition}

Define a $z$-dependent event 
$$\Xi(z):=\{\Lambda(z) \le (\log n)^{-1}\}.$$ 
Then on $\Xi$, using \eqref{Immc} and \eqref{eqn_randomerror}, we have that for $i\in \cal I_1\cup \cal I_2$ and $\mu\in \cal I_3$, 
\be\label{Gimu}G_{ii} \sim 1 ,\quad G_{[\mu\mu]}= \pi(z)  + \pi(z)  \cal E(z)\pi(z) \ \ \text{with}\ \  \|\cal E(z)\|=\OO\left((\log n)^{-1}\right).\ee
Then using \eqref{Gimu} and \eqref{resolvent8}, we obtain that for $k \in \mathcal I_1\cup \mathcal I_2$, $i,j \in \mathcal I_1\cup \mathcal I_2 \setminus \{k\}$ and $\mu,\nu\in \cal I_3$,
$$ \mathbf 1(\Xi)|G_{ij} - G_{ij}^{(k)}| \lesssim \Lambda_o^2, \quad \mathbf 1(\Xi)\left\|\pi^{-1}\left(G_{[\mu\nu]}-G^{(k)}_{[\mu\nu]}\right)\pi^{-1}\right\| \lesssim \left\|\pi^{-1} G_{[\mu],k}\right\|\left\|  G_{k,[\nu]}\pi^{-1}\right\| \lesssim \Lambda_o^2,$$
and
$$ \mathbf 1(\Xi)\left\|\pi^{-1}\left(G_{[\mu],i}-G^{(k)}_{[\mu],i}\right) \right\| \lesssim \left\|\pi^{-1} G_{[\mu],k}\right\| |G_{ik}|\lesssim \Lambda_o^2 .$$
Similarly, for $i,j \in \mathcal I_1 \cup  \mathcal I_2$, $\mu \in \cal I_3$ and $\al,\beta \in \cal I_3 \setminus \{\mu\}$, using \eqref{Gimu} and \eqref{eq_res3} we obtain that
$$ \mathbf 1(\Xi)\left|  G_{ij}  - G^{[\mu]}_{ij}  \right| \lesssim  \left\| G_{i,[\mu]}\right\| \left\|\pi^{-1}G_{[\mu],j}\right\| \lesssim \Lambda_o^2, $$
$$\quad \mathbf 1(\Xi)\left\|\pi^{-1}\left(G_{[\al\beta]}-G^{[\mu]}_{[\al\beta]}\right)\pi^{-1}\right\| \lesssim  \left\|\pi^{-1} G_{[\al\mu]}\pi^{-1}\right\|\left\|  G_{[\mu\beta]}\pi^{-1}\right\| \lesssim  \Lambda_o^2,$$
and
$$ \mathbf 1(\Xi)\left\|\pi^{-1}\left(G_{[\al],i}-G^{[\mu]}_{[\al],i}\right) \right\| \lesssim \left\|\pi^{-1} G_{[\al\mu]} \pi^{-1}\right\| \left\|G_{[\mu],i}\right\| \lesssim \Lambda_o^2 .$$
Thus with an induction on the indices, we obtain that for any $\mathbb T\subset \cal I$ such that $\mathbb T = \{ i_1, \cdots, i_k, \mu_1 ,\overline \mu_1, \cdots, \mu_l,\overline \mu_l\}$ for some fixed integers $k,l\in \N$, we have 
\be\label{simpleT}
\begin{split}
& \mathbf 1(\Xi)\max_{i , j \in \mathcal I_1\cup \cal I_2} \left| G_{ij} -G_{ij}^{(\mathbb T)} \right| + \mathbf 1(\Xi) \max_{\mu , \nu \in \cal I_3} \left\|\pi^{-1}\left(G_{[\mu\nu]}-G_{[\mu\nu]}^{(\mathbb T)} \right)\pi^{-1}\right\| \\
 + & \mathbf 1(\Xi) \max_{i\in \cal I_1\cup \cal I_2,\mu\in \cal I_3}\left( \left\| \left(G_{i,[\mu]}-G_{i,[\mu]}^{(\mathbb T)}\right) \pi^{-1}\right\| + \left\| \pi^{-1}\left(G_{[\mu],i}-G^{(\mathbb T)}_{[\mu],i}\right) \right\| \right) \lesssim \Lambda_o^2.
\end{split}\ee
In particular, we have 
\be\label{simpleT2}\mathbf 1(\Xi) \Lambda _o^{(\mathbb T)}= \OO(\Lambda_o), \quad \mathbf 1(\Xi) \Lambda^{(\mathbb T)}= \OO(\Lambda), \quad  \mathbf 1(\Xi) \theta^{(\mathbb T)}=  \mathbf 1(\Xi) \theta + \OO(\Lambda_o^2),\ee
which we shall use tacitly in the proof.

\subsection{Entrywise local law}\label{secentryweak}

%

In analogy to \cite[Section 3]{EKYY1} and \cite[Section 5]{Anisotropic}, we introduce the $Z$ variables
\begin{equation*}
  Z_{\fa}^{(\mathbb T)}:=(1-\mathbb E_{\fa})\big(G_{\fa\fa}^{(\mathbb T)}\big)^{-1}, \quad \fa\notin \mathbb T,
\end{equation*}
where $\mathbb E_{\fa}[\cdot]:=\mathbb E[\cdot\mid H^{(\fa)}],$ i.e. it is the partial expectation over the randomness of the $\fa$-th row and column of $H$. By (\ref{resolvent1}), we have that for $i\in \cal I_\al$, $\al=1,2$, 
\begin{equation}
Z_i = (\mathbb E_{i} - 1) \left( {WG^{\left( i \right)} W^T} \right)_{ii} = \sum_{\mu ,\nu\in \mathcal I_{\al+2}} G^{(i)}_{\mu\nu} \left(\frac{1}{n}\delta_{\mu\nu} - W_{i\mu}W_{i\nu}\right). \label{Zi}
\end{equation}
We also introduce the matrix value $Z$ variables:
\begin{equation*}
  Z_{[\mu]}^{(\mathbb T)}:=(1-\bbE_{[\mu]})\left(G_{[\mu\mu]}^{[\mathbb T]}\right)^{-1}, \quad \mu,\overline \mu \notin \mathbb T,
\end{equation*}
where $\mathbb E_{[\mu]}[\cdot]:=\mathbb E[\cdot\mid H^{[\mu]}],$ i.e. it is the partial expectation over the randomness of the $\mu$ and $\overline \mu$-th rows and columns of $H$. By (\ref{eq_res11}), we have
\begin{equation}
G_{[\mu\mu]}^{-1}  =  \frac{1}{z-1}\begin{pmatrix}1 & -z^{-1/2} \\ -z^{-1/2} & 1 \end{pmatrix}  -  \left( {\begin{array}{*{20}c}
   { m_{1}^{[\mu]} } & {0}  \\
   {0} & { m_{2}^{[\mu]} } \end{array}} \right) + Z_{[\mu]}, \label{resolvent_Gii}
\end{equation}
where
\begin{equation}\label{Zmu}
Z_{[\mu]} = \begin{pmatrix}\sum_{i, j \in \cal I_1} G^{[\mu]}_{ij} (n^{-1}\delta_{ij} -X_{i\mu  }X_{j\mu}) & \sum_{i \in \cal I_1, j\in \cal I_2} G^{[\mu]}_{ij} X_{i\mu}Y_{j\overline \mu} \\ \sum_{i \in \cal I_1, j\in \cal I_2}G^{[\mu]}_{ji} X_{i \mu} Y_{j\overline \mu } & \sum_{i, j \in \cal I_2} G^{[\mu]}_{ij} (n^{-1}\delta_{ij} - Y_{i\overline\mu }Y_{j\overline\mu}) \end{pmatrix} .
\end{equation}

Then using Lemma \ref{largedeviation}, we can prove the following large deviation estimates for the $Z$ variables and off-diagonal entries.

\begin{lemma}\label{Z_lemma}
Suppose the assumptions in Proposition \ref{prop_diagonal} hold. Let $c_0>0$ be a sufficiently small constant and fix $C_0, \epsilon >0$. Then the following estimates hold uniformly for all $i\in \mathcal I_1\cup \cal I_2$, $\mu\in \cal I_3$, and $z= E+\ii\eta\in S(\epsilon)$:
\begin{align}
& {\mathbf 1}(\Xi(z))\left(|Z_{i}| + \|Z_{[\mu]}\|\right)   \prec \phi_n+ \Psi_\Lambda; \label{Zestimate1}\\
& {\mathbf 1}(\Xi(z)) \Lambda_o  \prec \phi_n+\Psi_\Lambda; \label{Oestimate1} \\
& {\mathbf 1}\left(\eta \ge 1 \right)\left(|Z_{i}| + \|Z_{[\mu]}\|+\Lambda_o  \right)\prec \phi_n .\label{Zestimate2}
\end{align}
\end{lemma}
\begin{proof}
For $i\in \cal I_1$, applying Lemma \ref{largedeviation} to $Z_{i}$ in (\ref{Zi}), we get that on $\Xi$,
\begin{equation}\nonumber
\begin{split}
\left| Z_{i}\right| &=\left| \sum_{\mu ,\nu\in \mathcal I_{3}} G^{(i)}_{\mu\nu} \left(\frac{1}{n}\delta_{\mu\nu} - X_{i\mu}X_{i\nu}\right)\right| \prec  \phi_n +\frac{1}{n} \left( \sum_{\mu, \nu \in \cal I_3}  {\left| G_{\mu\nu}^{(i)}  \right|^2 }  \right)^{1/2} \\
& \prec  \phi_n + \frac{1}{n} \left[ \sum_{\nu\in \cal I_3} \left( 1+ \frac{ \im \left( U(z) G_{[\nu\nu]}^{(i)} \right)_{11} }{\eta}   \right]  \right)^{1/2} ,
\end{split}
\end{equation}
where we used \eqref{eq_sgsq2}, $U(z)$ is the $2\times 2$ matrix 
$$ U(z) := z^{1/2} \begin{pmatrix}  \overline z   &  \overline z^{1/2} \\  \overline z^{1/2}  &   \overline z  \end{pmatrix}    \begin{pmatrix}  z  & z^{1/2} \\ z^{1/2} &  z  \end{pmatrix}^{-1}  , $$
and the subscript ``$11$" means the $(1,1)$-th entry of the $2\times 2$ matrix. Then using (\ref{eqn_randomerror}), \eqref{simpleT2} and $n^{-1/2}\le \phi_n$, we obtain that 
\begin{equation}\label{estimate_Zi}
\begin{split}
{\mathbf 1}(\Xi(z))\left| Z_{i}\right| & \prec  \phi_n + \left(   \frac{ \im \left( U(z) \pi(z) \right)_{11} + \Lambda}{n\eta}   \right)^{1/2} \prec \phi_n + \Psi_\Lambda.
\end{split}
\end{equation}
Here we used that for $\pi$ in \eqref{def smallpi},
\begin{align*}
U(z) \pi(z) &= \frac{1}{2}\left( 1-c_1 - c_2 + \sqrt{(z-\lambda_-)(z-\lambda_+)}\right)  \begin{pmatrix}  \overline z^{1/2}   &  \overline z \\  \overline z  &   \overline z^{1/2}  \end{pmatrix} \\ 
&+ \frac{1}{2} \begin{pmatrix}  \overline z   &  \overline z^{1/2} \\  \overline z^{1/2}  &   \overline z  \end{pmatrix} \begin{pmatrix}  z^{1/2}  & -1 \\ -1 &  z^{1/2}  \end{pmatrix} \begin{pmatrix} 1-2c_1 & -z^{1/2} \\ -z^{1/2} & 1-2c_2 ,\end{pmatrix} 
\end{align*}
which, together with \eqref{Immc}, implies that 
\be\label{simpleU}\left\| \im \left( U(z) \pi(z) \right)\right\|  = \OO(\im m_c(z)).\ee
Similarly, for $i\in \cal I_2$, we can prove the same estimate \eqref{estimate_Zi} for $\mathbf 1(\Xi)|Z_i|$. Next we pick $\mu \in \cal I_3$. Applying Lemma \ref{largedeviation} again to $Z_{[\mu]}$ in (\ref{Zmu}) and using \eqref{eq_sgsq1}, we obtain that on $\Xi$,
\begin{equation}\label{estimate_Zmu} 
\begin{split}
\|Z_{[\mu]}\| &\prec \phi_n +\frac{1}{n} \left( \sum_{i,j \in \cal I_{1}\cup \cal I_2}  {\left| G_{ij}^{[\mu]}  \right|^2 }  \right)^{1/2} \\
& \prec \phi_n + \left( \frac {\im m^{[\mu]}_1(z)+\im m^{[\mu]}_2(z) }{n\eta}\right)^{1/2} \prec \phi_n +\Psi_\Lambda.
\end{split}
\ee
This completes the proof of \eqref{Zestimate1}.

Then we prove \eqref{Oestimate1}. For $i\ne j \in \mathcal I_1\cup \cal I_2$, using \eqref{resolvent3}, Lemma \ref{largedeviation} and Lemma \ref{lem_gbound0}, we obtain that on $\Xi$,
\begin{equation}
 |G_{ij}| \prec \phi_n +\frac{1}{n} \left( \sum_{\mu, \nu \in \cal I_3}  {\left| G_{\mu\nu}^{(i)}  \right|^2 }  \right)^{1/2} \prec \phi_n + \Psi_{\Lambda} . 
 \ee
For $\mu \ne \nu \in \mathcal I_3$, using \eqref{eq_res21}, Lemma \ref{largedeviation} and Lemma \ref{lem_gbound0}, we obtain that on $\Xi$,
$$ \left\| \pi^{-1} G_{[\mu\nu]}\pi^{-1} \right\| \prec  \phi_n +\frac{1}{n} \left( \sum_{i,j \in \cal I_{1}\cup \cal I_2}  {\left| G_{ij}^{[\mu]}  \right|^2 }  \right)^{1/2} \prec  \phi_n +\Psi_\Lambda.$$
For $i\in \cal I_1\cup \cal I_2$ and $\mu \in \mathcal I_3$, using \eqref{eq_res23}, Lemma \ref{largedeviation} and Lemma \ref{lem_gbound0}, we obtain that on $\Xi$,
\begin{align*} 
\left\| \pi^{-1} G_{[\mu],i} \right\| &\prec  \phi_n +\frac{1}{n} \left( \sum_{j \in \cal I_{1}\cup \cal I_2, \nu \in \cal I_3 \cup I_4}  {\left| G_{j\nu}^{(i\mu\overline \mu)}  \right|^2 }  \right)^{1/2} \\
&\prec  \phi_n + \left(  \frac{\im m_{1}^{(i\mu\overline \mu)} +\im m_{2}^{(i\mu\overline \mu)} }{n\eta}     \right)^{1/2} \prec \Psi_\Lambda.
\end{align*}
Thus we conclude \eqref{Oestimate1}.

The proof of \eqref{Zestimate2} is similar, except that when $\eta\ge 1$ we use $\mathbf 1(\eta\ge 1)\|G^{(\mathbb T)}(z)\| =\OO(1)$ with high probability by \eqref{op G1} and $\|\pi^{-1}(z)\|=\OO(1)$. For example, for the estimate \eqref{estimate_Zmu}, we have that for $\eta\ge 1$,
$$\|Z_{[\mu]}\|\prec \phi_n + \left( \frac {\im m^{[\mu]}_1(z)+\im m^{[\mu]}_2(z) }{n\eta}\right)^{1/2} \prec \phi_n + n^{-1/2} =\OO(\phi_n).$$
We omit the rest of the details.
\end{proof}

A key component of the proof for Proposition \ref{prop_diagonal} is an analysis of the self-consistent equation. Recall the equations in \eqref{selfm12}-\eqref{selfm32}.


\begin{lemma}\label{lemm_selfcons_weak}
Fix any constant $\e>0$. The following estimates hold uniformly in $z \in S(\epsilon)$: 
\begin{align}
& {\mathbf 1}(\Xi)\left(\left| m_1 + c_1 m_{3}^{-1} \right| +  \left| m_2 + c_2 m_{4}^{-1} \right| \right)\prec \phi_n+\Psi_\Lambda, \label{selfcons_lemm}\\
 & {\mathbf 1}(\Xi)\left| m_{3}^2 + \left[ (2c_1 -1)z - c_1+c_2\right]m_{3}  + c_1(c_1-1)z(z-1) \right| \prec \phi_n+\Psi_\Lambda. \label{selfcons_lemm12}
\end{align}
Moreover, we have the finer estimates
\begin{align}
&{\mathbf 1}(\Xi)\left(\left| m_1 + c_1 m_{3}^{-1} \right| +  \left| m_2+ c_2 m_{4}^{-1} \right| \right)\prec \left|\langle Z\rangle_1\right| + \left|\langle Z\rangle_2\right| +\phi_n^2+ \Psi^2_\Lambda, \label{selfcons_improved}\\
&{\mathbf 1}(\Xi)\left| m_{3}^2 + \left[ (2c_1 -1)z - c_1+c_2\right]m_{3} + c_1(c_1-1)z(z-1) \right| \nonumber\\
&\qquad \qquad \qquad\qquad \qquad \qquad \ \prec \left|\langle Z\rangle_1\right| + \left|\langle Z\rangle_2\right|+\left\|  [Z] \right\|  +\phi_n^2+ \Psi^2_\Lambda,\label{selfcons_improved2}
 \end{align}
where
\begin{equation}\label{def_Zaver}
\langle Z\rangle_1:=\frac{1}{n}\sum_{i\in \mathcal I_1} Z_i, \quad \langle Z\rangle_2:=\frac{1}{n}\sum_{j \in \mathcal I_2}  Z_j , \quad [Z]:= \frac1n\sum_{\mu\in \cal I_3}G_{[\mu\mu]}.
\end{equation}
Finally, there exists a constant $C_0>0$ such that 
\begin{align}
&{\mathbf 1}(C_0 \le \eta \le 2C_0)\left(\left| m_1 + c_1 m_{3}^{-1} \right| +  \left| m_2 + c_2 m_{4}^{-1} \right| \right) \prec \phi_n , \label{selfcons_lemm21}\\
&{\mathbf 1}(C_0 \le \eta \le 2C_0)\left| m_{3}^2 + \left[ (2c_1 -1)z - c_1+c_2\right]m_{3} + c_1(c_1-1)z(z-1) \right| \prec \phi_n .\label{selfcons_lemm22}
\end{align}

\end{lemma}

\begin{proof}
We first prove (\ref{selfcons_improved}) and \eqref{selfcons_improved2}, from which (\ref{selfcons_lemm}) and \eqref{selfcons_lemm12} follow due to (\ref{Zestimate1}). By (\ref{resolvent1}), (\ref{Zi}) and (\ref{resolvent_Gii}), we have that for $i\in \cal I_1$, $j\in \cal I_2$ and $\mu \in \cal I_3$,
\begin{equation}\label{self_Gii}
\frac{1}{{G_{ii} }}= - m_3 + \epsilon_i, \quad \frac{1}{{G_{jj} }}= - m_4 + \epsilon_j,
\end{equation}
and
\begin{equation}\label{self_Gmu}
G_{[\mu\mu]}^{-1} =  \frac{1}{z-1}\begin{pmatrix}1 & -z^{-1/2} \\ -z^{-1/2} & 1 \end{pmatrix} - \left( {\begin{array}{*{20}c}
   { m_{1}  } & {0}  \\
   {0} & { m_{2} } \end{array}} \right) + \epsilon_\mu,
\end{equation}
where 
$$\epsilon_i := Z_i + \left(m_3 - m_3^{(i)}\right)+\OO(n^{-10}), \quad \epsilon_j := Z_j + \left(m_4 - m_4^{(j)}\right)+\OO(n^{-10}),$$
and
$$ \epsilon_\mu := Z_{\mu} + \left( {\begin{array}{*{20}c}
   { m_{1} } & {0}  \\
   {0} & { m_{2} } \end{array}} \right) -\left( {\begin{array}{*{20}c}
   { m_{1}^{[\mu]} } & {0}  \\
   {0} & { m_{2}^{[\mu]} } \end{array}} \right).$$
By (\ref{simpleT}), \eqref{Zestimate1} and (\ref{Oestimate1}), we have 
\begin{equation}\label{erri}
\mathbf 1(\Xi)\left(|\epsilon_i | +|\epsilon_j | + \|\epsilon_\mu\| \right)\prec \phi_n + \Psi_\Lambda,
\end{equation}
and
\begin{equation}\label{high_err0}
\mathbf 1(\Xi)\left( |m_1 -m_1^{[\mu]}|+ |m_2 -m_2^{[\mu]}|+|m_3 - m_3^{(i)}| + |m_4 - m_4^{(j)}|  \right) \prec \phi_n^2 + \Psi_\Lambda^2.
\end{equation}
Now using \eqref{self_Gii}, \eqref{erri}, \eqref{high_err0}, \eqref{absmc} and the definition of $\Xi$, we can obtain that for $i\in \cal I_1$ and $j\in \cal I_2$,
\be\label{Gii0}
\begin{split}
& \mathbf 1(\Xi)G_{ii}=\mathbf 1(\Xi)\left(- \frac{1}{m_3} - \frac{Z_i}{m_3^2}  +\OO_\prec\left(\phi_n^2 + \Psi_\Lambda^2\right)\right), \\ 
& \mathbf 1(\Xi)G_{jj}=\mathbf 1(\Xi)\left(- \frac{1}{m_4} - \frac{Z_j}{m_4^2}  +\OO_\prec\left(\phi_n^2 + \Psi_\Lambda^2\right)\right).
\end{split}
\ee
Taking average $\frac{1}{n}\sum_{i\in \cal I_1}$ and $\frac{1}{n}\sum_{j\in \cal I_2}$, we get
\be\label{Gii}
\begin{split}
& \mathbf 1(\Xi)m_1 =\mathbf 1(\Xi)\left( -\frac{c_1}{m_3}  - \frac{\langle Z\rangle_1}{m_3^{2}}  +\OO_\prec\left(\phi_n^2+\Psi_\Lambda^2\right)\right), \\
& \mathbf 1(\Xi)m_2 =\mathbf 1(\Xi)\left( -\frac{c_2}{m_4} -\frac{\langle Z\rangle_2}{ m_4^{2}}  + \OO_\prec\left(\phi_n^2+\Psi_\Lambda^2\right)\right),
\end{split}
\ee
which proves \eqref{selfcons_improved}. On the other hand, using (\ref{self_Gmu}), \eqref{erri}, \eqref{high_err0} and the definition of $\Xi$, we obtain that for $\mu \in \cal I_3$,
\be\label{Gmumu0}
\mathbf 1(\Xi)G_{[\mu\mu]} =\mathbf 1(\Xi)\left( \wt \pi^{-1} + \e_\mu \right)^{-1} =\mathbf 1(\Xi)\left( \wt \pi  - \wt \pi \e_\mu \wt \pi  + \OO_\prec(\phi_n^2+\Psi_\Lambda^2)\right) .
\ee
where we define $\wt\pi(z)$ as
\be\label{def smallpir}
\wt\pi(z)^{-1} =\frac{1}{z-1} \begin{pmatrix}  1 - (z-1)m_{1} & - z^{-1/2} \\ - z^{-1/2}   &  1 - (z-1)m_{2} \end{pmatrix} . 
\ee
(Note $\wt\pi$ is actually a random version of $\pi$ in \eqref{def smallpi2}.) On $\Xi$, we have the estimate $\pi^{-1} \wt\pi = 1+ \OO(\Lambda) = 1+ \OO( (\log n)^{-1})$. Hence taking average of the $(1,1)$-th entry of \eqref{Gmumu0} over $\mu$, we get that
\be\label{Gmumu}
\begin{split}
\mathbf 1(\Xi)m_3=\mathbf 1(\Xi)\left[\frac{ 1-(z-1)m_{2}(z)}{z^{-1} - (m_{1}(z)+m_{2}(z)) + (z-1)m_{1}(z)m_{2}(z)} \right. \\
\left. - \left( \wt\pi [Z] \wt\pi\right)_{11} +\OO_\prec\left(\phi_n^2+\Psi_\Lambda^2\right)\right].
\end{split}
\ee
The plugging \eqref{Gii} into \eqref{Gmumu}, and using \eqref{absmc} and the definition of $\Xi$, we can obtain that
\be\label{end_rep}
\begin{split}
\mathbf 1(\Xi)\left[z^{-1}m_3 + m_3\left(\frac{c_1}{m_3} + \frac{c_2}{m_4} \right) +(c_1-1)(z-1)\frac{c_2}{m_4} - 1 \right] \\
\prec \left|\langle Z\rangle_1\right| + \left|\langle Z\rangle_2\right|+\left\|  [Z] \right\|  + \phi_n^2+\Psi_\Lambda^2 .
\end{split}
\ee
Then using \eqref{m3-4} and rearranging terms, we can obtain \eqref{selfcons_improved2}.

Then we prove (\ref{selfcons_lemm21}) and (\ref{selfcons_lemm22}). When $\eta\ge C_0$, by (\ref{Zestimate2}) we have 
$$\max_{i \in \cal I_1\cup \cal I_2} |Z_{i}| + \max_{\mu\in \cal I_3}\|Z_{[\mu]}\|+\Lambda_o \prec \phi_n  . $$
On the other hand, applying \eqref{op G1} and Lemma \ref{SxxSyy} to \eqref{self_Gii} and \eqref{self_Gmu}, we obtain that   
$$ \max_{i \in \cal I_1\cup \cal I_2}  {G^{-1}_{ii}} + \max_{\mu\in \cal I_3}  \|G_{[\mu\mu]}^{-1}\|=\OO(1)$$
with high probability when $\eta\ge C_0$. Together with \eqref{resolvent8} and \eqref{eq_res3}, we get that
\be\label{high_err}  |m_1 -m_1^{[\mu]}|+ |m_2 -m_2^{[\mu]}|+|m_3 - m_3^{(i)}| + |m_4 - m_4^{(j)}|  \prec \Lambda_o^2 \prec \phi_n^2. \ee
Moreover, we still have 
\begin{equation}\label{epsilonL}
\mathbf 1(C_0 \le \eta \le 2C_0) ( |\epsilon_i | +|\epsilon_j | + \|\epsilon_\mu\| ) \prec \phi_n .
\end{equation}
Then going through the previous argument on event $\Xi$, one can see that in order to prove \eqref{selfcons_lemm21} and \eqref{selfcons_lemm22}, it suffices to bound  $m_3^{-1}$, $m_4^{-1}$ and $\|\wt\pi\|$ from above. In particular, it suffices to prove the following bounds: with high probability,
\be\label{largeeta1} 
\mathbf 1(C_0 \le \eta \le 2C_0)\left[|m_3|^{-1} + |m_4|^{-1} \right] \le C  
\ee
and
\be\label{largeeta2} 
\mathbf 1(C_0 \le \eta \le 2C_0) |z^{-1} - (m_{1}+m_{2}) + (z-1)m_{1}m_{2}|^{-1}  \le C  
\ee
for some constant $C>0$.

First using \eqref{op G1} and Lemma \ref{SxxSyy}, we obtain that for $C_0 \le \eta \le 2C_0$, with high probability,
\begin{equation}\label{estimate_m2L0}
|m_1| + |m_2| + |m_3|+|m_4| \le \frac{C}{C_0}, 
\end{equation} 
for some constant $C>0$ that is independent of $C_0$. 
Using the spectral decomposition (\ref{spectral1}) and \eqref{simple obs}, it is easy to see that $\im m_1(z) \ge 0$ and $\im m_2(z) \ge 0$. Hence we have
\be\label{extraz-1} \left| \frac{1}{z-1} - m_1\right| \ge -\im \frac1{z-1} \ge c C_0^{-1}, \quad \left| \frac{1}{z-1} - m_2\right| \ge -\im \frac1{z-1} \ge c C_0^{-1},\ee
for some constant $c>0$ that is independent of $C_0$. Thus we obtain that 
$$  \left|\left(\frac{1}{z-1} - m_1\right) \left(\frac{1}{z-1} - m_2\right) - \frac{1}{z(z-1)^2}\right| \ge c^2 C_0^{-2} - \frac{1}{|z||z-1|^2} \ge \frac12c^2 C_0^{-2} $$
as long as $C_0$ is taken large enough, which then implies \eqref{largeeta2}. Now  by \eqref{def smallpir}, \eqref{largeeta2} and \eqref{estimate_m2L0}, we know that $\wt\pi(z)=\OO(1)$ with high probability. Thus as in \eqref{Gmumu}, we can derive from \eqref{Gmumu0} and \eqref{epsilonL} that 
$$ m_3= \frac{ 1-(z-1)m_{2}(z)}{z^{-1} - (m_{1}(z)+m_{2}(z)) + (z-1)m_{1}(z)m_{2}(z)} +\OO_\prec\left(\phi_n\right).$$
Then using \eqref{extraz-1} and \eqref{largeeta2}, we obtain that with high probability, $|m_3| \ge c$ for some constant $c>0$. Similarly, we can obtain the same bound for $m_4$. This gives \eqref{largeeta1}.
\end{proof}

The following lemma gives the stability of the equation $ f_3(u,z)=0$, where 
\be\nonumber
f_3(u, z) := u^2(z) + \left[ (2c_1 -1)z - c_1+c_2\right]u(z) + c_1(c_1-1)z(z-1) .
\ee
Roughly speaking, it states that if $f(m_{3}(z),z)$ is small and $m_3(\wt z)-m_{3c}(\wt z)$ is small for $\Im\, \wt z \ge \Im\, z$, then $m_{3}(z)-m_{3c}(z)$ is small. For an arbitrary $z\in S(\e)$, we define the discrete set
\begin{align*}
L(z):=\{z\}\cup \{z'\in S(c_0,C_0, \e): \text{Re}\, z' = \text{Re}\, z, \text{Im}\, z'\in [\text{Im}\, z, \e^{-1}]\cap (n^{-100}\mathbb N)\} .
\end{align*}
Thus, if $\text{Im}\, z \ge \e^{-1}$, then $L(z)=\{z\}$; if $\text{Im}\, z<\e^{-1}$, then $L(z)$ is a 1-dimensional lattice with spacing $n^{-100}$ plus the point $z$. 

\begin{lemma}\label{stability}
Fix a constant $\epsilon>0$. The self-consistent equation $f_3(u,z)=0$ is stable on $S(\epsilon)$ in the following sense. Suppose the $z$-dependent function $\delta$ satisfies $n^{-2} \le \delta(z) \le (\log n)^{-1}$ for $z\in S(\epsilon)$ and that $\delta$ is Lipschitz continuous with Lipschitz constant $\le n^2$. Suppose moreover that for each fixed $E$, the function $\eta \mapsto \delta(E+\ii\eta)$ is non-increasing for $\eta>0$. Suppose that $ u_3: S(\epsilon)\to \mathbb C$ is the Stieltjes transform of a measure $ \mu_3$ with $ \mu_3(\R)=\OO(1)$. Let $z\in S(\epsilon)$ and suppose that for all $z'\in L(z)$ we have 
\begin{equation}\label{Stability0}
\left| f_3(z',  u_3)\right| \le \delta(z').
\end{equation}
Then we have
\begin{equation}
\left|u_3(z)-m_{3c}(z)\right|\le \frac{C\delta}{\sqrt{\kappa+\eta+\delta}},\label{Stability1}
\end{equation}
for some constant $C>0$ independent of $z$ and $n$, where $\kappa$ is defined in (\ref{KAPPA}). 
\end{lemma}
\begin{proof}
This lemma can proved with the same method as in e.g. \cite[Lemma 4.5]{isotropic} and \cite[Appendix A.2]{Anisotropic}. The only inputs are the form of the function $f_3$ and Lemma \ref{lem_mbehavior}. 
\end{proof}

Note that by Lemma \ref{stability}, (\ref{selfcons_lemm21}), (\ref{selfcons_lemm22}) and \eqref{m3-4}, we immediately get that
\begin{equation}\label{average_L}
\mathbf 1(C_0 \le \eta \le 2C_0 )|m_\al(z) - m_{\al c}(z)| \prec \phi_n, \quad \al=1,2,3,4.
\end{equation}
From (\ref{Zestimate2}), we obtain the off-diagonal estimate
\begin{equation}\label{offD_L}
\mathbf 1(C_0 \le \eta \le 2C_0 )\Lambda_o(z) \prec \phi_n.
\end{equation}
Plugging \eqref{Zestimate2} and \eqref{average_L} into (\ref{self_Gii}), we get that 
\begin{equation}\label{diag_L1}
\mathbf 1(C_0 \le \eta \le 2C_0 ) \max_{i\in \cal I_1\cup \cal I_2}\left|G_{ii} - \Pi_{ii}\right|  \prec \phi_n.
\end{equation}
Plugging \eqref{Zestimate2} and \eqref{average_L} into  \eqref{self_Gmu}, we obtain that for $ C_0 \le \eta \le 2C_0 $,
\begin{equation}\label{diag_L}
G^{-1}_{[\mu\mu]}(z) =\pi^{-1}(z) + \OO_\prec(\phi_n) \Rightarrow  \max_{\mu\in \cal I_3}\left\|\pi^{-1}\left(G_{[\mu\mu]} - \pi\right)\pi^{-1}\right\|  \prec \phi_n. 
\end{equation}
Starting from these initial resolvent estimates, using a standard continuity (in $z$) argument, the self-consistent estimates \eqref{selfcons_lemm}-\eqref{selfcons_lemm12}, and Lemma \ref{Z_lemma}, we can prove the following weak version of \eqref{entry_law}. 


\begin{lemma}[Weak entrywise local law]\label{alem_weak} 
For any small constant $\e>0$, we have 
\begin{equation} \label{localweakm}
\Lambda(z) \prec \phi_n^{1/2}+(n\eta)^{-1/4},
\end{equation}
uniformly in $z \in S(\epsilon)$.
\end{lemma}
\begin{proof}
One can prove this lemma using a continuity argument as in e.g. \cite[Section 4.1]{isotropic}, \cite[Section 5.3]{Semicircle} or \cite[Section 3.6]{EKYY1}. The key inputs are Lemmas \ref{Z_lemma}, Lemma \ref{lemm_selfcons_weak}, Lemma \ref{stability}, and the estimates (\ref{average_L})-(\ref{diag_L}) in the $C_0 \le \eta \le 2C_0$ case. All the other parts of the proof are essentially the same. 
\end{proof}

To get the strong entrywise local law as in \eqref{entry_law}, we need stronger bounds on $\langle Z\rangle_1$, $\langle Z\rangle_2$ and $[Z]$ in \eqref{def_Zaver}. They follow from the following {\it{fluctuation averaging lemma}}. 

\begin{lemma}[Fluctuation averaging] \label{abstractdecoupling}
Suppose $\Phi$ and $\Phi_o$ are positive, $n$-dependent deterministic functions on $S(\epsilon)$ satisfying $n^{-1/2} \le \Phi, \Phi_o \le n^{-c}$ for some constant $c>0$. Suppose moreover that $\Lambda \prec \Phi$ and $\Lambda_o \prec \Phi_o$. Then for all $z \in S( \epsilon)$ we have
\begin{equation}\label{flucaver_ZZ}
\left|\langle Z\rangle_1\right| + \left|\langle Z\rangle_2\right|+\left\|  [Z] \right\|  \prec   {\Phi _o^2}.
\end{equation}
\end{lemma}
\begin{proof}
 The bound on $\left|\langle Z\rangle_1\right|+\left|\langle Z\rangle_2\right|$ can be proved in the same way as \cite[Theorem 4.7]{Semicircle}. 
The bound on $\left\|[  Z ]\right|$ can be proved in the same way as \cite[Lemma 4.9]{XYY_circular}. 
\end{proof}

Now we give the proof of Proposition \ref{prop_diagonal}.

\begin{proof}[Proof of Proposition \ref{prop_diagonal}]
By Lemma \ref{alem_weak}, the event $\Xi$ holds with high probability. Then by Lemma \ref{alem_weak} and Lemma \ref{Z_lemma}, we can take
\be\label{initial_phio}
\Phi_o =\phi_n + \sqrt{\frac{\im m_{c} + \Phi}{n\eta}} + \frac{1}{n\eta},\quad \Phi= \phi_n^{1/2}+ (n\eta)^{-1/4},
\ee
 in Lemma \ref{abstractdecoupling}. Then (\ref{selfcons_improved2}) gives
$$|f_3(z,m_3)| \prec \phi_n^2 + \frac{ \im m_{c} + \Phi}{n\eta} \lesssim \phi_n^2 + \frac{ \im m_{c} }{n\eta} + \frac{n^{2\tau}}{(n\eta)^2}+ n^{-2\tau}\Phi^2$$
for any fixed constant $\tau>0$. Using Lemma \ref{stability}, we get
\be\label{m2}
\begin{split}
|m_3-m_{3c}| & \prec \min \left\{\phi_n,\frac{\phi_n^2}{\sqrt{\kappa+\eta}}\right\}+\frac{\im m_{c}}{n\eta\sqrt{\kappa+\eta}}+\frac{n^\tau}{n\eta}+n^{-\tau}\Phi \\
&\prec  \min \left\{\phi_n,\frac{\phi_n^2}{\sqrt{\kappa+\eta}}\right\} +\frac{n^\tau}{n\eta} + n^{-\tau}\Phi ,
\end{split}
\ee
where we used $\im m_{c}=\OO(\sqrt{\kappa+\eta})$ by \eqref{Immc} in the second step. With (\ref{selfcons_improved}), \eqref{m3-4} and \eqref{m2}, we get the same bound for $m_\al$, $\al=1,2,3,4$,  
\be\label{m1}
|m_{\al}(z) - m_{\al c}(z)| \prec \min \left\{\phi_n,\frac{\phi_n^2}{\sqrt{\kappa+\eta}}\right\} +\frac{n^\tau}{n\eta} + n^{-\tau}\Phi .
\ee

Plugging \eqref{m1} into \eqref{self_Gii} and using \eqref{erri}, we obtain that 
$$ \max_{i\in \cal I_1 \cup \cal I_2}|G_{ii}-\Pi_{ii}|\prec  \Phi_o  +  \min \left\{\phi_n,\frac{\phi_n^2}{\sqrt{\kappa+\eta}}\right\} +\frac{n^\tau}{n\eta} + n^{-\tau}\Phi \prec \phi_n + \Psi(z) + \frac{n^\tau}{n\eta} + n^{-\tau}\Phi .$$
Similarly plugging \eqref{m1} into \eqref{self_Gmu} and using \eqref{erri}, we obtain that
\begin{equation} \nonumber
\max_{\mu\in \cal I_3}\left\|\pi^{-1} (G_{[\mu\mu]}-\pi )\pi^{-1}\right\| \prec \phi_n + \Psi(z) + \frac{n^\tau}{n\eta} + n^{-\tau}\Phi . 
\end{equation}
Finally by \eqref{Oestimate1}, we have
$$\Lambda_o \prec \phi_n + \Psi(z) + \frac{n^\tau}{n\eta} + n^{-\tau}\Phi . $$
In sum, we obtain a self-improving estimate on $\Lambda$:
$$ \Lambda \prec \Phi \ \Rightarrow \ \Lambda\prec \phi_n + \Psi(z) + \frac{n^\tau}{n\eta} + n^{-\tau}\Phi .$$
Afer $\OO(\tau^{-1})$ many iterations, we obtain that 
$$\Lambda \prec \phi_n + \Psi(z) + \frac{n^\tau}{n\eta}. $$ 
Since $\tau$ can be arbitrarily small, we conclude \eqref{entry_law}.\end{proof}

Finally, we prove the weak averaged local laws in Theorem \ref{thm_local}.
\begin{proof}[Proof of \eqref{aver_in1} and \eqref{aver_out1}]
We can repeat the argument at the beginning of the above proof of Proposition \ref{prop_diagonal}. Taking $\Phi=\Phi_o=\phi_n + \Psi$, (\ref{selfcons_improved2}) and Lemma \ref{abstractdecoupling} give that
$$|f_3(z,m_3)| \prec \phi_n^2 + \Psi^2(z) \lesssim \phi_n^2 + \frac{ \im m_{c} }{n\eta} + \frac{1}{(n\eta)^2}.$$
Using Lemma \ref{stability}, we get
\be\label{m22}
|m_3-m_{3c}|\prec \min \left\{\phi_n,\frac{\phi_n^2}{\sqrt{\kappa+\eta}}\right\}+\frac{\im m_{c}}{n\eta\sqrt{\kappa+\eta}}+\frac{1}{n\eta} \prec \min \left\{\phi_n,\frac{\phi_n^2}{\sqrt{\kappa+\eta}}\right\} +\frac{1}{n\eta}.
\ee
With (\ref{selfcons_improved}), \eqref{m3-4} and \eqref{m22}, we conclude \eqref{aver_in1}. 

For $z\in S_{out}(\e)$, we use Lemma \ref{stability} again to get that 
\be\label{m222}
\begin{split}
|m_3-m_{3c}| &\prec \min \left\{\phi_n,\frac{\phi_n^2}{\sqrt{\kappa+\eta}}\right\}+\frac{\im m_{c}}{n\eta\sqrt{\kappa+\eta}}+\frac{1}{(n\eta)^2 \sqrt{\kappa+\eta}} \\
&\prec \min \left\{\phi_n,\frac{\phi_n^2}{\sqrt{\kappa+\eta}}\right\} +\frac{1}{n(\kappa +\eta)} + \frac{1}{(n\eta)^2\sqrt{\kappa +\eta}},
\end{split}
\ee
where we used the stronger bound $\im m_{c}=\OO(\eta/\sqrt{\kappa+\eta})$ by \eqref{Immc} in the second step. With (\ref{selfcons_improved}), \eqref{m3-4} and \eqref{m222}, we conclude \eqref{aver_out1}. 
%
\end{proof}

\subsection{A centralization argument}\label{subsec_central}
In this subsection, we discuss how to relax the assumptions \eqref{entry_assmX1} and \eqref{entry_assmX2} to the weaker ones in \eqref{entry_assm0} and \eqref{entry_assm1}. First, under the relaxed variance assumption \eqref{entry_assm1}, the only differences from the previous argument in Section \ref{secentryweak} are the equations \eqref{self_Gii} and \eqref{self_Gmu}. More precisely, we now have 
$$\mathbb E_i ( {WG^{\left( i \right)} W^T} )_{ii} = m_{3,4}^{(i)}+\OO(n^{-1-\tau}\|G\|_{\max}) , \quad \text{if } i \in \cal I_{1,2} ,$$
and hence $\e_i$ in  \eqref{self_Gii} will contain an extra error $\OO(n^{-1-\tau}\|G\|_{\max})$. Similarly, the term $\e_\mu$ in \eqref{self_Gmu} will also contain this kind of error. This extra error will lead to an negligible term $ \OO(n^{-1-\tau})$ in all the bounds of Theorem \ref{thm_local}, and hence does not affect our results. 


Then we relax \eqref{entry_assmX1} to \eqref{entry_assm0}. For $X$ and $Y$ satisfying the assumptions in Theorem \ref{lem null}, we write $X=X_1 + \cal E_1$ and $Y=Y_1 + \cal E_2$, where $\cal E_1:= \mathbb EX$ and $\cal E_2:= \mathbb EY$. Then $X_1$ and $Y_1$ are random matrices satisfying  the assumptions in Theorem \ref{lem null} and  \eqref{entry_assmX1}, and $\cal E_1$ and $\cal E_2$ are deterministic matrices such that 
\begin{equation}\label{boundB}
\max_{i\in \cal I_1,\mu\in \cal I_3}|(\cal E_1)_{i\mu}|+ \max_{j\in \cal I_2,\nu\in \cal I_4}|(\cal E_2)_{j\nu}|\le n^{-2-\tau}.
\end{equation}
We denote $G_1(z):=H_1^{-1}(z)$ and $G(z):=\left[H_1(z)+V\right]^{-1}$, where
$$   H_1(\lambda) : = \begin{pmatrix} 0 & \begin{pmatrix}  X_1 & 0\\ 0 &  Y_1\end{pmatrix}\\ \begin{pmatrix} X^T_1 & 0\\ 0 &  Y^T_1\end{pmatrix}  & \begin{pmatrix}  \lambda  I_n & \lambda^{1/2}I_n\\ \lambda^{1/2} I_n &  \lambda I_n\end{pmatrix}^{-1}\end{pmatrix},\quad V:=  \begin{pmatrix} 0 & \begin{pmatrix}  \cal E_1 & 0\\ 0 &  \cal E_2\end{pmatrix}\\ \begin{pmatrix} \cal E^T_1 & 0\\ 0 &  \cal E^T_2\end{pmatrix}  & 0\end{pmatrix}. $$

\begin{lemma}\label{comp_claim} If Theorem \ref{thm_local} holds for $G_1$, then it also holds for $G$. \end{lemma}
\begin{proof}
We expand $G$ using the resolvent expansion
\begin{equation}\label{rsexp1}
G = G_1 - G_1 V G_1 + (G_1 V )^2 G_1 - (G_1 V )^3 G.
\end{equation}
For any deterministic unit vector $\mathbf v \in \mathbb C^{\mathcal I}$, we have
\begin{equation}\nonumber
\begin{split}
\left|\langle \mathbf v, G_1 V G_1\mathbf v\rangle \right| & \le 2\sum_{\mu\in \mathcal I_3\cup \cal I_4}\Big| \sum_{i\in \mathcal I_1\cup \cal I_2} \left(G_1\right)_{\mathbf v i} V_{i\mu}\Big| |\left(G_1\right)_{\mu\mathbf v}|  
 \\
 &\prec \max_{\mu} \Big( \sum_{i \in \mathcal I_1\cup \cal I_2} |V_{i\mu}|^2 \Big)^{1/2} \sum_{\mu\in \mathcal I_3\cup \mathcal I_4}  |\left(G_1\right)_{\mu\mathbf v}| \\
& \prec n^{-1-\tau} \Big(\sum_{\mu\in \mathcal I_3\cup \mathcal I_4}   |\left(G_1\right)_{\mu\mathbf v}|^2\Big)^{1/2} \prec n^{-1-\tau}\eta^{-1/2},
\end{split}
\end{equation}
where in the second step we used (\ref{aniso_law}) for $G_1$ with vectors $\bv$ and $\sum_i V_{i\mu}\mathbf e_{i}$, in the third step the Cauchy-Schwarz inequality and (\ref{boundB}), and in the last step Lemma \ref{lem_gbound0} and \eqref{aniso_law} for $G_1$. Together with a simple application of the polarization identity, we obtain the bound 
\begin{equation}\label{exp1}
\begin{split}
\left|\langle \mathbf v, G_1 V G_1\mathbf w\rangle \right|  \prec n^{-1-\tau}\eta^{-1/2},
\end{split}
\end{equation}
for any deterministic unit vectors $\mathbf v,\mathbf w \in \mathbb C^{\mathcal I}$. With a similar argument, we obtain that
\begin{equation}\label{exp2}
\begin{split}
\left|\langle \mathbf v, (G_1 V )^2 G_1 \mathbf w\rangle \right| \prec n^{-2-2\tau}\eta^{-1}.
\end{split}
\end{equation}
Combining this estimate with the rough bound (\ref{op G1}) for $G$, we get that
\begin{align}
 &\left|\langle \mathbf v, (G_1 V)^3 G\mathbf w\rangle \right| =\Big| \sum_{i\in \cal I_1\cup \cal I_2, \mu\in \mathcal I_3\cup \cal I_4}\left((G_1 V )^2 G_1\right)_{\mathbf v i} V_{i\mu} G_{\mu\mathbf w}\Big| \nonumber\\
 &\qquad \qquad \qquad\quad  \ \ +\Big| \sum_{i\in \cal I_1\cup \cal I_2, \mu\in \mathcal I_3\cup \cal I_4}\left((G_1 V )^2 G_1\right)_{\mathbf v \mu} V_{\mu i} G_{i\mathbf w}\Big| \nonumber\\
 & \prec \eta^{-1} \left[\sum_{\mu}\Big| \sum_{i}\left((G_1 V )^2 G_1\right)_{\mathbf v i} V_{i\mu}\Big|^2+\sum_{i}\Big| \sum_{\mu}\left((G_1 V )^2 G_1\right)_{\mathbf v \mu} V_{\mu i}\Big|^2\right]^{1/2}\nonumber\\
&\prec \left(n^{-2-2\tau}\eta^{-1}\right) \eta^{-1}\Big(\sum_{i,\mu} |V_{i\mu}|^2\Big)^{1/2} \le n^{-2-3\tau}\eta^{-1},\label{exp3}
\end{align}
where we used $\eta \gg n^{-1}$ in the last step. Plugging the estimates (\ref{exp1})-(\ref{exp3}) into (\ref{rsexp1}), we conclude that
\begin{equation}\label{truncate_compare}
\left|\langle \mathbf v, G \mathbf w\rangle - \langle \mathbf v, G_1 \mathbf w\rangle\right| \prec n^{-1-\tau}\eta^{-1/2},
\end{equation}
for any deterministic unit vectors $\mathbf v  \in \mathbb C^{\mathcal I}$. 
This concludes \eqref{aniso_law} and \eqref{aver_in1} for $G(z)$. 

For \eqref{aver_out1}, the bounds \eqref{exp2} and \eqref{exp3} are already good enough. It remains to show that 
\be\label{aver_outsuff}
\left|\frac{1}{n}\sum_{\fa\in \cal I_\al} \left( G_1 V G_1\right)_{\fa\fa} \right| \prec  \frac{1}{n(\kappa +\eta)} + \frac{1}{(n\eta)^2\sqrt{\kappa +\eta}}, \quad \al=1,2, 3,4,.
\ee
Using Lemma \ref{lem_gbound0}, we obtain that for $\al=1,2,$
\begin{align*}
& \left|\frac{1}{n}\sum_{i\in \cal I_\al} \left( G_1 V G_1\right)_{ii} \right|  \prec n^{-2-\tau}\sum_{j\in \cal I_1\cup \cal I_2, \mu \in \cal I_3\cup \cal I_4} \frac1n\sum_{i\in \cal I_\al} |(G_1)_{i j}(G_1)_{ \mu i}| \\
& \prec n^{-\tau} \max_{\mu \in \cal I_3\cup \cal I_4}\left( \frac1n + \frac{\im\left(\cal G_L \right)_{jj}+\im\left(\cal G_R \cal U^T\right)_{\mu\mu} }{n\eta} \right) \prec n^{-1-\tau} + n^{-\tau} \frac{\im m_c + \phi_n + \Psi(z)}{n\eta} \\
&\prec n^{-\tau} \left( \phi_n^2 + \Psi^2(z)\right) \prec  \frac{n^{-\tau}}{n(\kappa +\eta)} + \frac{n^{-\tau}}{(n\eta)^2\sqrt{\kappa +\eta}},
\end{align*}
where in the third step we used \eqref{aniso_law} for $G_1$ and \eqref{simpleU}. The proof for the $\al=3,4$ case is similar. This concludes \eqref{aver_out1}.
\end{proof}

\section{Proof of Theorem \ref{thm_local}: the anisotropic local law}\label{sec_aniso}
To conclude Theorem \ref{thm_local}, it remains to prove the anisotropic local law \eqref{aniso_law}. For any vector $\bu\in \C^{\cal I}$ and $\mu\in \cal I_3$, we denote
$u_{[\mu]}:=\begin{pmatrix}u_\mu \\ u_{\overline \mu} \end{pmatrix}$. By the entrywise local law \eqref{entry_law}, we have that for deterministic unit vectors $\bu,\bv\in \C^{\cal I}$,
\begin{equation}\label{eq_iso_1}
\begin{split}
\left|\langle \bu , (G(z)-\Pi(z))\bv\rangle\right| &\prec \phi_n + \Psi(z) + \Big| \sum_{i\ne j \in \cal I_1 \cup \cal I_2} \overline u_i  G_{ij} v_j  \Big| +  \Big| {\sum_{\mu \ne \nu\in \cal I_3} {{{u}^*_{\left[ \mu \right]}} {G_{\left[ \mu\nu\right ]}}{{v}_{\left[ \nu \right]}}} } \Big| \\
& + \Big| {\sum_{i\in \cal I_1\cup \cal I_2,\mu \in \cal I_3 } \overline u_i G_{i,\left[ \mu\right ]} {v}_{\left[ \mu \right]} }  \Big| +  \Big| {\sum_{i\in \cal I_1\cup \cal I_2,\mu \in \cal I_3 }{u}_{\left[ \mu \right]}^* G_{\left[ \mu\right ],i}v_i  }  \Big|.
\end{split}
\ee
Note that applying the entrywise local law naively, one can only get that
\begin{equation*}
\left|\langle \bu , (G(z)-\Pi(z))\bv\rangle\right| \prec (\phi_n +\Psi(z)) \|\mathbf u\|_1 \|\mathbf v\|_1 \le n(\phi_n +\Psi(z)),
\end{equation*}
using $\|\mathbf u\|_1 \le n^{1/2} \|\mathbf u\|_2$ and $\|\mathbf v\|_1 \le n^{1/2} \|\mathbf v\|_2$. To get \eqref{aniso_law}, we need to explore the cancellations (due to the random signs of the $G$ entries) in the four sums on the right hand side of \eqref{eq_iso_1}. 


We can simplify the problem a little bit. We first notice that by polarization identity of inner products, it suffices to take $\bu=\bv$ in \eqref{eq_iso_1}. Moreover, since $G$ is symmetric, the last two terms on the right hand side of \eqref{eq_iso_1} can be bounded in the same way. Then with Markov's inequality, it suffices to prove the following lemma. With Lemma \ref{comp_claim}, it suffices to assume \eqref{entry_assmX1} for $X$ and $Y$.

\begin{lemma}\label{iso_lemm_1}
Suppose \eqref{entry_assmX1} and \eqref{entry_law} hold. Let $\bv\in \C^{\cal I}$ be any deterministic unit vector. Then for any $a\in \mathbb N$, we have the following bounds: 
\begin{align}
& \mathbb E\Big| \sum_{i\ne j \in \cal I_1 \cup \cal I_2} \overline v_i  G_{ij} v_j  \Big|^{2a} \prec \Phi^{2a}; \label{eqiso1} \\
& \mathbb E\Big| {\sum_{\mu \ne \nu\in \cal I_3} {{{v}^*_{\left[ \mu \right]}} {G_{\left[ \mu\nu\right ]}}{{v}_{\left[ \nu \right]}}} } \Big|^{2a} \prec \Phi^{2a}; \label{eqiso2}  \\
&\mathbb E \Big| {\sum_{i\in \cal I_1\cup \cal I_2,\mu \in \cal I_3 } \overline v_i G_{i,\left[ \mu\right ]} {v}_{\left[ \mu \right]} }  \Big|^{2a} \prec \Phi^{2a}.\label{eqiso3} 
\end{align}
Here we denote $\Phi:= \phi_n +\Psi(z)$ for simplicity.
\end{lemma}

The proof of Lemma $\ref{iso_lemm_1}$ is based on a polynomialization method developed in \cite[section 5]{isotropic}. We first give the proof of \eqref{eqiso1} in Section \ref{sec easiest}, which is the easiest, and then give the proof of \eqref{eqiso2} in Section \ref{sec hardest}, which is the hardest. The proof of \eqref{eqiso3} is an easier version of \eqref{eqiso2}, and  will be omitted. 

\subsection{Proof of \eqref{eqiso1}}\label{sec easiest}
For the proof of \eqref{eqiso1}, we will adopt an argument in \cite[Appendix A.4]{XYY_VESD}. 
Recall that in \eqref{conditionA3}, we assumed that the $X$ and $Y$ entries have finite third moments. Together with the bounded support condition, we get
\begin{equation}\label{moment_n3}
\mathbb E |x_{i\mu}|^n \prec \phi_n^{n-3}n^{-3/2},\quad \mathbb E |y_{j\nu}|^n \prec \phi_n^{n-3}n^{-3/2}, \ \  i\in \mathcal I_1, \ \ j\in \mathcal I_2, \ \  \mu \in \mathcal I_3, \ \  \nu \in \mathcal I_4.
\end{equation}
Note that we have a stronger fourth moment assumption in (\ref{conditionA3}), but it is not necessary for the proof in this section. 

We first rewrite the product in (\ref{eqiso1}) as
\begin{align}\label{expandiso1}
 \Big| {  \sum\limits_{i \ne j} {{\overline v_i} {G_{ij}} {v_j}} } \Big|^{2a} =& \sum\limits_{ {i_k \ne j_k} \in {{\cal I}_1}}\prod\limits_{k = 1}^{a} {{\overline v_{i_k}} {G_{i_kj_k}} {v_{j_k}}}  \cdot \prod\limits_{k = a + 1}^{2a} \overline {{\overline v_{i_k}} {G_{i_kj_k}} {v_{j_k}}}.
\end{align}
To organize the sum over indices, we consider all possible partitions of the indices, such that two indices always take the same value if they are in the same partition, and different values otherwise. We use symbol-to-symbol functions to represent the partitions: a partition $\Gamma$ denotes a map 
$$\Gamma: \{i_1, \cdots, i_{2a}, j_1,\cdots, j_{2a}\} \to L(\Gamma), \quad L(\Gamma)=(b_1, \cdots, b_{n(\Gamma)}),$$
where $\Gamma^{-1}(b_k)$ is an equivalence class of the partition, $n(\Gamma)$ is the number of equivalence classes, and $b_k$ are indices taking values in $\cal I_1\cup \cal I_2$. Then we can write \eqref{expandiso1} as
\begin{align*}
\sum_\Gamma \sum_{b_1,...,b_{n(\Gamma)}}^*\prod\limits_{k = 1}^{a} {{\overline v_{\Gamma(i_k)}} {G_{\Gamma(i_k)\Gamma(j_k)}} {v_{\Gamma(j_k)}}}  \cdot \prod\limits_{k = a + 1}^{2a} \overline {{\overline v_{\Gamma(i_k)}} {G_{\Gamma(i_k)\Gamma(j_k)}} {v_{\Gamma(j_k)}}},
\end{align*}
where $\Gamma$ ranges over all the partitions, and $\sum^*$ denotes the summation subject to the condition that $b_1,\ldots, b_{n(\Gamma)}$ all take distinct values and $\Gamma(i_k) \ne \Gamma(j_k)$ for all $k$. Since the number of such partitions $\Gamma$ is finite and depends only on $a$, to prove \eqref{eqiso1} it suffices to show that for any fixed $\Gamma$,
\begin{equation}\label{temp1}
\mathbb E\sum_{b_1,...,b_{n(\Gamma)}}^*\prod\limits_{k = 1}^{a} {{\overline v_{\Gamma(i_k)}} {G_{\Gamma(i_k)\Gamma(j_k)}} {v_{\Gamma(j_k)}}}  \cdot \prod\limits_{k = a + 1}^{2a} \overline {{\overline v_{\Gamma(i_k)}} {G_{\Gamma(i_k)\Gamma(j_k)}} {v_{\Gamma(j_k)}}}\prec \Phi^{2a}.
\end{equation}
We abbreviate 
$$P(b_1,...,b_{n(\Gamma)}) := \prod\limits_{k = 1}^{a} { { G_{\Gamma(i_k)\Gamma(j_k)}} }  \cdot \prod\limits_{k = a + 1}^{2a} \overline { {G_{\Gamma(i_k)\Gamma(j_k)}} }.$$
For simplicity, we shall omit the overline for complex conjugate in the following proof. In this way, we can avoid a lot of immaterial notational complexities that do not affect the proof.

For $k=1,...,n(\Gamma)$, we denote $\deg(b_k, P):= |\Gamma^{-1}(b_k)|$, which is the number of times that $b_k$ appears as an index of the $G$ entries in $P$. We define $h:=\#\{ k : \deg(b_k, P) = 1\}$, i.e. $h$ is the number of $b_k$'s that only appear once in the indices of $P$. Without loss of generality, we assume these $b_k$'s are $b_1,...,b_h$. These indices are the ones that cause the main trouble: the sum of $b_k$, $1\le k \le h$, in \eqref{temp1} contributes a factor $\sum_{b_k}|v_{b_k}|$, which can be of order ${n}^{1/2}$ as discussed above. However, we can obtain an extra $n^{-1/2}$ factor from the $G$ entries with indices $b_k$, $1\le k \le h$. Heuristically, suppose that $b_1=i$, there is an entry $G_{ij}$ in $P$, and all the other $G$ entries are independent of the entries in the $i$-th row and column of $H$. Then using \eqref{resolvent3} we get
$$\mathbb E_i G_{ij}=G_{jj}^{\left( i \right)}  \mathbb E_i \left[(G_{ii} -m ) \left( {WG^{\left( {ij} \right)} W^T } \right)_{ij}\right], \quad i\ne j. $$ 
Recalling \eqref{self_Gii}, if we replace $G_{ii} -m$ with the leading term $\cal Z_i$, then 
$$ \mathbb E_i \left[Z_i \left( {WG^{\left( {ij} \right)} W^T } \right)_{ij}\right] =  -G^{(i)}_{\mu\mu}\sum_{\mu}\left(\mathbb E_iW_{i\mu}^3\right) (G^{(ij)}W^T)_{\mu j}\prec n^{-1/2} \Phi ,$$
where we used \eqref{moment_n3} and the fact that $(G^{(ij)}W^T)_{\mu j}$ has the same order as the $G^{(i)}_{\mu j}$ entry by \eqref{resolvent6}. In general, one can expand $G_{ii} -m$ using the Taylor expansion of \eqref{self_Gii}. It is easy to see that each term in the expansion contains even number of $W_{i\star}$ entries. Together with the $W_{i\star}$ entry in $( {WG^{\left( {ij} \right)} W^T } )_{ij}$, there cannot be a perfect pairing of all of them, so we obtain an extra $n^{-1/2}$ factor due to the loss of a half free index. Finally, even without the exact independence, we know that the other $G_{\fa\fb}$ entries only have weak correlations with the entries in the $i$-th rows and columns of $H$ if $\fa,\fb\ne i$. This fact will be explored using resolvent expansions in Lemma \ref{lemm_resolvent} as in Definition \ref{defn_stringoperator0} below.

\begin{claim}\label{isoclaim1}
We have
\begin{equation}\label{main_bound}
\left|\mathbb E P\right|\prec n^{-h/2}\Phi^{2a}.
\end{equation}
\end{claim}

With this claim, we can complete the proof of \eqref{eqiso1}.
\begin{proof}[Proof of \eqref{eqiso1}]
Note that by $\|\mathbf v\|_2=1$ and Cauchy-Schwarz inequality, we have $\sum_i |v_i| \le \sqrt{n}$ and $\sum_i |v_i|^n \le 1$ for $n\ge 2$. Then if (\ref{main_bound}) holds, we can bound the left hand side of (\ref{temp1}) by 
$$n^{-h/2}\Phi^{2a}\prod_{k=1}^{n(\Gamma)}\sum_{b_k}|v_{b_k}|^{\deg(b_k, P)}\le n^{-h/2}\Phi^{2a}(\sqrt{n})^{h} \le C\Phi^{2a},$$
which further concludes \eqref{eqiso1}. 
\end{proof}

It remains to prove Claim \ref{isoclaim1}. We define the $S$ variables as 
\begin{equation}\label{S_var}
S_{ij} := (WG^{(L)}W^T)_{ij}, \quad i,j \in \cal I_1\cup \cal I_2,
\end{equation} 
where $L:=\{b_1,...,b_{n(\Gamma)}\}$. With the entrywise local law \eqref{entry_law}, \eqref{resolvent3} and \eqref{self_Gii}, we have that 
$$|S_{ij}-c_{\al}^{-1}m_{\al c}\delta_{ij}|\prec \Phi ,\quad  i,j \in \cal I_\al,\quad \al=1,2. $$ 
Our first step is to keep expanding the $G$ entries in $P$ using the resolvent expansions in Lemma \ref{lemm_resolvent}, until each monomial either consists of $S$ variables only or has sufficiently many off-diagonal terms. 
To perform the resolvent expansion in a systematic way, we introduce the following notions of {\it{string}} and {\it{string operator}}. 
\begin{definition}[Strings]\label{def string}
Let $\mathfrak A$ be the alphabet containing all symbols that will appear during the expansion: 
$$\mathfrak A=\{G^{(J)}_{kl}: J\subset L, k,l\in L\} \cup \{(G^{(J)}_{kk})^{-1}: J\subset L, k \in L\} \cup \left\{S_{kl}:k,l\in L\right\}.$$ 
We define a string $\mathbf s$ to be a concatenation of the symbols from $\mathfrak A$, and we use $\left\llbracket\bf s\right\rrbracket$ to denote the random variable represented by $\mathbf s$. 
We denote an empty string by $\emptyset$ with value $\left\llbracket\emptyset\right\rrbracket = 0$. Here we need to distinguish the difference between a string $\mathbf s$ and its value $\left\llbracket\bf s\right\rrbracket$. For example, $``G^{(L\setminus \{i,j\})}_{ij}"$ and $``G^{(L\setminus \{i,j\})}_{ii}G_{jj}^{(L\setminus \{j\})}S_{ij}"$ are different strings, but they represent the same random variable by (\ref{resolvent3}). 
\end{definition}

We shall say $G^{(J)}_{kl}$ (resp. $(G^{(J)}_{kk})^{-1}$) is maximally expanded if $J\cup \{k,l\}=L$ (resp. $J\cup\{k\}=L$). Also the $S$ variables are always maximally expanded.
A string $\mathbf s$ is said to be maximally expanded if all of its symbols are maximally expanded.
We shall call $G^{(J)}_{kl}$ and $S_{kl}$ off-diagonal symbols if $k\ne l$, and all the other symbols are diagonal. Note that by the local law \eqref{entry_law}, 
we have $\left\llbracket\mathbf a_o\right\rrbracket \prec \Phi$ if $\mathbf a_o$ is an off-diagonal symbol. 
We use ${\cal F}_{n{\text{-}}max}(\mathbf s)$ and ${\cal F}_{\rm{off}}(\mathbf s)$ to denote the number of non-maximally expanded symbols and the number of off-diagonal symbols in string $\mathbf s$, respectively.

\begin{definition}[String operators]\label{defn_stringoperator0}
We define the following operators. 

\begin{itemize}
\item[(i)] We define the operator $\tau_0$ acting on a string $\bf s$ in the following way. Find the first non-maximally expanded symbol in the $\mathbf s$, if $G^{(J)}_{ij}$ is found, replace it with $G^{(J\cup \{k\})}_{ij}$ for the first $k$ in $L\setminus (J\cup \{i,j\})$; if $(G^{(J)}_{ii})^{-1}$ is found, replace it with $(G^{(J\cup\{k\})}_{ii})^{-1}$ for the first $k\in L\setminus (J\cup \{i\})$; if neither is found, set $\tau_0(\bf s) = \bf s$ and we say that $\tau_0$ is trivial for $\bf s$.

\item[(ii)] We define the operator $\tau_1$ acting on a string $\bf s$ in the following way. Find the first non-maximally expanded symbol in the $\mathbf s$, if $G^{(J)}_{ij}$ is found, replace it with $G^{(J)}_{ik}(G^{(J)}_{kk})^{-1}G^{(J)}_{kj}$ for the first $k$ in $L\setminus (J\cup \{i,j\})$; if $(G^{(J)}_{ii})^{-1}$ is found, replace it with 
$$-G_{ik}^{(J)} G_{ki}^{(J)} (G^{(J)}_{ii})^{-1} (G_{ii}^{(J\cup \{k\})})^{-1} (G_{kk}^{(J)})^{-1}$$ for the first $k\in L\setminus (J\cup \{i\})$; if neither is found, set $\tau_1(\bf s) =\emptyset$ and we say that $\tau_0$ is null for $\bf s$.

\item[(iii)] Define the operator $\rho$ acting on a string $\bf s$ in the following way. Replace each maximally expanded off-diagonal $G^{(L\setminus \{i,j\})}_{ij}$ in $\bf s$ with $G_{ii}^{(L\setminus \{i,j\})}G_{jj}^{(L\setminus \{j\})} S_{ij}$.
\end{itemize}
\end{definition}
By Lemma \ref{lemm_resolvent}, it is clear that for any string $\bf s$,
\begin{equation}\label{resolvent_string}
\llbracket\tau_0(\mathbf s)\rrbracket + \llbracket\tau_1(\mathbf s)\rrbracket = \llbracket\bf s\rrbracket, \quad \llbracket\rho(\mathbf s)\rrbracket = \llbracket\mathbf s\rrbracket.
\end{equation}
Moreover, a string $\mathbf s$ is trivial under $\tau_0$ and null under $\tau_1$ if and only if $\mathbf s$ is maximally expanded. Given a string $\bf s$, we abbreviate ${\mathbf s}_0 := \tau_0(\mathbf s)$ and ${\bf s}_1 := \rho(\tau_1(\bf s))$. Then by (\ref{resolvent_string}) we have
\begin{equation}\label{resolvent_string2}
\sum_{|w|=m}\llbracket{\mathbf s}_w\rrbracket = \llbracket\mathbf s\rrbracket,
\end{equation}
where $w=w(1)w(2)\ldots w(m) $ with $w(i)\in \{0,1\}$ ranges over all binary sequences $w$ with length $|w|=m$, and we used the notation 
$${\mathbf s}_{w}:=\rho^{w(m)}\tau_{w(m)}\ldots \rho^{w(2)}\tau_{w(2)}\rho^{w(1)}\tau_{w(1)}(\mathbf s), \ \ \text{ where }\rho^0\equiv 1.$$


\begin{lemma}[Lemma 5.9 of \cite{isotropic}]\label{iso_lem_3} 
Consider the string $\mathbf s = ``P(b_1,...,b_{n(\Gamma)}) "$. Fix any $l_0\in \N$. There exists a constant $K(a,l_0)\in \N$ depending on $a$ and $l_0$ only such that the following property holds. For any binary sequence $w$ with $|w| = K(a,l_0)$ and $\mathbf s_w \ne \emptyset$, either ${\cal F}_{\rm{off}}(\mathbf s_{w})\ge l_0$ or $\mathbf s_{w}$ is maximally expanded. 
\end{lemma}

Let $ \omega>0$ be a constant such that $\Phi\le  n^{-\omega/2}$. If we choose $l_0 =\lceil  (h\omega^{-1} +2a) \rceil$, then 
\begin{equation}\label{0off_bound} \sum_{|w|=K(a,l_0)}\llbracket\mathbf s_w\rrbracket \cdot \mathbf 1({\cal F}_{\rm{off}}(\mathbf s_{w})\ge l_0)\prec 2^{K(a,l_0)}\Phi^{l_0}\prec n^{-h/2}\Phi^{2a}.\ee
Then by Lemma \ref{iso_lem_3}, to prove Claim \ref{isoclaim1}  it suffices to show that
\begin{equation}\label{maximal_bound}
\left|\mathbb E \llbracket\mathbf s_w\rrbracket\right| \prec n^{-h/2}\Phi^{2a}
\end{equation}
for any maximally expanded string $\mathbf s_w$ with $|w|=K(a,l_0)$. Note that the maximally expanded string $\mathbf s_w$ thus obtained consists only of $S$ symbols and diagonal $G$ symbols $G_{ii}^{(L\setminus\{i\})}$ and $ (G_{ii}^{(L\setminus \{i\})})^{-1}$. By (\ref{resolvent1}), we can replace $(G_{kk}^{(l)})^{-1}$ with
$(G_{ii}^{(L\setminus\{i\})})^{-1} =  -S_{ii} -z n^{-10}.$
Then as in \eqref{self_Gii}, for $i\in \cal I_\al$, $\al=1,2$, we can expand $G_{ii}^{(L\setminus\{i\})}$ as, 
\begin{align}
G_{ii}^{(L\setminus\{i\})}&=\frac{1}{- m_{(\al+2) c} + \left(m_{(\al+2)c} - S_{ii}-zn^{-10}\right)} \nonumber\\
&= \frac{-1}{m_{(\al+2)c}}\sum_{k=0}^{K(a,l_0)}\left(\frac{m_{(\al+2)c} - S_{ii}-zn^{-10}}{m_{(\al+2)c}}\right)^k+\OO_{\prec}(n^{-h/2}\Phi^{2a}). \nonumber 
\end{align}
We apply the above expansions to the $G$ symbols in $\mathbf s_w$, disregard the sufficiently small tails, and denote the resulting polynomial (in terms of the symbols $S_{ij}$) by $P_w$. Then $P_w$ can be written as a finite sum of maximally expanded strings (or monomials) consisting of the $S$ symbols only. Moreover, the number of such monomials depends only on $a$ and $l_0$.
Hence it suffices to show that for any such monomial $M_w$, we have
\begin{equation}\label{maximal_bound2}
|\mathbb E \llbracket M_w \rrbracket | \prec n^{-h/2}\Phi^{2a}.
\end{equation}


Recall that in the initial string $P$, we assume the following setting
\begin{equation}\label{deg1}
\sum_{k=1}^n\deg(b_k,P)=4a, \ \ \text{and} \ \ \deg(b_k,P)=1,\ \text{ for }k=1,...,h.
\end{equation}
Now in $M_w$, let $\deg_o(b_k, M_w)$ denotes the number of times that $b_k$ appears as an index of the \textit{off-diagonal} $S$ variables in $M_w$. Then it is easy to verify the following relations: 
\begin{equation}\label{parity2}
\deg_o(b_k, M_w) \ge \deg(b_k, P), \quad \deg_o(b_k, M_w) - \deg(b_k, P)=0 \mod 2,
\end{equation}
where the first inequality is trivial, and the second identity follows from the simple fact that none of the above expansions changes the parity of the index $b_k$. 

Suppose $M_w$ takes the form
\begin{align*}
M_w & = \prod_{j=1}^{K_w} S_{b_{k_j}b_{l_j}} =\sum_{\substack{\mu_j,\nu_j\in\mathcal I_2}}\prod_{j=1}^{K_w} W_{b_{k_j}\mu_j}G^{(L)}_{\mu_j\nu_j}W^T_{\nu_j b_{l_j}}\\
&=\sum_{\wt \Gamma}\sum_{\wt b_1,...,\wt b_{n(\wt\Gamma)}}^*\prod_{j=1}^{K_w} W_{b_{k_j}\wt\Gamma(\mu_j)}G^{(L)}_{\wt\Gamma(\mu_j)\wt\Gamma(\nu_j)}W_{b_{l_j}\wt\Gamma(\nu_j) }
\end{align*}
where $K_w$ is the number of $S$-variables in $M_\omega$, $\wt \Gamma$ ranges over all partitions of the set of the labels $\{\mu_1,...,\mu_{K_w},\nu_1,...,\nu_{K_w}\}$, $\{\wt b_1,...,\wt b_{n(\wt\Gamma)}\}$ denotes the set of distinct equivalence classes for a particular $\wt\Gamma$, and $\sum^*$ denotes the summation subject to the condition that $\wt b_k$'s all take distinct values. Here again $\wt\Gamma(\cdot)$ is regarded as a symbolic mapping from the set of labels to the set of equivalence classes. Note that the number of partitions depends only on $K_w$. For a fixed partition $\wt\Gamma$, we denote 
$$R(\wt b_1,...,\wt b_{n(\wt\Gamma)};\wt\Gamma):= \prod_{j=1}^{K_\omega} W_{b_{k_j}\wt\Gamma(\mu_j)}G^{(L)}_{\wt\Gamma(\mu_j)\wt\Gamma(\nu_j)}W_{ b_{l_j}\wt\Gamma(\nu_j)}.$$
Then to prove (\ref{maximal_bound2}), it suffices to show that
\begin{equation}\label{main_bound3}
\left|\mathbb E R(\wt b_1,...,\wt b_{n(\wt\Gamma)};\wt\Gamma) \right|\prec n^{-n(\wt\Gamma)-h/2}\Phi^{2a}.
\end{equation}
for any partition $\wt\Gamma$. 

To facilitate the description of the proof, we introduce the graphical notations. 
We use a connected graph $({\rm V},{\rm E})$ to represent $R$,
where the vertex set ${\rm V}$ consists of black vertices $b_1,\ldots, b_{n(\Gamma)}$ and white vertices $\wt b_1,\ldots,\wt b_{n(\wt\Gamma)}$, and the edge set ${\rm E}$ consists of $(k,\alpha)$ edges representing $W_{b_{k}\wt b_{\alpha}}$ and $(\alpha,\beta)$ edges representing $G_{\wt b_{\alpha}\wt b_{\beta}}$. We denote 
$$e_{k\alpha}:=\text{number of }(k, \alpha)\text{ edges in }R, \ \ d_\alpha:=\text{number of }(\alpha,\alpha)\text{ edges in }R,$$
and
$$e_{k\alpha}^{(o)}:=\text{number of }(k, \alpha)\text{ edges that are from off-diagonal }S \text{ in } M_w.$$
Due to the mean zero condition \eqref{entry_assmX1}, to attain a nonzero expectation we must have
\begin{equation}\label{ka3}
e_{k\alpha}=0 \ \text{ or } \  e_{k\alpha}\ge 2 \ \ \text{ for all } k,\alpha.
\end{equation}
We also have that, by definition,
\begin{equation}\label{off_e}\sum_\alpha e_{k\alpha}^{(o)} = \deg_o(b_k, M_w)\end{equation}

By (\ref{deg1}), (\ref{parity2}) and (\ref{ka3}), there exist edges $(1,\alpha_1),...,(h,\alpha_h)$ such that $e_{k\alpha_k}$ is odd and $e_{k\alpha_k}\ge 3$, $1\le k \le h$. Let $H:=\{(1,\alpha_1),...,(h,\alpha_h)\}$ be the set of these edges. Denote by $F$ the set of $(k,\alpha)$ edges such that $e_{k\alpha} \ge 2$ and $(k,\alpha)\notin H$. Denote
$$s_\alpha :=\sum_{k=1}^{n(\Gamma)} e_{k\alpha},\quad h_{k\alpha}:=\mathbf 1_{(k,\alpha)\in H},\quad h_\alpha:=\sum_{k=1}^{n(\Gamma)} h_{k\alpha},\quad  f_\alpha:=\sum_{k=1}^{n(\Gamma)} \mathbf 1_{(k,\alpha)\in F},$$
for all $k=1,...,n(\Gamma)$ and $\alpha=1,...,n(\wt\Gamma)$. 
From the above definitions, it is easy to see that $s_\alpha \ge 2$ and $h_\alpha +f_\alpha>0$ (since the classes $\wt b_\alpha$ are nontrivial), $s_\alpha \ge 2d_\alpha$ (since one $(\al,\al)$ edge corresponds to two $(k,\al)$ edges), and
\begin{equation}\label{sum_h}
\sum_{\alpha}h_{k\alpha} = \mathbf 1(1\le k \le h), \quad \sum_\alpha h_\alpha=h.
\end{equation}


Since there are totally $\frac{1}{2}\sum_{\alpha}s_{\al} - d_\alpha$ off-diagonal $G$ edges in $R$, by (\ref{entry_law}) and (\ref{moment_n3}) we have
\begin{align*}
|\mathbb E R| &\prec\prod_{\alpha=1}^{n(\wt\Gamma)}\Big(\Phi^{\frac12s_\al-d_\alpha}\prod_{k=1}^{n(\Gamma)}  \mathbb E |W_{b_k\wt b_\alpha}|^{e_{k\alpha}} \Big)\nonumber\\
&\prec \prod_{\alpha=1}^{n(\wt\Gamma)}\Phi^{\frac12s_\alpha - d_\alpha}\Big(\prod_{(k,\alpha)\in H} \phi_n^{e_{k\alpha}-3}n^{-3/2}\Big)\Big(\prod_{(k,\alpha)\in F} \phi_n^{e_{k\alpha}-2}n^{-1}\Big) =: \prod_{\alpha=1}^m R_\alpha.
\end{align*}
Now we consider the following four cases for $R_\alpha$. The arguments essentially are the same as the ones in \cite[Appendix A.4]{XYY_VESD}, and we repeat them for reader's convenience.

\vspace{5pt}

\noindent{\bf Case 1:} $d_\alpha=0$. In this case we have
\begin{align*}
R_\alpha &\prec  \Phi^{s_\alpha/2}n^{-(h_\alpha + f_\alpha)-h_\alpha/2} \prec\Phi^{s_\alpha/2}n^{-1 -h_\alpha/2} \prec \Phi^{\sum_{k=1}^h h_{k\alpha}/2+\sum_{k=h+1}^n e_{k\alpha}^{(o)}/2}n^{-1 -h_\alpha/2}
\end{align*}
where in the second step we used $h_\al+f_\al > 0$, and 
in the third step we used
$$s_\alpha \ge\sum_k e_{k\alpha}^{(o)} \ge \sum_{k=1}^h h_{k\alpha} +\sum_{k=h+1}^n e_{k\alpha}^{(o)},$$
where we used that $e_{k\alpha}^{(o)}\ge h_{k\alpha}$ for $1\le k \le h$ (recall that if $(k,\alpha)\in H$, then $e_{k\alpha_k}$ is odd and hence one of the edges must come from the off-diagonal $S$).
 
\vspace{5pt}

\noindent{\bf Case 2:} $d_\alpha \ne 0$, $h_\alpha =1$ and $f_\alpha=0$. Then there is only one $k$ such that $e_{k\alpha}>0$ and $s_\alpha =e_{k\alpha}$ is odd. Hence we have $s_\al/2 \ge d_\al + 1/2$ and we can bound $R_\alpha$ as
\begin{align*}
R_\alpha &\prec\Phi^{\frac{1}{2}s_\alpha-d_\alpha}n^{-(h_\alpha + f_\alpha)-h_\alpha/2}\prec \Phi^{1/2}n^{-1 -h_\alpha/2} = \Phi^{\sum_{k=1}^h h_{k\alpha}/2+\sum_{k=h+1}^n e_{k\alpha}^{(o)}/2}n^{-1 -h_\alpha/2} ,
\end{align*}
where in the last step we used  $1 = \sum_{k=1}^h h_{k\alpha} +\sum_{k=h+1}^n e_{k\alpha}^{(o)} , $ since all the summands except one $h_{k\alpha}$ are $0$. 

\vspace{5pt}

\noindent{\bf Case 3:} $d_\alpha\ne 0$, $h_\alpha=0$ and $f_\alpha=1$. Then there is only one $k$ such that $e_{k\alpha}>0$ and $s_\alpha=e_{k\alpha}$. Thus the  $(\alpha,\alpha)$ edges are expanded from the diagonal $S$ variables, 
which implies $ s_\alpha - 2d_\alpha = e_{k\alpha}^{(o)}$. Then we can bound 
\begin{align*}
R_\alpha & \prec\Phi^{\frac{1}{2}s_\alpha - d_\alpha}n^{-(h_\alpha + f_\alpha)-h_\alpha/2}= \Phi^{\sum_ke_{k\alpha}^{(o)}/2}n^{-1 -h_\alpha/2}  \prec \Phi^{\sum_{k=1}^h h_{k\alpha}/2+\sum_{k=h+1}^n e_{k\alpha}^{(o)}/2}n^{-1 -h_\alpha/2} 
\end{align*}
where we used $e_{k\alpha}^{(o)}\ge h_{k\alpha}$ for $1\le k \le h$ in the last step.

\vspace{5pt}

\noindent{\bf Case 4:} $d_\alpha\ne 0$ and $h_\alpha+f_\alpha \ge 2$. Then using $s_\alpha \ge 2d_\alpha$, $\phi_n\prec\Phi $ and $n^{-1/2}\prec\Phi$, we get that
\begin{align*}
R_\alpha &\prec \prod_{(k,\alpha)\in H} \Phi^{e_{k\alpha} - 3 }n^{-3/2}\prod_{ (k,\alpha)\in F } \Phi^{e_{k\alpha} -2}n^{-1}\prec \prod_{(k,\alpha)\in H} \Phi^{e_{k\alpha} - 2 }n^{-1}\prod_{ (k,\alpha)\in F } \Phi^{e_{k\alpha} -1}n^{-1/2} \\
& \le \Phi^{\sum_{k=1}^h h_{k\alpha}/2+\sum_{k=h+1}^n e_{k\alpha}^{(o)}/2}n^{-1 -h_\alpha/2}
\end{align*}
where in the last step we used $e_{k\alpha} \ge h_{k\alpha} + 2$ for $(k,\al)\in H$ and $e_{k\alpha}\ge 2$ for $(k,\al)\in F$.

 



Combining the above four cases, we obtain that 
\begin{align*}
|\mathbb ER| = \prod_{\alpha=1}^{n(\wt\Gamma)} R_\alpha \prec n^{-n(\wt\Gamma)}n^{-\frac{1}{2}\sum_\alpha h_\alpha} \Phi^{\sum_\alpha \left(\sum_{ k=1}^h h_{k\alpha}/2+\sum_{k=h+1}^n e_{k\alpha}^{(o)}/2\right)}.
\end{align*}
Since $\sum_\alpha h_\alpha = h$, to prove (\ref{main_bound3}) it remains to show that
\begin{equation}\label{p_toshow}
\sum_\alpha\left(\sum_{ k=1}^h h_{k\alpha} +\sum_{k=h+1}^{n(\Gamma)} e_{k\alpha}^{(o)}\right) \ge 4a.
\end{equation}
For $k=1,...,h$, using (\ref{sum_h}) and (\ref{deg1}) we get that
$$\sum_{\alpha=1}^m h_{k\alpha} = 1 = \deg(b_k,P).$$
For $k=h+1,...,n$, using (\ref{off_e}) and (\ref{parity2}) we get that
$$\sum_{\alpha=1}^m e_{k\alpha}^{(o)} = \deg_o(b_k, Q)\ge \deg(b_k,P) .$$
With (\ref{deg1}), we then conclude (\ref{p_toshow}), which concludes the proof of Claim \ref{isoclaim1}.

\subsection{Proof of \eqref{eqiso2}}\label{sec hardest}
The proof of this sections adopts the arguments in \cite[Section 5]{XYY_circular}. We expand the left-hand side in (\ref{eqiso2}) as
\begin{equation}
\begin{split}
 \left| {  \sum\limits_{\mu \ne \nu} {v^*_{\left[ \mu \right]} {G_{\left[ \mu\nu \right]}} {{v}_{\left[ \nu \right]}}} } \right|^{2a}  =\sum\limits_{\Gamma } \sum\limits_{b_1,...,b_{n(\Gamma)}}^* \prod\limits_{k = 1}^{a} {v^*_{\left[ \Gamma(\mu_k) \right]}G_{\left[{\Gamma(\mu_k)}{\Gamma(\nu_k)}\right]}{v}_{\left[ \Gamma(\nu_k) \right]}}  \\
 \times \prod\limits_{k = a + 1}^{2a} \overline {v^*_{\left[ \Gamma(\mu_k) \right]}G_{\left[{\Gamma(\mu_k)}{\Gamma(\nu_k)}\right]}{v}_{\left[ \Gamma(\nu_k) \right]}} ,\label{sum_1}
 \end{split}
\end{equation}
where we again define partition of indices
$$\Gamma: \{\mu_1, \cdots, \mu_{2a}, \nu_1,\cdots, \nu_{2a}\} \to L(\Gamma), \quad L(\Gamma)=(b_1, \cdots, b_{n(\Gamma)}),$$
 $\Gamma^{-1}(b_k)$ are equivalence classes of the partition, $n(\Gamma)$ is the number of equivalence classes, $b_k$ are indices taking values in $\cal I_3 $, and $\sum^*$ denotes the summation subject to the condition that $b_1,\ldots, b_{n(\Gamma)}$ all take distinct values and $\Gamma(\mu_k) \ne \Gamma(\nu_k)$ for all $k$. Since the number of such partitions $\Gamma$ is finite and depends only on $a$, to prove \eqref{eqiso2} it suffices to show that for any fixed $\Gamma$,
\begin{equation}\label{temp1111}
\sum_{b_1,...,b_{n(\Gamma)}}^*   \mathbb E\Delta(b_1,...,b_{n(\Gamma)})\prec \Phi^{2a},
\end{equation}
where we abbreviated 
\begin{equation}\label{iso_defn_Delta}
\Delta(\Gamma) := {\prod\limits_{k = 1}^{a} {v^*_{\left[ \Gamma(i_k) \right]}G_{\left[{\Gamma(i_k)}{\Gamma(j_k)}\right]}{v}_{\left[ \Gamma(j_k) \right]}}  \cdot \prod\limits_{k = a + 1}^{2a} \overline {v^*_{\left[ \Gamma(i_k) \right]}G_{\left[{\Gamma(i_k)}{\Gamma(j_k)}\right]}{v}_{\left[ \Gamma(j_k) \right]}}}.
\end{equation}
For simplicity, we again omit the overline for complex conjugate in the following proof. In this way, we can avoid a lot of immaterial notational complexities that do not affect the proof.

For any $b_k\in L$, we can define a corresponding $\mathcal I_4$-valued variable $\overline b_k$ in the obvious way, and we denote
\begin{equation}\label{iso_defn_barL}
[L]:=\{b_1,...,b_{n},\overline b_1,...,\overline b_{n}\}.
\end{equation}
We shall abbreviate $G^{([J])}\equiv G^{[J]}$ for any index set $J\subset \cal I_3.$
Then we define $S$ groups as 
$$S_{[\mu\nu]}:=\begin{pmatrix}( X^T G^{[L]} X)_{\mu\nu}& ( X^T G^{[L]}Y)_{\mu \overline \nu} \\  (Y^TG^{[L]}  X)_{\overline \mu  \nu} &  (Y^T G^{[L]}  Y)_{\overline \mu\overline\nu} \end{pmatrix}.$$

We can define strings as in Definition \ref{def string} with bigger alphabet which includes the new $G$ and $S$ groups. Given any index set $J\subset L$, we shall say $G^{[J]}_{[\mu\nu]}$ (resp. $(G^{[J]}_{[\mu\mu]})^{-1}$) is maximally expanded if $J\cup \{\mu,\nu\}=L$ (resp. $J\cup\{\mu\}=L$). Also the $S$ groups are always maximally expanded. We shall call $G^{[J]}_{[\mu\nu]}$ and $S_{[\mu\nu]}$ off-diagonal symbols if $\mu\ne \nu$, and all the other symbols are diagonal. Note that by the local law \eqref{entry_law} and \eqref{simpleT}, we have that for off-diagonal group $\mathbf a_o$,   $\pi^{-1}\left\llbracket\mathbf a_o\right\rrbracket \pi^{-1}\prec \Phi . $
We use ${\cal F}_{\rm{off}}(\mathbf s)$ to denote the number of off-diagonal symbols in the string $\mathbf s$. 
We can define string operators as in Definition \ref{defn_stringoperator0} using the resolvent expansions in Lemma \ref{lemm_resolvent_group}, and we can perform the resolvent expansions systematically using string operators as in \eqref{resolvent_string2}. We omit the detailed definitions here. Instead, we directly give the following result, which has been proved in \cite{XYY_circular}.
 
\begin{lemma}[Lemma 5.9 of \cite{XYY_circular}]\label{isolem3} 
Consider the string $\mathbf s = ``\Delta(\Gamma)"$. Fix any $l_0\in \N$. There exists a constant $K(a,l_0)\in \N$ depending on $a$ and $l_0$ only such that the following property holds. For any binary sequence $w$ with $|w| = K(a,l_0)$ and $\mathbf s_w \ne \emptyset$, either ${\cal F}_{\rm{off}}(\mathbf s_{w})\ge l_0$ or $\mathbf s_{w}$ is maximally expanded. 
\end{lemma}

As in \eqref{0off_bound}, if we take $l_0 =\lceil  (h\omega^{-1} +2a) \rceil$, then 
$$ \sum_{|w|=K(a,l_0)}\llbracket\mathbf s_w\rrbracket \cdot \mathbf 1({\cal F}_{\rm{off}}(\mathbf s_{w})\ge l_0)\prec n^{h/2} \cdot n^{-h/2}\Phi^{2a} = \Phi^{2a},$$
where the $n^{h/2} $ comes from the vector $\bv$, since we have included $v_{\Gamma(\cdot)}$ into $\Delta(\Gamma)$. It remains to handle the maximally expanded strings. With \eqref{eq_res11} and \eqref{self_Gmu}, we can write
$$(G_{[\mu\mu]}^{[L\setminus\{\mu\}]})^{-1} = \pi^{-1}(z) - \left[S_{[\mu\mu]} - \begin{pmatrix} m_{1c} & 0 \\ 0 & m_{2c}\end{pmatrix}\right],$$
where by local law \eqref{entry_law} and \eqref{Zestimate1}, we have
$$\left\|S_{[\mu\mu]} - \begin{pmatrix} m_{1c} & 0 \\ 0 & m_{2c}\end{pmatrix}\right\| \prec \Phi.$$
Then we can Taylor expand $G_{[\mu\mu]}^{[L\setminus\{\mu\}]}$ in terms of $S_{[\mu\mu]}$, and replace the diagonal maximally expanded $G$ groups with diagonal $S$ groups. 
Finally, as in \eqref{maximal_bound2}, one can see that in order to prove \eqref{temp1111} it suffices to show that for any such monomial $M_w(\Delta(\Gamma))$ consisting of $S$ variables only, we have 
\begin{equation}\label{eq_iso_goal3} {\sum\limits_{\scriptstyle b_1 , \cdots, b_n(\Gamma)}^* \left|\mathbb E{\llbracket M_w( \Delta(\Gamma)) \rrbracket} \right|} \prec  {\Phi ^{2a}}.
\end{equation}

Now we decompose $S_{[\mu\nu]}$ in $M_w$ as
\begin{equation}\label{iso_decomp_1}
S_{\left[\mu\nu\right]}= S_{\mu \nu} \begin{pmatrix} 1 & 0\\
0 & 0
\end{pmatrix} +S_{\mu\overline \nu} \begin{pmatrix} 0 & 1  \\
0 & 0
\end{pmatrix}+S_{\overline \mu \nu}\begin{pmatrix} 0 & 0 \\ 1 & 0
\end{pmatrix}+S_{\overline \mu \overline \nu}\begin{pmatrix} 0& 0 \\ 0 & 1
\end{pmatrix},
\end{equation}
where we define the following symbols:
\begin{equation}\label{SR1}
\begin{split}
& S_{\mu\nu}:=\left(X^T G^{[L]} X\right)_{\mu\nu}  ,\quad S_{\mu\overline \nu}:=\left(X^T G^{[L]} Y\right)_{\mu\overline \nu}, \\ 
& S_{\overline \mu\nu}:=\left(Y^T G^{[L]} X\right)_{\overline \mu\nu} ,\quad S_{\overline \mu\overline \nu}:=\left(Y^T G^{[L]} Y\right)_{\overline \mu\overline \nu} .
\end{split}
\end{equation}
We expand the $S_{[ij]}$'s in $M_w(\Delta(\Gamma))$ using (\ref{iso_decomp_1}), and write $M_w(\Delta(\Gamma))$ as a sum of monomials 
\begin{equation}\label{eqn_monomial2}
M_w(\Delta(\Gamma))=\sum_{\gamma}P_w( \Gamma,\gamma) \cal M_{\gamma},
\end{equation}
where $\gamma$ is an index to label these monomials, $P_w( \Gamma,\gamma)$ denotes a scalar monomial in terms of $S_{\mu\nu}$ variables only, and $\cal M_{\gamma}$ contains the factor depending only on the entries of $\bv$. We make a simple observation that 
\begin{equation}\label{Foff2}
{\cal F}_{\rm{off}}\left(P_w( \Gamma)\right)={\cal F}_{\rm{off}}\left( M_w(\Delta(\Gamma))\right) \ge {\cal F}_{\rm{off}}\left(\Delta(\Gamma)\right)= 2a.
\end{equation}

As in the proof of \eqref{eqiso1}, we define $\deg(b_k, \Delta(\Gamma)):= |\Gamma^{-1}(b_k)|$ and $h:=\#\{ k : \deg(b_k, P) = 1\}$. We have already seen that the indices with degree 1 cause the main trouble. 
Since the number of summands in (\ref{eqn_monomial2}) is of order $\OO(1)$, to prove (\ref{eq_iso_goal3}) it suffices to prove that for any fixed monomial $P_w( \Gamma)$ in (\ref{eqn_monomial2}),
\begin{equation}\label{eq_iso_goal4}
  \left|\mathbb E{\llbracket P_w( \Gamma)\rrbracket} \right| \prec n^{-h/2}{\Phi ^{2a}}.
\end{equation}
To prove \eqref{eq_iso_goal4}, 
we need to keep track of the ``single" indices in $[L]$ during the expansion. In $P_w$, let $\deg_o(b_k, P_w)$ denotes the number of times that $b_k$ or $\overline b_k$ appears as an index of the \textit{off-diagonal} $S$ variables in $P_w$. Again we have the following simple relations: 
\begin{equation}\label{parity2222}
\deg_o(b_k, P_w) \ge \deg(b_k, \Delta(\Gamma)), \quad  \deg_o(b_k, P_w) - \deg(b_k, \Delta(\Gamma)) =0 \mod 2.
\end{equation}

We expand the $S$ variables in $P_w$ using \eqref{SR1}, and call the resulting string $Q_w$.  We now introduce graphs to conclude the proof of (\ref{eq_iso_goal4}). We use a connected graph to represent the string $Q_w$, call it by $\mathfrak G_{Q}$. The indices in $[L]$ are represented by black nodes in $\mathfrak G_{Q}$, and the $i,j$ summation indices in the $S$ variables are represented by white nodes. The $X$ or $Y$ variables are represented by wavy edges, and $G$ are represented by solid lines.  In Figure \ref{M_1}, we give an example of the graph $\mathfrak G_{Q}$. Note that in graph $\mathfrak G_{Q}$, the $G$ edges and $X,Y$ edges are mutually independent, since the $G$ variables are maximally expanded.

\color{black}
\begin{figure}[htb]
\centering
\begin{tikzpicture}
 \tikzset{dot/.style={circle,fill=#1,inner sep=3,minimum size=0.5pt}}
 \tikzset{wdot/.style={circle,draw,inner sep=3,minimum size=0.5pt}}
 \tikzset{zig/.style={decoration={
    zigzag,
    segment length=4,
    amplitude=.9,post=lineto,
    post length=2pt}}}
 \node (a2) [label=right:{$\overline b_1$}, dot] {};
 \node (c1) [left of=a2, xshift=-6mm, yshift = 6mm, wdot] {} edge [decorate,zig]  (a2);
 \node (c2) [left of=a2, xshift=-7mm, yshift = -4mm, wdot] {} edge[decorate,zig] (a2) edge (c1);
 \node (b1) [above of=a2, xshift=3mm, yshift = 8mm, wdot] {} edge[decorate,zig] (a2);
  \node (a3) [label=left:{$b_1$}, below of=a2, xshift=0mm, yshift = -15mm, dot] {}; 
  \node (a4) [label=below:{$\overline b_2$}, right of=a2, xshift = 20mm, dot] {} ; 
 \node (b2) [above of=a4, xshift=-3mm, yshift = 8mm, wdot] {} edge[decorate,zig] (a4) edge (b1);
 \node (a5) [label=left:{$b_2$}, right of=a3, xshift = 20mm, dot] {};
 \node (a6) [label=below:{$\overline b_3$}, right of=a4, xshift = 20mm, dot] {};
 \node (a7) [label=left:{$b_3$}, right of=a5, xshift = 20mm, dot] {};
 \node (b9) [above of=a4, xshift=3mm, yshift = 8mm, wdot] {} edge [decorate,zig] (a4);
 \node (b0) [above of=a6, xshift=-3mm, yshift = 8mm, wdot] {} edge [decorate,zig] (a6) edge  (b9);
 \node (b3) [right of=a6, xshift=8mm, yshift = -3mm, wdot] {} edge[decorate,zig] (a6);
 \node (b4) [right of=a7, xshift=8mm, yshift = 3mm, wdot] {} edge[decorate,zig] (a7) edge (b3);
 \node (b5) [below of=a3, xshift=3mm, yshift = -8mm, wdot] {} edge[decorate,zig] (a3);
 \node (b6) [below of=a5, xshift=-3mm, yshift = -8mm, wdot] {} edge[decorate,zig] (a5) edge (b5);
 \node (b7) [below of=a5, xshift=3mm, yshift = -8mm, wdot] {} edge[decorate,zig]  (a5);
 \node (b8) [below of=a7, xshift=-3mm, yshift = -8mm, wdot] {} edge[decorate,zig] (a7) edge (b7);
 \node (c3) [right of=a4, xshift=3mm, yshift=-4mm, wdot] {} edge[decorate,zig] (a4);
 \node (c4) [above of=a7, xshift=-5mm, yshift=3mm, wdot] {} edge[decorate,zig] (a7) edge (c3);
 \node (l) [draw=black,thick,rounded corners=2pt,below left=10mm, minimum width = 85pt, minimum height = 50pt, right of=a6,xshift = 40mm, yshift = -5mm]  {};
 \node (l1) [wdot, above of=l, xshift = -35pt, yshift = -13pt] {};\node (l2) [wdot, right of= l1, xshift = 5mm, label=right:{$G$}] {} edge (l1) ;
 \node (l3) [dot, below of = l1] {};\node (l4) [wdot, right of= l3, xshift = 5mm] {} edge [decorate, zig, label=right:{$X$, $Y$}] (l3) ;
\end{tikzpicture}
\caption{An example of the graph $\mathfrak G_{Q}$.}
\label{M_1}
\end{figure}
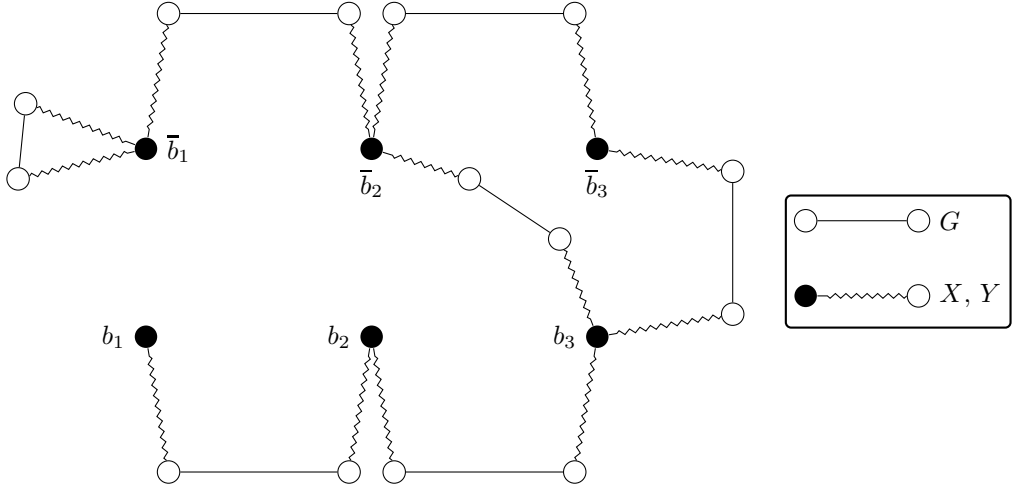

Notice that each white node represents a summation index. As we have done for the black nodes, we first partition the white nodes into blocks and then assign values to the blocks when doing the summation. Let $\wt\Gamma$ be a fixed partition and denote its blocks by $\wt w_1,...,\wt w_{n(\wt\Gamma)}$. If two white nodes of some {\bf off-diagonal} $S$ variable happen to lie in the same block, then we merge the two nodes into one white node and call the resulting graph $\mathfrak G_{Q1}$. Note that we do not merge the white nodes for diagonal $S$ variables. 
Let $n_d^{(o)}$ be the number of diagonal $G$ edges in the {\it off-diagonal} $S$ variables. 
We trivially have 
\begin{equation}\label{number_white_node}
\text{\# of white nodes} = -n_d^{(o)}+\sum_{k = 1}^{n} \left[{\deg \left( {{b_k}} \right) + \deg ({{\overline b}_k})} \right],
\end{equation}
where the degrees of the nodes are defined in the usual graphical sense.

We define the subset of single indices
\begin{equation}\label{Vnode}
\mathcal V := \{b_k \in L|\ \deg(b_k ,\Delta(\Omega)) = 1\},\quad |\mathcal V |=h.
\end{equation}
By (\ref{parity2222}), there are at least $h$ black nodes with odd $\deg$ in $[\mathcal V]$. WLOG, we may assume these nodes are $b_1,...,b_{h}$. To have nonzero expectation, for each $k=1,...,h$, there must exist a block $\wt w_{i_k}$ connecting to $b_k$ which contains at least $3$ white nodes. Then we denote by $A(b_k)\subseteq \wt w_{i_k}$ the set of the adjacent white nodes to $b_k$ in $\wt w_{i_k}$ (the block $\wt w_{i_k}$ may also contain white nodes that do not connect to $b_k$, hence in general $A(b_k)$ may not be equal to $\wt w_{i_k}$).  We call a white node that is not connected to loops of solid $G$ edges a {\bf normal white node}, i.e. normal white nodes are white nodes that have not been merged before. The other white nodes are called merged white nodes. Then we define
\[\mathcal V_{0} := \left\{b_k|\ A(b_k)\text{ has no normal white nodes,} \ 1\le k\le h\right\},\]
and
\[\mathcal V_{1} := \left\{b_k|\ A(b_k)\text{ has at least one normal white node,}\ 1\le k\le h \right\}.\]
The following lemma gives the key estimates we need.

\begin{lemma}
For any partition of white nodes $\wt\Gamma$,
\begin{equation}\label{claim3_1}
2n(\wt\Gamma) \le -|\mathcal V_{1}|-|\mathcal V_{0}|/2- n_d^{(o)}+\sum_{k = 1}^{n(\wt\Gamma)} \left[{\deg \left( {{b_k}} \right) + \deg ({{\overline b}_k})} \right] ,
\end{equation}
and
\begin{equation}\label{claim3_2}
 n_o \ge 2a +  |\mathcal V_{0}| , 
\end{equation}
where $n_{o}$ is the total number of off-diagonal $S$ variables in $Q_w$. 
\end{lemma}
\begin{proof}
WLOG, let $\wt w_{1},...,\wt w_{d}$ be the distinct blocks among the blocks $\wt w_{i_k}$, $k=1,\cdots, h$. 
A merged white node is connected to two black nodes and a normal white node is connected to one black node. Hence a merged white node belongs to two sets $A(b_{k_1}), A(b_{k_2})$, and a normal white node belongs to exactly one set $A(b_k)$. Therefore for each $i=1,...,d$, if $ \wt w_{i}$ contains exactly one $A(b_k)$, then
 \[\left| \wt w_{i}\right|\ge 3\ge 2+\mathbf 1_{\mathcal V_{1}}(b_k)+\frac{\mathbf 1_{\mathcal V_{0}}(b_k)}{2}.\]
If $  \wt w_{i}$ contains at least two $A(b_k)$, then
 \be\label{addcount1}
 \begin{split}
 \left| \wt w_{i}\right| &\ge \sum\limits_{b_k: A(b_k)\subseteq \wt w_{i}}\left(2 \cdot \mathbf 1_{\mathcal V_{1}}(b_k)+\frac{3}{2}\cdot\mathbf 1_{\mathcal V_{0}}(b_k)\right)\\
 &\ge 2+\sum\limits_{b_k: A(b_k)\subseteq W_{i}}\left(\mathbf 1_{\mathcal V_{1}}(b_k)+\frac{\mathbf 1_{\mathcal V_{0}}(b_k)}{2}\right).
 \end{split}\ee
Here the first inequality holds due to the following reasoning. For each black node $b_k$ with $A(b_k)\subseteq  \wt w_i$, we count the number of white nodes in $A(b_k)$ and add them together. During the counting, we assign weight 1 to a normal white node and weight $1/2$ to a merged white node (since it is shared by two different black nodes). If $b_k\in \mathcal V_{0}$, there are at least three merged white nodes in $A(b_k)$ with total weight $\ge 3/2$. If $b_k\in \mathcal V_{1}$, there are at least one normal white node and two other white nodes in $A(b_k)$ with total weight $\ge 2$. 

Then summing \eqref{addcount1} over $i$, we get that
\[\sum\limits_{i=1}^{d}\left|\wt w_{i}\right|\ge 2d+|\mathcal V_{1}|+\frac{|\mathcal V_{0}|}{2}.\]
For the other $n(\wt\Gamma)-d$ blocks, each of them contains at least two white nodes, so we get that
\[2n(\wt\Gamma)+|\mathcal V_{1}|+ \frac{|\mathcal V_{0}|}{2} \le\sum\limits_{i=1}^{d}\left|\wt w_{i}\right|+2(n(\wt\Gamma)-d) \le -n_d^{(o)}+\sum_{k = 1}^{n(\wt\Gamma)} \left[{\deg \left( {{b_k}} \right) + \deg ({{\overline b}_k})} \right],\]
where we used (\ref{number_white_node}) in the last step. Rearranging terms gives (\ref{claim3_1}).

For $b_k \in \mathcal V_{0}$, $A(b_k)$ contains at least three white nodes from off-diagonal $R$-groups. Hence we have $\deg(b_k,Q_w)\ge 3$, compared with $\deg(b_k,\Delta(\Gamma)) =1 $. This means that we have applied \eqref{eq_res3} or \eqref{eq_res4} with respect to $\mu=b_k$ at least once and picked the second terms at some step of the expansions (which corresponds to the $\tau_1$ operation in Definition \ref{defn_stringoperator0}). Each such operation increases the off-diagonal $S$ variables at least by 1, which gives \eqref{claim3_2}.
\end{proof}

Now we prove (\ref{eq_iso_goal4}). By \eqref{moment_n3} and the discussion below \eqref{Vnode} for the single indices $b_1, \cdots, b_h$, we get
\begin{align*}
 \left|\mathbb E \llbracket P_w \rrbracket\right|=  \left|\mathbb E \llbracket Q_w \rrbracket\right| & \prec \sum\limits_{\wt\Gamma } {\sum\limits_{\wt w_1, \ldots ,\wt w_{n(\wt\Gamma)}}^* \Phi^{n_o - n_d^{(o)}} n^{-n(\wt\Gamma) - h/2} \phi_n^{\sum_{k=1}^{n(\wt\Gamma)}\left[\deg \left( {{b_k}} \right) + \deg \left( {{{\overline b}_k}} \right)\right] - 2n(\wt\Gamma) - h}  } \nonumber\\
& \prec n^{-h/2} \Phi^{\sum_{k=1}^{n(\wt\Gamma)}\left[\deg \left( {{b_k}} \right) + \deg \left( {{{\overline b}_k}} \right)\right] - 2n(\wt\Gamma) - h  +n_o - n_d^{(o)}}  \\
& \prec n^{-h/2} \Phi^{  - |\mathcal V_{0}|/2 +n_o } \prec n^{-h/2} \Phi^{ 2a} ,
\end{align*}
where in the third step we used (\ref{claim3_1}) and $|\mathcal V_0|+|\mathcal V_1|=h$, and last step (\ref{claim3_2}). Thus we have proved (\ref{eq_iso_goal4}), which concludes the proof of \eqref{eqiso2}.




\end{document}